\documentclass[reqno,a4paper]{amsart}
\usepackage{amsmath,amsthm}
\usepackage{verbatim}
\usepackage[usenames]{color}
\usepackage[usenames]{color}
\usepackage{a4wide} % Nice margins for A4 paper
\usepackage{amssymb} % Needed for \Subset
\usepackage{nicefrac} % Needed for \nicefrac
\usepackage{aliascnt} % To use single numbering + \autoref
\usepackage[pagebackref,colorlinks]{hyperref}
\usepackage{marvosym} % dingbats
\usepackage{graphicx}

\def\rnew{\color{red}}

\newtheorem{lemma}{Lemma}
\newtheorem{theorem}{Theorem}
\newtheorem{remark}{Remark}
\newtheorem{corollary}{Corollary}
\newtheorem{example}{Example}
\newtheorem{proposition}{Proposition}

\newtheorem{definition}{Definition}

\newcommand\eref[1]{(\ref{#1})}

\newcommand{\Mdual}[1]{\ensuremath{\left\langle #1 \right\rangle_{1/\maxw}}}
\renewcommand{\Mdual}[1]{\ensuremath{\left\langle #1 \right\rangle}}
\newcommand{\dup}[1]{\ensuremath{\left\langle #1 \right\rangle}}

\newcommand{\ulambda}{\underline{\lambda}}
\newcommand{\rhs}{{\rm rhs}}

\newcommand{\beqn}{\begin{equation}}
\newcommand{\eeqn}{\end{equation}}

\newcommand{\be}{\begin{equation}}
\newcommand{\ee}{\end{equation}}

\newcommand{\mI}{\mathfrak{I}}

\newcommand{\ve}{\varepsilon}
\newcommand{\cT}{\mathcal{T}}

\def\lsim{\raisebox{-1ex}{$~\stackrel{\textstyle<}{\sim}~$}}

\newcommand{\N}{\mathbb{N}}
\newcommand{\R}{\mathbb{R}}

\newcommand{\BIGOP}[1]{\mathop{\mathchoice%
{\raise-0.22em\hbox{\huge $#1$}}%
{\raise-0.05em\hbox{\Large $#1$}}{\hbox{\large $#1$}}{#1}}}
\newcommand{\bigtimes}{\BIGOP{\times}}
\newcommand{\BIGboxplus}{\mathop{\mathchoice%
{\raise-0.35em\hbox{\huge $\boxplus$}}%
{\raise-0.15em\hbox{\Large $\boxplus$}}{\hbox{\large $\boxplus$}}{\boxplus}}}

\newcommand{\dd}{\,{\mathrm d}} % Differential
\newcommand{\abs}[2][y]{\if#1y\left\fi\lvert#2\if#1y\right\fi\rvert} % Absolute value
\newcommand{\norm}[2][y]{\if#1y\left\fi\lVert#2\if#1y\right\fi\rVert} % Norm

\newcommand{\ua}{\underline{a}}

\newcommand{\cB}{\mathfrak{B}}

\newcommand{\tripnorm}[1]{|\!|\!| #1|\!|\!|}

\newcommand{\HH}{{\mathrm H}}
\newcommand{\LL}{{\mathrm L}}

\newcommand{\cE}{\mathcal{E}}

\newcommand{\bc}{\mathbf{c}}

\newcommand{\bm}{\mathbf{m}}

\newcommand {\vphi}{{\varphi}}

\newcommand{\cA}{\mathcal{A}}
\newcommand{\cR}{\mathcal{R}}

\newcommand{\C}{\mathbb{C}}
\newcommand{\cost}{{\rm cost}}
\newcommand{\rr}{\rangle}
\newcommand{\uu}{u}

\newcommand{\cI}{\mathcal{I}}
\newcommand{\rd}[1]{\textcolor{black}{#1}}

\newcommand{\lgr}[1]{\textcolor{red}{#1}}

\numberwithin{equation}{section}

\title[Tensor Compressibility  of Solutions to High-Dimensional Elliptic PDEs]{Tensor-Sparsity of Solutions to
High-Dimensional\\ Elliptic Partial Differential Equations}\thanks{This work has been supported in
part by the DFG Special Priority Program SPP-1324, by the DFG SFB-Transregio 40, by the DFG Research Group 1779,
the Excellence Initiative of the German Federal and State Governments,
 (RWTH Aachen  Distinguished Professorship,
Graduate School AICES),
 and NSF grant DMS 1222390.  The second author's  research was supported by the Office of Naval Research Contracts   ONR
 N00014-09-1-0107, ONR N00014-11-1-0712, ONR N00014-12-1-0561;  and by  the NSF Grants   DMS 0915231, DMS 1222715.  This research was initiated when he was an AICES Visiting Professor at RWTH Aachen.}
\author{Wolfgang Dahmen, Ronald DeVore, Lars Grasedyck, and Endre S\"uli}
\address{\!\!Institut f\"ur Geometrie und Praktische Mathematik\\ RWTH Aachen\\ Templergraben 55\\
52056 Aachen\\ Germany\\ {\tt dahmen@igpm.rwth-aachen.de} \& {\tt lgr@igpm.rwth-aachen.de}}
\address{\!\!Department of Mathematics, Texas A \& M University, College Station, Texas 77840, USA\\
{\tt  ronald.a.devore@gmail.com}}
\address{\!\!Mathematical Institute, University of Oxford, Oxford OX2 6GG, UK\\ {\tt endre.suli@maths.ox.ac.uk}}
\date{\today}

\begin{document}

\begin{abstract}
A recurring theme in attempts to break the curse of dimensionality in the numerical approximations of solutions to
high-dimensional partial differential equations (PDEs) is to employ some form of  sparse tensor approximation. Unfortunately, there are only a few results that quantify the possible advantages of such an approach. This paper
introduces a class $\Sigma_n$ of functions, which can be written as a sum of rank-one tensors using a total of at most $n$
parameters and then uses this notion of sparsity to prove a regularity theorem for certain high-dimensional elliptic PDEs.   It is shown, among other results,   that whenever the right-hand side $f$ of the elliptic
PDE can be approximated with a certain rate $\mathcal{O}(n^{-r})$ in the norm of $\HH^{-1}$ by  elements of $\Sigma_n$,
then the solution $u$ can be approximated in $\HH^1$ from $\Sigma_n$ to accuracy $\mathcal{O}(n^{-r'})$ for any $r'\in (0,r)$.
Since these results require knowledge of the eigenbasis of the elliptic operator considered, we propose
 a second ``basis-free''
model of tensor sparsity  and prove a  regularity theorem for this second sparsity model as well.  We then proceed  to address the important question
of the extent such regularity theorems translate into results on computational complexity.
It is  shown how this second model can be used to derive computational algorithms with performance that breaks the curse of dimensionality on certain model high-dimensional elliptic PDEs with tensor-sparse data.
\end{abstract}

\maketitle



\section{Introduction}\label{sec:introduction}
%%%%%%%%%%%%%%%%%%%%%%%%%%%%%%%%%%%%%%%%%%%

Many important problems that arise in applications involve a large number of spatial variables.   These
high-dimensional problems    pose a serious computational challenge because of the so-called   {\em curse of dimensionality}; roughly speaking,
this means that when using classical methods of approximation or numerical methods based on classical approximations,  the computational work required to approximate or to recover a function of $d$ variables with a desired target accuracy typically scales {\em exponentially} in $d$. This has led to so-called {\em intractability} results, which say that even under the assumption that the target function has very high order
of classical regularity (in terms of various notions of derivatives), the exponential effect of the spatial dimension $d$ prevails, see \lgr{\cite{NW09}}.
Subsequent attempts to overcome the curse of dimensionality have been mainly based on exploring the effect of very strong regularity assumptions,
or constraining the dependence of the functions on some of the variables, {\cite{NW08}}. It is not clear though, for which problems of practical interest such strong assumptions are actually satisfied.

The tacit assumption behind these negative tractability results is
that classical notions of smoothness  are used to characterize the regularity of the solution and classical methods of approximation are  used in the development of the numerical algorithms.  In low spatial dimensions,  smoothness  is typically
exploited by using classical approximation methods,  such as splines or finite elements,  based on {\em localization}. This is often  enhanced by {\em adaptation concepts}, which exploit
weaker smoothness measures in the sense of Besov spaces and thus provide somewhat better convergence rates. However, the larger the
spatial dimension the smaller the difference between the various smoothness notions \rm{becomes}.  Adaptivity based on localization is therefore not a decisive remedy, and alternative strategies are called for.

It has recently been recognized that the situation is not as bleak as has been described above.  Indeed, solutions to real-world high-dimensional problems are thought to have a structure different from  high-dimensional regularity, that renders them more
amenable to numerical approximation.  The challenge is to explicitly define these new structures, for a given class of problems, and then to build numerical methods that exploit them.  This has led to various notions, such as sparsity, variable reduction, and reduced modelling.
In contrast with the low-dimensional regime, essentially governed by smoothness, there is of course no universal recipe for discovering the correct notion of  sparsity:
the correct  structural sparsity will depend on the problem at hand.

In the context of numerical approximation of PDEs, there are roughly two groups of high-dimensional problems.
The first group involves {\em parameter-dependent} families of partial differential equations,
where  the number of ``differential'' variables is still small but  the data and the coefficients in the
PDE depend on possibly many additional parameters, which often are to be optimized in a design or optimization context. Hence the solution becomes a function of the spatial variables and of the additional parameters.
In particular, such parameter-dependent PDEs
are obtained when coefficients in the PDE are random fields. Expanding such a random field
may even lead to an {\em infinite} number of deterministic parameters, see e.g.  \cite{CDS1,CDS2}.  Reduced-order modeling concepts
such as POD (proper orthogonal decomposition) or the reduced-basis method aim at
  constructing
solution-dependent dictionaries comprised
of judiciously chosen ``snapshots'' from the solution manifold. This can actually be viewed as a {\em separation ansatz} between the spatial variables and the  parameters.

The second group of problems, and those of interest to us here, concern
partial differential equations posed in a phase space of large spatial dimension    (e.g. Schr\"odinger,  
Ornstein--Uhlenbeck,  Smoluchowski, and Fokker--Planck equations).
As the underlying domain $D$ is typically a Cartesian product $D = \times_{j=1}^d D_j$ of low-dimensional domains,  $D_j$, $j=1,\dots, d$, it is
natural to seek (approximate) representations in terms of {\em separable functions - viz.   low-rank tensors.}  {Kolmogorov equations and related PDEs in (countably) infinitely many space dimensions, which arise as evolution equations for the probability density function for stochastic PDEs, require a different functional-analytic setting from the one considered here and are therefore not treated in the present paper; for further details in this direction the reader is referred to \cite{SchS}.}

{\em The main purpose of the present paper is to propose specific notions of sparsity, based on tensor decompositions, and then to show that the solutions of certain high-dimensional diffusion equations inherit this type of sparsity from given data. This is then shown to lead  indeed to tractability of solving such  high-dimensional PDEs.}

 To motivate the results that follow we sketch a simple  example, albeit not in a PDE context yet,
 which indicates  how tensor structure can   mitigate the curse
of dimensionality.  Suppose that $f\in C^{s}([0,1]^d)$  for some $s\in\N$. Approximating
$\int_{[0,1]^d}f(x)\,{\rm d}x$ by  a standard tensor-product quadrature method $I^d_{s,n}(f)$
 of order $s$ on a Cartesian grid with meshsize $n^{-1}$ in each of the $d$ co-ordinate directions
(e.g. with an $s$th order accurate composite Newton--Cotes {or Gau{\ss}} quadrature possessing nonnegative weights) yields accuracy
in the sense that
$$
\Big| \int_{[0,1]^d}f(x)\,{\rm d}x -I^d_{s,n}(f)\Big| \leq {Cd}\, n^{-s}\|f\|_{C^s([0,1]^d)},
$$
at the expense of  the order of $N= (sn)^d$ operations. Here ${C}$ is a fixed constant  depending on the univariate    quadrature rule.
If, in addition, one knows that $f$ is a product of  {given} univariate functions: $f(x)=f_1(x_1)\cdots f_d(x_d)$, i.e., $f$ is a {\em rank-one tensor or separable function}, then
  $  \int_{[0,1]^d}f(x)\,{\rm d}x = \prod_{j=1}^d\Big( \int_0^1 f_j(x_j)\,{\rm d}x_j\Big)$, and  one obtains
$$
\Big| \int_{[0,1]^d}f(x)\,{\rm d}x - \prod_{j=1}^d I_{s,n }^1(f_j)\Big| \leq {Cd}\, n^{-s}\|f\|_{C^s([0,1]^d)}
$$
at the expense of the order of $dsn$ operations  only, where again $C$ is a fixed constant depending on the univariate quadrature rule.

 Thus, in the first case,   accuracy $\ve$ costs
 \beqn
 \label{N1}
N_1(\ve) = \mathcal{O}(s^d {d^{d/s}} \ve^{-d/s})
 \eeqn
operations so that the curse is apparent in the factor $d$ appearing in the power of $\varepsilon$. One could do much better
by employing Smolyak quadrature rules, which, however, requires  the assumption of  significantly higher regularity of
the integrand,  \cite{NW08}, in order to obtain specific rates. In the second case, accuracy $\ve$ is achieved at the
expense of  the order of {$d^{1+1/s}s\, \ve^{-1/s}$} operations.  Thus, one has the full one-dimensional gain
of smoothness and the spatial dimension enters only linearly as opposed to exponentially as in the first case.
Of course, assuming that $f$ is a rank-one tensor is rather restrictive. If however one knew that it takes only
 $r(\ve)$ rank-one summands to approximate $f$ to within accuracy $\ve$, the computational cost would  still be  of the order of
\beqn
\label{N2}
N_2(\ve)= r(\ve)\,d^{1+1/s} s\, \ve^{-1/s}
\eeqn
operations.
This is preferable to $N_1(\ve) = \mathcal{O}(s^d \ve^{-d/s})$, even when $r(\ve)$ grows like $\ve^{-\alpha}$ for some fixed positive $\alpha$.

The main purpose of this paper is to show that the type of saving exhibited in the above example of numerical integration is present in the numerical approximation of certain high-dimensional elliptic equations described  in \S \ref{sect:gen-setting}.
 To expect such savings,  the elliptic operator under consideration should exhibit a ``tensor-friendly'' structure.
A differential operator that is a tensor-product of low-dimensional elliptic operators would trivially be ``tensor-friendly'', but it would lead us into classes of hypo-elliptic problems. Here, instead, we consider  elliptic operators, which are {\em sums} of tensor-product
operators.  The high-dimensional Laplacian is a prototypical example. We work in
a somewhat more general setting than that of a simple Laplace operator for the following reasons. The original motivation for this work was a class
of Fokker--Planck equations, whose numerical treatment, after suitable operator splitting steps, reduces to solving
a high-dimensional symmetric elliptic problem over a Cartesian product domain, where the energy space of the elliptic operator is a weighted Sobolev space with a strongly degenerate weight. The setting we shall describe in
\S \ref{sect:gen-setting} covers such situations as well, the aim of considering a general class of elliptic problems being to extract and highlight the relevant structural assumptions.

In \S\ref{sec:ts},  we turn to proposing notions of tensor-sparsity that we will prove are relevant for approximating the solutions
to operator equations of the type discussed in the above paragraph and formally introduced in \S \ref{sect:gen-setting}.
To briefly describe in this introduction the form such sparsity takes, consider a
function space $X=X_1(D_1)\otimes\cdots \otimes X_d(D_d)$ over a Cartesian product domain $D=\times_{j=1}^d D_j$. Suppose, only for the sake of simplifying the present discussion, that the component domains $D_j$, $j=1,\dots, d$, are intervals.
  A standard way of approximating the elements of $X$  is to start with $d$  (a priori chosen) univariate bases $\Psi_j=\{\psi^j_\nu:\nu\in \Lambda_j\}$
for the spaces $X_j$, $j=1,\ldots,d$, where $\Psi_1\otimes \cdots \otimes \Psi_d$ is dense in $X$.
Examples of such bases could be trigonometric functions, polynomials or wavelets.
Hence, the   product basis $\Psi = \Psi_1\otimes\cdots\otimes \Psi_d$ allows us to expand $v\in X$ as
\beqn
\label{expf0}
v= \sum_{\nu\in \times_{j=1}^d\Lambda_j} v_{\nu_1\ldots \nu_d}\psi^1_{\nu_1}\otimes\cdots\otimes \psi^d_{\nu_d}.
\eeqn
We use, here and throughout this paper, the standard multi-index notation
\be
\label{mvindices} \nu=(\nu_1,\dots,\nu_d), \quad {\rm and}\quad \nu\le \mu \ {\rm means~that}\ \nu_j\le \mu_j \ {\rm for}\ j=1,\dots,d.
\ee
Once we have decided to use $\Psi$ as an expansion system,  the standard approach to  obtaining a possibly {\em sparse}
approximation is to retain as few terms of the expansion \eqref{expf0} as possible, while still meeting a given accuracy tolerance.  This is a nonlinear selection process known as $N$-term approximation, see \cite{DNL} for a general treatment of $N$-term approximation and   \cite{DSS} for a proposed implementation {and analysis in the context of high-dimensional elliptic} PDEs. However, using such {\em universal} bases $\Psi$, which are
independent of the specific $v$, best $N$-term approximation procedures {significantly mitigate but do not quite avoid  the curse of dimensionality. In fact, while in contrast to conventional {\em  isotropic} multiresolution approximations, under much stronger (mixed)  smoothness
assumptions, the factor $\ve^{-d/s}$ in \eref{N1} can be replaced by $\ve^{-1/s}$, the constants in the corresponding error bounds still exhibit an exponential growth in $d$.}
 
{Thus, it is not clear how useful $N$-term approximation with respect to a fixed background tensor product basis
will be for treating truly high-dimensional problems. }
 This is to be contrasted by allowing $v$ to be expanded in terms of separable functions where the factors are {now} allowed to
{\em depend} on $v$
\beqn
\label{expff}
v= \sum_{k=1}^\infty v_{k,1}\otimes \cdots\otimes v_{k,d},\quad v_{k,j}=v_{k,j}(v).
\eeqn
{Of particular interest is then the case where}, in spite of a possibly moderate smoothness of $v$, the terms in this expansion decay so rapidly that only a few of them
suffice in order to meet the target accuracy. {Thus we are asking for an alternative structural property of a function of many variables that
 leads to  computational tractability despite the lack of high regularity. This will lead us below to proposing new notions of {\em tensor sparsity}.}

Of course, {to that end}, aside from the question for which $v$ such an expansion converges rapidly
and how to identify the summands, the ultimate computational cost depends also on how well, i.e., at what cost, can the factors $v_{k,j}$ be
approximated, e.g. in terms of the \textit{universal} univariate bases $\Psi_j$, $j=1,\ldots,d$. Thus, two approximation processes have to be intertwined, which, in the present setting, takes the following form:
\beqn
\label{expfappr}
v \approx v_N := \sum_{k=1}^r \Big(\sum_{\nu\in \Gamma_{k,1}}c^1_{k,\nu}\psi^1_{\nu}\Big)\otimes\cdots\otimes
\Big(\sum_{\nu\in \Gamma_{k,d}}c^d_{k,\nu}\psi^d_{\nu}\Big),
\eeqn
and ideally, one would like to find $r, \Gamma_{k,j}, c^j_{k,\nu}, \nu\in \Gamma_{k,j}$, $k = 1,\ldots,r$,
$j=1,\ldots,d$, subject to
\beqn
\label{constraints}
\sum_{k=1}^r\sum_{j=1}^d \#(\Gamma_{k,j})\leq N,
\eeqn
so that $\|v -v_N\|_X$ is (near-)minimized; see \cite{B,BD}, where   algorithms are proposed that nearly minimize
this error. This is obviously a much more demanding (and much more nonlinear) optimization
task than activating the best coefficients in \eref{expf0}. In fact, it is not clear that a best approximation
in this sense exists. However, if $v$ admits
such an expansion, where {$r=r(\ve)$} increases slowly with decreasing $\ve$ while the $\#(\Gamma_{k,j})$,
$k=1,\dots,r$, $j=1,\dots,d$,
scale
like typical low-dimensional processes for moderate regularity, it is clear that the number of coefficients needed in
\eref{expfappr} would exhibit essentially the same dependence on $\ve$ as $N_2(\ve)$ in the integration example
discussed above, and hence would be much smaller than the number of coefficients needed in \eref{expf0}, for the
same accuracy.

Another way to look at this comparison is to note that when expanding the products of sums in \eref{expfappr} one obtains
a representation of the form \eref{expf0}. However, the coefficients in the   tensor array $(v_\nu)$, are
strongly dependent,  as is exhibited by the fact that they can be written as a sum of a few ($r$ in number) rank-one
tensors.   In this sense the tensor $(v_\nu)$ is {\em information sparse}, i.e., loosely speaking,
its entries $v_\nu$ depend ``on much fewer parameters''.
 
We are now ready to formulate more precisely the central question in this paper: suppose we are given an elliptic problem
over a high-dimensional product domain and suppose that the data  (the right-hand side function in the partial differential equation)
is {\em tensor-sparse} in the sense that it  can be approximated  by terms of the form \eref{expfappr}
at a certain rate; then, is the solution  $u$ also tensor-sparse, and, if so, at which rate?
In other words: in which
circumstances does a highly nonlinear process offer significantly higher sparsity than \eref{expf0}, and does this break the curse of dimensionality?

In \S\ref{sec:ts}, we shall formalize the above ideas and define sparse tensor structures and their spaces $\Sigma_n$, whose elements depend on at most $n$ parameters.  Sparse spaces of this type  {can, in principle, be defined relative to any {\em universal background} tensor basis.
The most convenient one for us in what follows for the purpose of highlighting the essential mechanisms is however a certain eigenbasis for the elliptic operator under consideration. In particular, this allows us to employ exponential sum approximations to
 {reciprocals} of eigenvalues,
which turns out to be critical for estimating the effect of  {the inversion of an elliptic operator}
on the growth of ranks in  {the approximation of solutions}.}
We then formulate approximation classes for this form of approximation.   An important issue for computational considerations are the
spaces $\Sigma_n(\cR)$, introduced in that section, that impose restrictions on the positions of the tensor basis that
are allowed to participate.
In \S \ref{sec:main}   
  {we present our first main contributions}, which are regularity results for elliptic PDEs stated in terms of the tensor-sparsity of \S\ref{sec:ts}.

 The above approach is primarily of theoretical interest, as a regularity result.
From a numerical perspective its practical relevance is limited by the fact that only in rare cases is an eigenbasis available.
Therefore, we propose in \S \ref{s:numerical} an alternative model for tensor-sparsity, which does not make
use of {\em any} background basis. However, now sparsity is expressed by approximability in terms of short
sums of rank-one functions, which are no longer described by a finite number of parameters but are constrained
by a certain {\em excess regularity}, which, in turn, will render them computable.
Hitherto, {\em regularity} has been described by the number of such constrained  terms needed to approximate the solution to within a given
target accuracy. Thanks to the excess regularity, these constrained rank-one terms can again be approximated
to within the same target accuracy by a finite number of parameters. We then proceed to quantify this last principal fact.
Again exponential sums serve as a key vehicle. This time however, instead of using them for the construction of
separable approximations to  reciprocals of eigenvalues, they are used to approximate inverses of operators. In contrast with previous applications of exponential sums in the literature \cite{Gr2004,Kh2008} for the approximation of
inverses of discrete (finite-dimensional) linear operators, it is now crucial to take into account the mapping properties of the operators concerned in terms of Sobolev scales.

{In \S \ref{sec:complexity}}  and \S \ref{sec:discr} we turn to the question of constructing concrete numerical algorithms which  exhibit these gains in numerical efficiency.   In certain settings, discussed in that section, we  give concrete bounds on the {\em representation and numerical complexity} of further approximating the regularity-controlled rank-one sums by analogous finitely parametrized expressions. By representation-complexity we mean the total number of parameters needed to eventually represent a tensor expansion to within the target accuracy.
Our main tool, in that section,  is the Dunford integral representation of the exponential of a sectorial operator. Again, in contrast with previous such uses (e.g. \cite{GaHaKh2005,GaHaKh2005b}),
it is crucial to apply these representations on the continuous (infinite-dimensional) level \cite{DJ}.  It is shown that the spatial dimension $d$ enters the (total) complexity only in a  {moderate} superlinear fashion, and thereby the curse of dimensionality is avoided.
The results presented here differ in essential ways from those in \cite{WW} since the required information on the data is
highly nonlinear and the differential operators cannot be diagonalized on the data classes under consideration. Thus, the main point
is the complexity of a structure preserving approximate  inversion of the operator.

We close the paper with some comments on open problems in \S \ref{sect:loss}.

%%%%%%%%%%%%%%%%%%%%%%%%%%%%%
\section{The General Setting}\label{sect:gen-setting}
%%%%%%%%%%%%%%%%%%%%%%%%%%%%%%%%%

%%%%%%%%%%%%%%%%%%%
\subsection{A class of elliptic problems}\label{ssec:prob}
%%%%%%%%%%%%%%%%%%%%%%%%%%%%
In this section, we formulate a class of high-dimensional problems for which we will prove that
their solutions can be effectively captured by sparse tensor approximations for suitable right-hand sides $f$.

Let $D=\bigtimes_{j=1}^d D_j$ be a Cartesian product domain in $\R^{dp}$ where the $D_j\subset \R^p$ are
low-dimensional domain factors. Typically,    $p$ takes the value 1, 2, or 3. So, the high-dimensionality occurs
because $d$ is large.   For each $j=1,\ldots,d$, we denote by $H_j$  a  separable Hilbert space,
with norm $\|\cdot\|_{H_j}$, comprised of functions defined on $D_j$.  We assume that each  $H_j$  is continuously and
densely embedded
in   $L_2(D_j)=L_2(D_j,\mu_j)$ with $\mu_j$ a positive Borel measure on the factor domain $D_j$ that is absolutely continuous
with respect to the Lebesgue measure.
Thus, the measures $\mu_j$, $j=1,\ldots,d$, are not necessarily the Lebesque measure
but could involve weights that are positive a.e. on $D_j$, $j=1,\ldots, d$. The dual pairing $\dup{\cdot,\cdot}$ is always understood to be induced by the inner product
for $L_2(D_j)$.
It will be clear from the context whether $\dup{\cdot,\cdot}$ denotes a dual pairing or the $L_2$ inner product.
Thus we have the Gel'fand triple (rigged Hilbert space)
$$
H_j \subset L_2(D_j) \subset (H_j)',\quad j=1,\ldots, d,
$$
with continuous and dense embeddings, where $(H_j)'$ denotes the normed dual of $H_j$.
We can think of the $H_j$
to stand for a (possibly weighted) Sobolev space, and possibly incorporating boundary conditions.

We assume that we are given ``low-dimensional'' symmetric $H_j$-elliptic operators $\cB_j : H_j \to (H_j)'$, i.e., the
bilinear forms $\mathfrak{b}_j(v,w):= \dup{\cB_j v,w}$ satisfy, for some fixed positive constants $c,\alpha$,
the following inequalities:
\beqn
\label{Di-ell}
|\mathfrak{b}_j(v,w)|\leq c \|v\|_{H_j}\|w\|_{H_j},\quad \mathfrak{b}_j(v,v)\geq \alpha \|v\|_{H_j}^2,\qquad v, w \in H_j,\; j=1,\ldots, d,
\eeqn
and $\mathfrak{b}_j(v,w) = \mathfrak{b}_j(w,v)$ for all $v, w \in H_j$ and all $j=1,\ldots, d$.

We next introduce a Hilbert space $\HH$ over $D$ for which we will
formulate the high-dimensional elliptic variational problems of interest to us.   For each $j=1,2,\dots,d$, we consider the separable Hilbert space
\[\HH_j:=\HH_j(D):=L_2(D_1)\otimes\cdots\otimes L_2(D_{j-1})\otimes H_j\otimes
L_2(D_{j+1})\otimes\cdots\otimes L_2(D_{d}),\]
(not to be confused with $H_j$), with its natural norm $\|\cdot\|_{\HH_j}$ and inner product. From these component spaces, we define
 $$
\HH:= \bigcap_{j=1}^d \HH_j,
$$
which we  equip with the norm $\|\cdot\|_{\HH}$, defined by
\beqn
\label{H1-norm1}
\|v\|_{\HH}^2 := \sum_{j=1}^d \|v\|_{\HH_j}^2.
\eeqn
In the following we will introduce and use an equivalent norm for $\HH$ based on eigenexpansions.
Again, this gives  rise to a Gel'fand triple
\beqn
\label{G-triple}
\HH \subset \LL_2(D) \subset \HH'  ,
\eeqn
with dense continuous embeddings.  For example, if $\mu_j$ is the Lebesgue measure on $D_j$ and $H_j$ is $H^1(D_j)$, $j=1,\ldots,d$, then $\HH$ is identical to the standard Sobolev space $H^1(D)$, and the above norm is equivalent to the usual $H^1(D)$ norm.

The bilinear form on the space $\HH\times \HH$, under consideration,  is given by
\beqn
\label{H-a}
\mathfrak{b}(v,w):=  \sum_{j=1}^d \dup{(\cI_1\otimes \cdots \otimes \cI_{j-1}\otimes \cB_j\otimes \cI_{j+1}\otimes \cdots\otimes \cI_d)v,w},\quad v, w\in \HH,
\eeqn
where $\cI_j$ is the identity operator on $L_2(D_j)$, $j=1,\dots,d$.  This form is $\HH$-elliptic, i.e., there exist constants $0< c_a, c_a'< \infty$ such that
\beqn
\label{H-ell}
|\mathfrak{b}(v,w)|\leq c_a\|v\|_{\HH}\|w\|_{\HH},\qquad \mathfrak{b}(v,v)\geq c_a'\|v\|_{\HH}^2\qquad \forall\,v, w\in \HH.
\eeqn
Furthermore, $\mathfrak{b}$ is symmetric; i.e., $\mathfrak{b}(v,w) = \mathfrak{b}(w,v)$ for all $w, v \in \HH$.
Thus, the symmetric linear operator
\beqn
\label{Adef}
\cB : \HH \to \HH', \quad \mbox{defined by }\,  \dup{\cB v,w}= \mathfrak{b}(v,w) \quad \forall\, v,w\in \HH ,
\eeqn
 is an isomorphism of the form
\beqn
\label{tensor-sum}
\cB =  \sum_{j=1}^d\cI_1\otimes \cdots \otimes \cI_{j-1}\otimes \cB_j\otimes \cI_{j+1}\otimes \cdots\otimes \cI_d,
\eeqn
which is the sum of rank-one tensor-product operators.

The central theme of the subsequent sections  is to show, under  suitable notions of tensor-sparsity, that for
$f\in\HH'$ the solution $u$ of the variational problem:
\beqn
\label{varprob}
\mbox{Find $u \in \HH^1$ such that}\quad \mathfrak{b}(u,v)=\dup{f,v}\quad \forall\, v\in \HH^1,
\eeqn
will inherit a certain tensor-compressibility from that of $f$.  This means that for such right-hand sides $f$ the solution $u$ avoids the curse of dimensionality.

%%%%%%%%%%%%%%%%%%%%%%%%%%%%
\subsection{Spectral representations and generalized Sobolev spaces }\label{sect:spec}
%%%%%%%%%%%%%%%%%%%%%%%%%%%%%
In this section, we  apply well-known results on elliptic operators and spectral theory to the operator $\cB$ to obtain an eigensystem and to define {an associated scale of Hilbert spaces that can be viewed as generalizations of classical Sobolev spaces. In fact, for specific examples of $\cB$ and for a certain range of smoothness scales they agree with classical Sobolev spaces (with equivalent norms)}. The next lemma, which we quote from \cite{FS,FS:Arxiv}, is a version of the Hilbert--Schmidt theorem and will be relevant in the discussion that follows.

\begin{lemma}\label{lem:abstractEV}
Let $\mathcal{H}$ and $\mathcal{V}$ be separable infinite-dimensional Hilbert spaces, with $\mathcal{V}$ continuously and densely embedded in $\mathcal{H}$. Let $a\colon \mathcal{V} \times \mathcal{V} \to \mathbb{R}$ be a nonzero, symmetric, bounded and coercive
bilinear form. Then, there exist sequences of real numbers $(\lambda_n)_{n \in \mathbb{N}}$ and unit
$\mathcal{H}$ norm members $(e_n)_{n \in \mathbb{N}}$ of $\mathcal{V}$, which solve the eigenvalue problem:
\begin{equation}\label{variational-ev}
\mbox{Find $\lambda \in \mathbb{R}$ and
$e \in \mathcal{H} \setminus \{ 0 \} $ such that}\quad \mathfrak{b}(e,v) = \lambda (e, v)_{\mathcal{H}} \quad \forall\, v \in \mathcal{V},
\end{equation}
where $(\cdot,\cdot)_{\mathcal{H}}$ signifies the inner product of $\mathcal{H}$.
The $\lambda_n$, which can be assumed to be in increasing order with respect to $n$,
are positive, bounded from below away from $0$, and
$\lim_{n\to\infty}\lambda_n = \infty$.

Moreover, the $e_n$ form an $\mathcal{H}$-orthonormal system whose $\mathcal{H}$-closed span is
$\mathcal{H}$ and the rescaling $e_n/\sqrt{\lambda_n}$ gives rise to a $\mathfrak{b}$-orthonormal
system whose $\mathfrak{b}$-closed span is $\mathcal{V}$. Thus, we have
\begin{equation}\label{Fourier-Parseval-H}
h = \sum_{n=1}^\infty ( h, e_n )_{\mathcal{H}} e_n,  \quad  
\norm{h}_{\mathcal{H}}^2 = \sum_{n=1}^\infty \left[( h, e_n )_{\mathcal{H}}\right]^2,\qquad   h \in \mathcal{H},
\end{equation}
as well as
\begin{equation}\label{Fourier-Parseval-V}
v = \sum_{n=1}^\infty \mathfrak{b}\!\left(v, \frac{e_n}{\sqrt{\lambda_n}}\right) \frac{e_n}{\sqrt{\lambda_n}}, \quad  
\norm{v}_{\mathfrak{b}}^2 := \mathfrak{b}(v,v) = \sum_{n=1}^\infty \left[\mathfrak{b}\!\left( v, \frac{e_n}{\sqrt{\lambda_n}} \right)\right]^{2}, \qquad   v \in \mathcal{V}.
\end{equation}
Furthermore, one has
\begin{equation}\label{fastDecay}
h \in \mathcal{H} \quad\text{and}\quad \sum_{n=1}^\infty \lambda_n \left[( h, e_n)_{\mathcal{H}} \right]^2 < \infty \iff h \in \mathcal{V}.
\end{equation}
\end{lemma}
\begin{proof}
The proofs of the stated results can be partially found in
textbooks on functional analysis (see, for example, Theorem VI.15 in Reed \& Simon \cite{ReedSimon}
or Section 4.2 in Zeidler \cite{Zeidler}). A version of the proof for the special
case in which $\mathcal{V}$ and $\mathcal{H}$ are standard Sobolev spaces is contained in
Section~IX.8 of Brezis \cite{Brezis}; using the abstract results
in Chapter~VI of \cite{Brezis}, the result in Section~IX.8 of \cite{Brezis}
can be easily adapted to the setting of the present theorem.
For a detailed proof we refer to Lemmas 15 and 16 in  \cite{FS:Arxiv}.
\end{proof}

\smallskip

It follows from \eqref{Di-ell} and Lemma \ref{lem:abstractEV} that for each $j\in  \{1,\ldots,d\}$ there exists
an eigensystem $(e_{j,n})_{n\in \N}\subset  H_j $ in the factor space $H_j$, with the properties
\beqn
\label{eigenj}
\langle e_{j,n},e_{j,m}\rangle  =\delta_{n,m},\quad {\cB}_je_{j,n}=\lambda_{j,n}e_{j,n},\quad n, m\in \N,\,\,j=1,\dots,d,
\eeqn
where $\lambda_{j,n}\geq \lambda_0>0$ are increasing in $n$, and $\lambda_{j,n} \rightarrow \infty$ as $n \rightarrow \infty$ for each $j=1,\dots,d$.   Therefore, by  \eref{tensor-sum},
\beqn
\label{eigen}
e_\nu := e_{1,\nu_1}\otimes \cdots \otimes e_{d,\nu_d},\quad \lambda_\nu:=\lambda_{1,\nu_1}+\cdots +
\lambda_{d,\nu_d},\quad \nu\in \N^d,
\eeqn
satisfy
\beqn
\label{eigenD}
\cB e_\nu = \lambda_\nu e_\nu,\quad \mathfrak{b}(\lambda_\nu^{-1/2} e_\nu,\lambda_{\mu}^{-1/2}e_\mu)=
\delta_{\nu,\mu},\quad\nu,\mu \in\N^d,
\eeqn
and hence
\beqn
\label{eigenexp}
{\rm e}^{-\cB}e_\nu = {\rm e}^{-\lambda_\nu}e_\nu = \bigotimes_{j=1}^d {\rm e}^{-\lambda_{j,\nu_j}}e_{j,\nu_j},\quad \nu\in \N^d.
\eeqn

Since  $\HH$ is dense  in $\HH'$,  the linear span of the system of eigenfunctions $(e_\nu)_{\nu\in\N^d}$ is also dense in $\HH'$.
We now define the fractional-order {(generalized Sobolev)} space $\HH^s$, $s\ge 0$, as the set of all $v\in\HH'$ for which
\beqn
\label{Hs-norms}
\|v\|_s^2:=\|v\|^2_{\HH^s}:=  \sum_{\nu\in\N^d}\lambda_\nu^s|\langle v,e_\nu\rangle|^2 < \infty.
\eeqn
In particular, $\|\cdot\|_{\HH}= \|\cdot\|_1$.  Furthermore, when $s<0$ we define the spaces $\HH^{s}:= (\HH^{-s})'$,  by duality.  It is easy to see that their norms are also given by \eref{Hs-norms}.   The spaces
\beqn
\label{Hs}
\HH^s := \{v\in \HH' : \|v\|_s <\infty\},\quad s\in\R,
\eeqn
form a scale of  {separable} Hilbert spaces.   {Note further that, thanks to the orthogonality property in \eqref{eigenj} and the definition of the norm \eqref{Hs-norms}, $e_\nu \in \HH^s$
for any $\nu \in \mathbb{N}^d$ and any $s \in \mathbb{R}$.} In classical settings, when $\mathfrak{B}$ is an elliptic differential operator,
the spaces $\HH^s$ agree with Sobolev spaces for a certain range of $s$, which depends
on the geometry of the domain $D$ and the coefficients in the operator.

Having introduced this scale of spaces, let us note that the operator $\cB$, while initially defined on $\HH$ can now be extended to all of the spaces $\HH^s$, $s\in\R$:  if $v\in\HH^s$, then
\be
\label{extB}
\cB v:= \sum_{\nu\in\N^d} \lambda_\nu\langle v,e_\nu\rangle e_\nu.
\ee
So $\cB$ is an isometry between $\HH^t$ and $\HH^{t-2}$, for all $t\in\R$,
\be
\label{isometry}
\|\cB v\|_{t-2}=\|v\|_{t}.
\ee
{Hence,  for any $t\in\R$, and any $f\in \HH^t$, the variational problem:
\beqn
\label{varprob1}
\mbox{Find $u \in \HH^{t+2}$ such that}\quad \mathfrak{b}(u,v)=\dup{f,v}\quad \forall\, v\in  {\HH^{-t}},
\eeqn
has a unique solution.}
It will be convenient to interpret \eref{varprob1} as an operator equation
\beqn
\label{operatoreq}
\cB u=f.
\eeqn

The next remark records some related useful facts.
 \begin{remark}
\label{remcons}
For any $v\in \HH^s$, {$s\in\R$}, one has
\begin{eqnarray}
\label{Fourier1}
v &= &\sum_{\nu\in \N^d}\langle v,e_\nu\rangle  e_\nu, \quad \mbox{in $\HH^s$},
 \\
\label{Fourier2}
\cB^{-1}v &=& \sum_{\nu\in \N^d}\lambda_\nu^{-1}\langle v,e_\nu\rangle  e_\nu ,\quad \mbox{in $\HH^{s+2}$},\\
\label{Fourier3}
  {\rm e}^{-\cB}v &= &\sum_{\nu\in\N^d} {\rm e}^{-\lambda_\nu}\langle v, e_\nu\rangle e_\nu,\quad  \mbox{in $\HH^s$.
  }
 \end{eqnarray}
 \end{remark}

%%%%%%%%%%%%%%%%%%%%
\subsection{Rank-one tensors}
%%%%%%%%%%%%%%%%%%%%%

We proceed to establish in this section some facts concerning rank-one tensors  that will be used several times later in the paper.
The following two observations   relate the regularity of a rank-one tensor to the regularity of its factors.

\begin{lemma}
\label{lem:H-1}
 Let $s\ge 0$.  
A rank-one tensor $\tau=\tau_1\otimes \cdots\otimes \tau_d$ belongs to $\HH^s$ if and only if
$\tau_j \in H^{s}_j$,  $j=1,\dots,d$.  Moreover, when $\tau\in \HH^s$, there is a representation $\tau=\tau_1\otimes\cdots\otimes \tau_d$ for which
 \be
\label{taunormHs}
\max_{j=1,\ldots,d}
  \|\tau_j\|_{H^s_j} \Big(\prod_{i\neq j}\|\tau_i\|_{L_2(D_i)}\Big)\le \|\tau\|_s \le  (d^{\max\{0,s-1\}})^{1/2}\sum_{j=1}^d\|\tau_j\|_{H^s_j}
  \Big(\prod_{i\neq j}\|\tau_i\|_{L_2(D_i)}\Big)\rd{.}
\ee
  \end{lemma}
\begin{proof}
%(i)
We begin by noting that \eqref{taunormHs} holds trivially for $\tau =0$. Let us therefore assume that 
 $\tau\in \HH^s\setminus\{0\}$, with $s \geq 0$.
Assume first that  each $\tau_j\in H_j^s$, $j=1,\dots,d$.
From the definition   \eref{Hs-norms},   we have that
\begin{eqnarray*}
\|\tau\|_s ^2& = & \sum_{\nu\in\N^d}(\lambda_{1,\nu_1}+\cdots + \lambda_{d,\nu_d})^s\dup{\tau_1,e_{1,\nu_1}}^2\cdots \dup{\tau_d,e_{d,\nu_d}}^2\\
& \leq & d^{\max\{0,s-1\}}\sum_{\nu\in\N^d}(\lambda_{1,\nu_1}^s+\cdots + \lambda_{d,\nu_d}^s)\dup{\tau_1,e_{1,\nu_1}}^2\cdots \dup{\tau_d,e_{d,\nu_d}}^2\\
&=& d^{\max\{0,s-1\}}\sum_{j=1}^d \Big(\sum_{\nu_j=1}^\infty \lambda_{j,\nu_j}^s\dup{\tau_j,e_{j,\nu_j}}^2\Big)\Big(\prod_{i\neq j}\|\tau_i\|^2_{0}\Big)
= d^{\max\{0,s-1\}}
\sum_{j=1}^d\|\tau_j\|_s^2 \Big(\prod_{i\neq j}\|\tau_i\|^2_{0}\Big).
\end{eqnarray*}
Now, we can replace the right-hand side of the last inequality by the right-hand side of \eref{taunormHs}
because the $\ell_2$ norm does not exceed the $\ell_1$ norm.  Thus, $\tau\in \HH^s$ and we have established the second inequality in \eref{taunormHs}.

 {Concerning, the first inequality in \eqref{taunormHs},    for any fixed} $j\in\{1,\dots,d\}$, we have
\begin{eqnarray*}
\|\tau\|_s ^2& = & \sum_{\nu\in\N^d}(\lambda_{1,\nu_1}+\cdots + \lambda_{d,\nu_d})^s\dup{\tau_1,e_{1,\nu_1}}^2\cdots \dup{\tau_d,e_{d,\nu_d}}^2\\
&\ge&\sum_{\nu\in\N^d}\lambda_{j,\nu_j}^s\dup{\tau_1,e_{1,\nu_1}}^2\cdots \dup{\tau_d,e_{d,\nu_d}}^2\\
&\ge&\|\tau_j\|^2_{H_j^s}\Big(\prod_{i\neq j}\|\tau_i\|^2_{0}\Big) = \|\tau_j\|^2_{H_j^s}\Big(\prod_{i\neq j}\|\tau_i\|^2_{L_2(D_i)}\Big),
\end{eqnarray*}
which yields \eref{taunormHs}  {for $\tau\in \HH^s\setminus\{0\}$, with $s \geq 0$}. 
\end{proof}
 
\smallskip

According to Lemma \ref{lem:H-1}, %(i),
$\tau \in \HH^s$ for $s \geq 0$ if and only if $\tau \in \HH^{s,\ldots,s}:= \otimes_{j=1}^d H^s_j$; for related results we refer to \cite{SU}, where Besov and Sobolev spaces of dominating mixed smoothness are shown to be tensor products of Besov and Sobolev spaces of univariate functions.

 We record the following consequence of the above observations  in particular  for later use in
\S \ref{sec:proofs}.
\begin{corollary}
\label{cor:duality}
For any $s\in \mathbb{R}$ there exist\rd{s} a constant  $C$ such that, for any $\tau =\tau_1\otimes \cdots \otimes \tau_d \in \HH^{-s}$,
\beqn
\label{crossnorm}
\prod_{j=1}^d \|\tau_j\|_{H^{-s}_j} \leq C \|\tau\|_{-s}.
\eeqn
\end{corollary}
\begin{proof}
Since, for $s\geq 0$, $\HH^{s,\ldots,s}$  is  continuously embedded in  $\HH^s$\rd{,} there exists a constant $C$
such that we have
$$
\|\tau \|_s\leq C\|\tau\|_{\HH^{s,\ldots,s}} = \prod_{j=1}^d \|\tau_j\|_{H_j^{s}}.
$$
By duality, $(\HH^s)'=\HH^{-s}$  is continuously embedded in $(\HH^{s,\ldots,s})'$. Since
$$
(\HH^{s,\ldots,s})' = (H^s_1 \otimes \cdots \otimes H^s_d)' = (H^s_1)'\otimes \cdots \otimes (H^s_d)' = H_1^{-s} \otimes \cdots \otimes H^{-s}_d
$$
is a tensor-product Hilbert space endowed with the corresponding cross-norm, the claim follows.

\end{proof}

\smallskip

While the collection of all rank-one tensors in $\HH^s$ whose $\HH^s$ norm is uniformly bounded is not compact
in $\HH^s$, one has the following
consequence of Lemma \ref{lem:H-1}.

\begin{lemma}
\label{lem:compact} {For any  $C>0$} and any $s'\in \R$,
the collection
$$
\cT(s',C) := \{\tau = \tau_1\otimes \cdots \otimes \tau_d \in \HH^{s'}\,:\, \|\tau\|_{{s'}}\leq C\}
$$
is  a compact subset of $\HH^s$ provided $s'>s$.
\end{lemma}
\begin{proof}  {It is easy to see (see Lemma \ref{soblemma}) that any closed bounded ball $B$ of $\HH^{s'}$  is a compact subset of $\HH^s$.  Therefore, we only need to show that
the set $\cT(s',C)$ is closed in $\HH^s$.  Let $\tau^{(n)}$ be a sequence from this set\rd{,} which converges in the topology of $\HH^s$ to a function $g\in \HH^s$.   We need only show that $g=g_1\otimes\cdots\otimes g_d$ for some $g_j\in H_j^s$.}
If $g=0$, then clearly this limit function is in  $\cT(s',C)$.     If $\|g\|_{s}>0$,
then $\|\tau^{(n)}\|_{s}\ge C>0$ for $n$ sufficiently large.

{Consider now first the case $s'>0$.} Since the norm topology of the spaces $\HH^s$ gets weaker as $s$ decreases, it is sufficient to prove the Lemma in this case {for $s'>s \geq 0$}.
We can assume without loss of generality that for each $n$  the norms  $\|\tau_j^{(n)}\|_0$ are all equal.    {It follows from Lemma \ref{lem:H-1} for $s'\geq 0$}, that each of the components
 {satisfies} $\|\tau^{(n)}_j\|_{s'}\le C'$ for an absolute constant $C'$.
Hence, each of the sequences $(\tau_j^{(n)})_{n\ge 1}$ is compact in $H_j^s$, $j=1,\dots,d$.  A common subsequence of
each of them, indexed by $(n_k)$, converges to a limit $g_j\in H_j^s$, with $\|g_j\|_{H_j^s}\le C'$, $j=1,\dots,d$.

We claim that $g$ is equal to $g_1\otimes\cdots \otimes g_d$.   Indeed, for any $\tau^{(n)} $, we can write
\be
\label{ae24}
(g_1\otimes\cdots \otimes g_d)-(\tau_1^{(n)}\otimes\cdots \otimes \tau_d^{(n)})= \sum_{j=1}^d  g_1\otimes\cdots
\otimes g_{j-1}\otimes (g_j-\tau_j^{(n)})\otimes\tau_{j+1}^{(n)}\otimes   \cdots \otimes \tau_d^{(n)},
\ee
{with an obvious interpretation of the summands for $j\in \{1,\rd{\ldots,}d\}$}.
 Given the fact that for each of the components $g_j$, we have $\|g_j\|_{L_2(D_j)}\le \|g_j\|_{H_j^s}\le C'$, and a similar bound for the components of the $\tau^{(n)}$, we  {infer from \eref{taunormHs} in Lemma \ref{lem:H-1}} that the $\HH^s$ norm of each of the terms in the sum appearing \eref{ae24} tends to zero as $n\to \infty$.  {This proves the claim for $s> 0$.}

{When $s'\leq 0$,   we renormalize so that for each $n$, all of the norms $\|\tau_j^{(n)}\|_{H_j^{s'}}$,
$j=1,\dots,d$,  are equal.
Then, \eref{crossnorm} in Corollary \ref{cor:duality} gives that all of these norms have a uniform bound.  Hence, we again   derive  the existence of   subsequences
$(\tau_j^{(n)})_{n\ge 1}$ that converge in $H_j^s$ to respective limits $g_j$. To show that $g$ agrees with $g_1\otimes \cdots\otimes g_d$
it suffices to prove that $\|g_1\otimes \cdots\otimes g_d - \tau_1^{(n)}\otimes\cdots \otimes \tau_d^{(n)}\|_{\HH^{s,\ldots,s}}\to 0$, $n\to \infty$.
Employing again the decomposition \eref{ae24}, this in turn, follows by using that $\|\cdot\|_{\HH^{s,\ldots,s}}$ is a
tensor product norm. That} finishes the proof of the Lemma.
 \end{proof}

We now turn to the central question of this paper:   \textit{Can some formulation of tensor-sparsity, or more generally tensor-compressibility, help to break the curse of dimensionality when solving the variational problems \eref{varprob} and  \eref{varprob1}?}

 %%%%%%%%%%%%%%%%%%%%%%
\section{Tensor-Sparsity and Compressibility}\label{sec:ts}

  Numerical methods for solving partial differential
equations are based on some form of approximation.  In our case, we want to approximate the solution $u$ to
\eref{varprob} in the $\HH^1$ norm.  The simplest numerical methods utilize a sequence $(V_n)_{n=1}^\infty$ of {\em linear} spaces
with  $\dim(V_n)\sim n$ for the approximation.   Adaptive and more advanced numerical methods replace the $V_n$ by nonlinear spaces
$\Sigma_n$.   {Since the case of linear subspaces is subsumed by the nonlinear case,
we continue our discussion   in the nonlinear context.  We assume in the following that the elements of $\Sigma_n$ are described by $\sim n$ parameters.}

  To understand how well a potential numerical method built on $\Sigma_n$  could  perform, we need first to understand the approximation
  capabilities of   $\Sigma_n$.   {The following quantification of performance will be described in a general setting of approximation in a Banach space $X$ and therefore we assume that each $\Sigma_n\subset X$.}   For  any function $v\in X$, the approximation error
\be
\label{ae}
\sigma_n(v)_{X}:=\inf_{g\in \Sigma_n}\|v-g\|_{X}
\ee
tells us how well $v$ can be approximated by the elements of $\Sigma_n$.   In the case of most interest to us, $X=\HH^1$,  $\sigma_n(u)_{\HH^1}$ gives
the optimal performance, in computing $u$, that would be possible by any numerical method based on this form of approximation.
Of course, there is also the  problem  of  constructing a numerical algorithm with this level of performance, and proving that the implementation of
the algorithm can be achieved with $\mathcal{O}(n)$, or perhaps slightly more, computations.

Given a sequence $(\Sigma_n)_{n\ge 0}$, with $\Sigma_0:=\{0\}$, the approximation space $\cA^r:=\cA^r((\Sigma_n)_{n \geq 0},X)$ consists of all functions $v\in X$ such that
\be
\label{ac}
\sigma_n(v)_{X}\le M(n+1)^{-r},\quad n\ge 0,
\ee
and the smallest  such $M$ is the norm of $v$ for { the approximation space $\cA^r$}.
More generally, we have  the approximation classes  $\cA_q^r:=\cA_q^r((\Sigma_{n})_{n \geq 0},X)$, which are defined,
for any $0<q\le \infty$ and $r>0$, as the set  of all $v\in X$ such that
\be
\label{ac2}
\|v\|_{\cA_q^r((\Sigma_{n})_{n \geq 0},X)}:= \left\{\begin{array}{cl}
\bigg(\sum_{n=0}^\infty\, [(n+1)^r\sigma_n(v)_X]^q\frac{1}{n+1}\bigg)^{1/q} &\mbox{if $0 < q<\infty$},\\
{\displaystyle \sup_{n\ge 0}\, (n+1)^r\sigma_n(v)_{X}},& \mbox{if $q=\infty$},                                                           \end{array}\right.
 \ee
is finite.  So, $\cA^r=\cA^r_\infty$.

When $X$ is an $L_p$ space or a Sobolev space, the approximation spaces for classical methods based on
polynomials, splines or wavelets are well studied and are either completely characterized or very well understood through embeddings.  They turn out to be describable by  Sobolev spaces or  Besov spaces (see \cite{DNL}).
In particular, these known results show that classical methods of approximation suffer from the curse of dimensionality.   For example, membership {in} the approximation
space $\cA^r((\Sigma_n)_{n \geq 0}, \HH^1)$ for such methods requires the function $v$ to have $rd$ orders of smoothness
{in $\HH^1$}.
As $d$ increases, this requirement will not be satisfied for typical right-hand sides $f$ in \eqref{varprob}
that arise in applications.  {This means that numerical methods built on
such classical approximation spaces are not effective when dealing with high-dimensional problems of the form \eref{varprob}}.

One fairly recent viewpoint reflected by the discussion in the  {Introduction}, and the one taken here,  is that, in contrast to classical constructions, spaces $\Sigma_n$ built on suitable tensor formats
may  perform much more favorably on certain high-dimensional problems, and, indeed, break the curse of dimensionality for them.  However,  to the best of our knowledge no such rigorous results, in this direction, exist as of yet.    In this subsection, we put forward some natural possibilities for defining sparse tensor classes $\Sigma_n$,
subject to a budget of $\sim n$ parameters.   A later subsection  of this paper  will show that the nonlinear spaces $(\Sigma_n)_{n \geq 0}$ built on   these tensor constructions  effectively capture the solutions to the variational problem \eref{varprob}.  To show this effectiveness,  one has to understand what conditions on the right-hand  side $f$ guarantee that the solution $u$ is in an approximation class $\cA^r((\Sigma_n)_{n \geq 0}, \HH^1)$ for a reasonably large
value of $r$.    One can view any theorem that deduces membership of $u$ in  $\cA^r((\Sigma_n)_{n \geq 0}, \HH^1)$  from a suitable property of $f$ as a regularity theorem for the variational problem under consideration.  Such regularity results are then proved in \S \ref{sec:main}.

%%%%%%%%%%%%%%%%%%%%%%%%%%%%
\subsection{Formulations of tensor-sparsity}\label{ssec:3.1}
%%%%%%%%%%%%%%

In this section we introduce and compare several possible formulations of tensor-sparsity and tensor-compressibility.  The main property one seeks in such sparsity classes $\Sigma_n$ is  that the elements in $\Sigma_n$ should depend on $\sim n$ parameters.
 The common feature of these formulations is that they are based on low-dimensional eigensystems whose tensor products form a fixed {\em background basis}.
In later sections we will propose an alternative way of defining {\em tensor-sparsity}, which is independent of any specific background system.

   \subsubsection{$n$-term approximation}
   \label{ss:nterm}
   We first discuss the well-studied nonlinear approximation procedure of $n$-term approximation from a
basis $(\vphi_\nu)_{\nu \in \mathbb{N}^d}$.   In our setting,    we know that $(e_\nu)_{\nu\in\N^d}$
is a tensor basis for $\HH^t$, {$t\in \R$,}   and in particular for $\HH^1$.   In $n$-term approximation, the space $\Sigma_n$ consists of all functions $g=\sum_{\nu\in\Lambda}c_\nu e_\nu$ where $\Lambda\subset \N^d$ is a subset of indices with cardinality at most $n$.   The functions in $\Sigma_n$ are said to be sparse of order $n$. Sometimes one places further restrictions on $\Lambda$ to ease numerical implementation and the search for the set of the best $n$ co-ordinates.   For example, when $d=1$, a typical assumption is
that the indices must come from the set $\{1,\dots,n^A\}$ where $A$ is a fixed positive integer.

 Since $(e_\nu)_{\nu\in \N^d}$ is also a basis for $\HH^{t}$,  the same nonlinear space $\Sigma_n$ can be used to approximate the right-hand side $f$ of \eref{varprob}.  Let $\cB u=f$. If $g=\sum_{\nu\in\Lambda}c_\nu e_\nu $ is any element of $\Sigma_n$, then $\bar u:=\cB^{-1}g=\sum_{\nu\in\Lambda}c_\nu \lambda_\nu^{-1}e_\nu $ is also in $\Sigma_n$ and  satisfies
\be
\label{nterm1}
\|u-\bar u\|_{t+2}=\|f-g\|_{t}.
\ee
Therefore,
\be
\label{nterm2}
\sigma_n(u)_{\HH^{t+2}}\le \sigma_n(f)_{\HH^{t}}.
\ee

We therefore have the following simple regularity theorem: for each $r>0$, $f\in \cA^r((\Sigma_n)_{n \geq 0},\HH^{t})$ implies that $u\in\cA^r((\Sigma_n)_{n \geq 0},\HH^{t+2})$.   While this is a favorable-looking result, that seems to break the curse of dimensionality, closer inspection reveals that for $u$ or $f$  to belong to the corresponding  $\cA^r$ space,
one requires that when the coefficients in its eigenexpansion are rearranged in decreasing order
the $n$-th coefficient should decay like $n^{-r-1/2}$.  This in turn is like a smoothness condition of order $rd$ placed on $f$ and $u$.
Hence, the conditions on the data $f$ become more and more restrictive when  the dimension $d$ increases.

\subsubsection{Variable rank-one tensors}
\label{ss:rankone}  Our next  sparsity model  again draws on the eigensystem $(e_\nu)_{\nu\in\N^d}$ as a reference basis but now employs low-dimensional eigenbases to parametrize variable rank-one tensors.
 More precisely, let us first consider,  for any $j\in \{1,\dots,d\}$, the eigenbasis $(e_{j,k})_{k=1}^\infty$
  associated with the low-dimensional operator $\cB_j$ on $H_j$.  The
$e_{j,k}$, $k=1,2, \ldots$,  are functions that depend only on $x_j$.  For each $j \in \{1,\dots,d\}$, let $\Sigma_m^j$ be the collection of all linear combinations of at most $m$ of the $j$th eigensystem elements $e_{j,k}$. If we fix a value of $j \in \{1,\ldots,d\}$, then the space $\Sigma_m^j$
is the space of $m$-sparse functions  for this basis.  Thus, any function $g_j\in\Sigma_m^j$ can be written as
\be
\label{msparse}
g_j=\sum_{k\in \Gamma_j}c_ke_{j,k},
\ee
where $\#(\Gamma_j)\le m$.

Next consider any $\bm=(m_1,\dots,m_d)$ and functions $g_j\in\Sigma_{m_j}^j$, $j=1,\dots,d$.  Then, the functions
\be
\label{rankone1}
g=\bigotimes_{j=1}^d g_{j}
\ee
are rank-one tensors, which depend only on the $|\bm|:=m_1+\cdots +m_d$ positions $\Gamma_j$, $j=1,\dots,d$, and the $|\bm|$
coefficients of the $g_j$.  We define $\Sigma_\bm$ to be the set of all such functions $g$.

Note that if $g\in\Sigma_\bm$  is expanded
in terms of the eigenbasis $(e_\nu)_{\nu \in \mathbb{N}^d}$, it would involve $m_1\cdots m_d$ terms.   However, what is key is that
the coefficients in the eigenbasis expansion of $g$
only depend --- in a highly nonlinear fashion --- on $2|\bm|$ parameters, namely the $|\bm|$ indices in the sets $\Gamma_j$ and the $|\bm|$ corresponding coefficients that appear in $g_j$, $j=1,\dots,d$.

 In order to have a more compact notation for these functions $g$ in what follows,
we introduce $\Gamma:=\Gamma_1  {\times\cdots\times}\Gamma_d$, which is the set of indices appearing in $g$
and $C=C_g$ is the rank-one tensor of coefficients
\be
\label
{coeffs}
C=C_g:=(c_{\nu_1,1}\cdots c_{\nu_d,d})_{ \nu\in\Gamma}.
\ee
We also use the notation
\be
\label{notation} T_{\Gamma,C}=g,
\ee
when we want to explicitly indicate for a $g\in\Sigma_\bm$ the index set $\Gamma$ and the tensor of coefficients $C$

We define the sparsity space $\Sigma_n$ as the set of all functions
\be
\label{st}
g=\sum_{k=1}^sg_k,\quad g_k\in\Sigma_{\bm_k},\quad \sum_{k=1}^s|\bm_k|\le n,
\ee
 and introduce, for any $v\in \HH^t$,   the approximation error
\be
\label{aebm}
\sigma_n(v,\HH^t):=\inf_{g\in\Sigma_n}\|v-g\|_{t}.
\ee

Obviously a function in $\Sigma_n$ depends on at most $n$ positions and $n$ coefficients, so it can be viewed
as an analogue of $n$-term approximation except that it strongly exploits tensor structure.  Indeed, if such a function
were expanded into the tensor basis $(e_\nu)_{\nu \in \mathbb{N}^d}$, it would possibly require $(n/d)^d$ terms.
 Note that this definition of $\Sigma_n$ includes all of the $n$-term functions of \S\ref{ss:nterm}.

\subsubsection{Restricted index sets}
\label{ss:restricted}

 The space $\Sigma_n$ as  defined by \eref{st} is not suitable for use in numerical computations.  One  reason for this is that there is no control on the  index sets appearing in the $\Gamma_j$.  This can be circumvented in applications by
 placing restrictions on the set of indices similar to the restrictions in  $n$-term approximation already described.  To make all of this formal, we suppose that for each $m\geq 1$ and
 $j\in \{1,\dots,d\}$, we have a finite  set $\cR_{m,j}\subset \N$.  We will require that for any given $\bm$, the sets $\Gamma_j$ are subsets
 of $\cR_{m_j,j}$ for each $j=1,\dots,d$,  or, in other words, that $\Gamma\subset \cR_\bm$ where $\cR_\bm:=\cR_{m_1,1}\times\cdots\times \cR_{m_d,d}$.
 The typical choice for the \textit{restriction sets} $\cR_{m,j}$ is $\cR_{m,j}:=\{1,\dots,m^A\}$ where $A$
 is a fixed integer.  Such choices are independent of $j$.  In what follows, if we wish to indicate the difference between $\Sigma_n$ and the space defined with restrictions,
 we will denote the latter by $ \Sigma_{n}(\cR)$.

One possible restriction, certainly a strong one,  is that the sets
 $\cR_{m,j}$ are all one and the same,  and equal to $\{1,\dots,m\}$ for each $j=1,\dots,d$.
 Notice that in that case only the $e_\nu$ with $\|\nu\|_{\ell_\infty}\le m$ are available to be used.  Thus the component spaces $\Sigma_{m_j}^j$ are now linear spaces, however the spaces $\Sigma_\bm$ and $\Sigma_n$ are not  linear spaces.

For later use we record in the following lemma a simple sufficient condition for a rank-one tensor to be in the approximation space $\cA^r$.
 We shall suppose that we are measuring the error in  the norm of $\HH^t$.

 \begin{lemma}
 \label{soblemma}
 Suppose  that $\bm=(m,\dots,m)$.  Consider the above case of restricted approximation when  $\cR_\bm$ contains $\{1,\dots,m\}^d$.   Then, for any $\delta>0$ and any rank-one tensor function $v(x)=v_1(x_1)\cdots v(x_d)$ in $\HH^{t+\delta}$, we have that %
 \be
 \label{roapprox}
  \inf_{g\in \Sigma_\bm}\|v-g\|_t \le [ \lambda_m^*]^{-\delta/2}\|v\|_{{t+\delta}},
 \ee
 where
 \be
 \label{defm*}
 \lambda_m^*:=\min_{1\le j\le d}\lambda_{j,m+1}.
 \ee
 \end{lemma}
\begin{proof}  As has been already noted,  the function  $\displaystyle{g:=\sum_{\nu\le\bm}  \langle v,e_\nu\rangle e_\nu}$ is in $\Sigma_\bm$.   If $\nu\not \le \bm$, then,
 for some $j\in\{1,\dots,d\}$, we have $\nu_j\ge m+1$ and
 $$\displaystyle{\lambda_\nu\ge \lambda_{j,m+1}\ge \lambda_m^*}.$$
 Therefore,  from \eref{Hs-norms} we have
 \be
 \label{terror}
 \|v-g\|^2_{{t}}=\sum_{\nu\not\le\bm } \lambda_\nu^{t+\delta}\lambda_\nu^{-\delta}\langle v,e_\nu\rangle^2\le
[\lambda_m^*]^{-\delta} \sum_{\nu>\bm }  \lambda_\nu^{t+\delta}\langle v,e_\nu\rangle^2\le [\lambda_m^*]^{-\delta}\|v\|_{{t+\delta}}^2.
 \ee
This gives \eref{roapprox}.
 \end{proof}

\smallskip

As a simple example of the above result, we take the $\cB_j$ to be the Laplacian on a Lipschitz domain $D_j \subset \R^2$,
with homogeneous Dirichlet boundary condition on $\partial D_j$, $j=1,\dots,d$.
Then the eigenvalues $\lambda_{j,k}$ grow like  {$k$ and $\lambda_m^*\approx m$}.  Hence the approximation rate for the $v$ in the lemma is of order  {$m^{-\delta/2}$}.
Thus for an investment of  computing $dm$ coefficients, we obtain accuracy $m^{-\delta}$. For an analogous result on the spectral asymptotics of degenerate elliptic operators that arise in the context of Fokker--Planck equations, we refer to \cite{FS}.

\subsubsection{Other bases for rank-one sparse decompositions}
\label{ss:otherbases}
The spaces $  \Sigma_n$ and their restricted counterparts $\Sigma_n(\cR)$ may still be inappropriate for
numerical implementation (except perhaps in the case of Fourier spectral methods) because they require the computation of the eigenvalues and eigenfunctions of the elliptic operator under consideration.   For this reason, numerical algorithms use other readily available tensor bases $(\varphi_\nu)_{\nu\in\N^d}$, such as wavelet or spline bases.   In this case, the role  of $e_\nu$ is replaced by $\varphi_\nu$,
$\nu\in \N^d$, in the definition of $\Sigma_n$ and $\sigma_n$.
   In general, different bases give different approximation classes
$\cA^r$ and these classes should be viewed as ways of measuring smoothness.  Sometimes, it is possible to prove for some choices of bases that the approximation classes are the same.  For example, this is the case for univariate
$n$-term approximation using
wavelet and spline bases.  We do not wish, in this paper, to enter too deeply into this issue since it depends heavily on the properties of $D$ and the constructed bases.  For the construction and analysis of algorithms using
such background base{s} we refer to \cite{B,BD}.

%%%%%%%%%%%%%%%%%%%
\section{Regularity theorems based on tensor-sparsity: data with low regularity}\label{sec:main}
%%%%%%%%%%%%%%%%%%%

We now turn to proving regularity theorems in terms of the approximation classes based on  the  tensor systems  $(\Sigma_n)_{n \geq 0}$
 introduced in \S \ref{ss:rankone}.
  We shall  prove that if the right-hand side $f$ is in an approximation class $\cA^\alpha((\Sigma_n)_{n \geq 0},\HH^{t})$ for some $\alpha>0$, then the solution $u$
of the variational problem:
\beqn
\label{varprob-2}
\mbox{Find $u \in \HH^{t+2}$ such that}\quad \mathfrak{b}(u,v)= \dup{f,v}\quad \forall\,v\in \HH^{-t},
\eeqn
belongs to the approximation class $\cA^{\alpha'}((\Sigma_n)_{n \geq 0},\HH^{t+2})$ for all $\alpha' \in (0,\alpha)$.  To prepare for the proof of this result, we begin with the following lemma.

\begin{lemma}
\label{explemma}  Let $G(x):=1/x$, $x>0$, and fix any $\beta>0$.  For  each  integer $r\ge 1$, there exists a function
\be
\label{expf}
S_r(x):=\sum_{k=1}^r \omega_k\, {\rm e}^{-\alpha_k x},\quad x>0,
\ee
with $\alpha_k= {\alpha_{r,k}}>0$, $\omega_k={\omega_{r,k}}>0$, $k=1,\dots,r$, such that

\noindent
{\rm (i)} we have the error bound\be
\label{expfa+}
\|G-S_r\|_{L_\infty[\beta,\infty)}\le \frac{16}{\beta}{\rm e}^{-\pi\sqrt{r}};
\ee
\noindent
{\rm (ii)} in addition,
\be
\label{expfa1+}
S_r(x)\le \frac{1}{x}\quad \mbox{for all $x\ge \frac{1}{8}{\rm e}^{\pi \sqrt{r}}$}.
\ee

\end{lemma}
\begin{proof} This follows from the results in  \cite{BH1,BH2} and we only sketch the argument therein.
The starting point is that $G$ is a {\em completely monotone} function.
Defining the class of  exponential sums
$$
\mathcal{E}_r:= \Big\{ \sum_{k=1}^r \omega_k\, {\rm e}^{-\alpha_k x}: x>0,\,  \alpha_k >0,\, \omega_k\in \R\Big\},
$$
it can be shown that for any bounded closed interval $[1,R]$, $R>1$, the function
 \beqn
 \label{sab}
 S_{r,R}(x)= \sum_{k=1}^r \omega_{r,k,R}\, {\rm e}^{- \alpha_{r,k,R} x}
 \eeqn
 from $\mathcal{E}_r$ that minimizes
$$
E_{r,[1,R]}(G):= \inf\,\big\{ \| G- S\|_{L^\infty[1,R]}: S \in \cE_r\big\}  
$$
exists, is unique,  and is characterized by an alternation property, from which one deduces that
\be
\label{mon}
0\le S_{r,R}(x)\le \frac{1}{x},\quad x\ge R.
\ee
The following estimate for the error is proved in \cite{BH1}:%
\beqn
\label{BH2}
E_{r,[1,R]}(G)=\|G-S_{r,R}\|_{L_\infty[1,R]}
 \leq 16\, {\rm e}^{-\frac{\pi^2 r}{\log(8R)}}.
\eeqn

In order to obtain a bound on the whole real line one uses that $G$ decays and aims to minimize
$\max\,\{E_{r,[1, R]}, 1/ R)\}$ over all choices $R>1$. Taking $R=R_r:= \frac 18\, {\rm exp}\big(  \pi \sqrt{r})\big)$, gives
\beqn
\label{expest2}
E_{r,[1,R_r]}(G)\leq
 16\, {\rm e}^{-\pi\sqrt{r}}.
\eeqn
Because of \eref{mon} we can replace the interval $[1,R_r]$ in \eref{expest2} by $[1,\infty)$.
This proves \eref{expfa+} and \eref{expfa1+} with $S_r:=S_{r,R_r}$ for the case $\beta=1$.   The case of general $\beta$ follows from a simple rescaling.
 {The positivity of the weights $\omega_{r,k,R}$ in \eref{sab} is known to hold in general for best exponential sum approximations to completely monotone
functions, see \cite{Braess}, and that then implies the desired positivity of the $\rd{\omega_{r,k}}:=\omega_{r,k,R_r}$.}
\end{proof}
\smallskip

The following lemma gives a similar result but with an added restriction on the coefficients $\alpha_k$.
\begin{lemma}
\label{explemma1}  Let $G(x):=1/x$, $x>0$, and fix any $\beta>0$.  For  each  integer $r\ge 1$, there exists a function
\be
\label{expf-bar}
\bar S_r(x):=\sum_{k=1}^r \bar \omega_k\, {\rm e}^{-\alpha_k x},\quad x>0,
\ee
with $\alpha_k=\rd{\alpha_{r,k}}>0$ and $\bar \omega_k=\rd{\omega_{r,k}}\ge 0$, $k=1,\dots,r$, such that the following hold.

\noindent
{\rm (i)} \hspace*{1mm}Whenever $\bar \omega_k>0$, we have  $\alpha_k\ge T_r^{-1}, {\rm where} \  T_r:=\frac{1}{8}{\rm exp}\big(  \pi \sqrt{r})\big)$, $k=1,\dots,r$.

\noindent
{\rm (ii)} We have the error bound
\be
\label{expfa}
\|G-\bar S_r\|_{L_\infty[\beta,\infty)}\le \left(\frac{16}{\beta}+ 8r{\rm e}\right){\rm e}^{-\pi\sqrt{r}}.
\ee
\noindent
{\rm (iii)} In addition,
\be
\label{expfa1}
\bar S_r(x)\le \frac{1}{x} \quad \mbox{for all $x\ge \frac{1}{8}{\rm e}^{\pi \sqrt{r}}$}.
\ee

\end{lemma}
\begin{proof}  We start with the function $S_r=\sum_{k=1}^r\omega_{k}{\rm e}^{-\alpha_kx}$ of the previous lemma
with $\alpha_k>0$ and $\omega_k>0$ for all $k \in \{1,\dots,r\}$.  Suppose that for some $k \in \{1,\dots,r\}$ we have $\alpha_k\le T_r^{-1}$.  Then, $ {\rm e}^{-1} \leq {\rm e}^{-\alpha_k T_r}$ and therefore
$\omega_k {\rm e}^{-1} \le \omega_k {\rm e}^{-\alpha_k T_r}$.  In
view of (ii) of the previous lemma with $x=T_r$, we have
$\omega_k{\rm e}^{-1}\le
\omega_k {\rm e}^{-\alpha_k T_r} \le  T_r^{-1}$, and thus also $\omega_k \le  {\rm e} T_r^{-1}$. Hence, if for each such $\alpha_k$, we set
$\bar\omega_k:=0$ and define $\bar\omega_k = \omega_k$ for all other $k \in \{1,\dots,r\}$, we obtain a new function $\bar S_r$ defined by $\bar S_r(x):=\sum_{k=1}^r\bar\omega_{k}{\rm e}^{-\alpha_kx}$, $x>0$, such that
\be
\label{newsr}
\bar S_r(x)\le S_r(x)\le \bar S_r(x)+ r{\rm e} T_r^{-1} = \bar S_r(x)+ 8r{\rm e}\, {\rm e}^{-\pi\sqrt{r}} \qquad \mbox{for all $x>0$}.
\ee
Part (ii) then follows from \eqref{newsr} and part (i) of Lemma \ref{explemma} via the triangle inequality, while part (i) follows directly from the definition of $\bar{\omega}_k$.

All of the remaining claims of the lemma follow from \eref{newsr} together with the corresponding statement of
Lemma \ref{explemma}.
\end{proof}

Exponential sums such as the one in Lemma \ref{explemma} have been considered in several works
for the purpose of constructing low-rank approximations to inverses of {\em finite-dimensional} linear operators
(i.e., matrices); see e.g.  {\cite{Gr2004,GaHaKh2005}}.
However, in the finite-dimensional setting the metrics for the domain and the range of the linear operators were
taken to be {\em the same}: typically the Euclidean metric. When the discrete operator is to approximate
an infinite-dimensional operator, such as a differential operator, with {\em different} domain and
range topologies (e.g. $\HH^1$ and $\HH^{-1}$), the discrete operator becomes more and more ill-conditioned when
the discretization is refined. As a consequence,   the accuracy of an approximate discrete solution cannot be well estimated
by the corresponding discrete residual. The deviation of the corresponding expansion
in a ``problem-relevant'' function space norm is in general not clear.   Therefore, the fully discrete approach does not immediately provide a rigorous rate distortion result in the continuous PDE setting.

Expanding further on the above point, note that the operator $\cB$ is an isomorphism only
as a mapping between spaces of different regularity levels that are not endowed with tensor product norms. A rigorous assessment of the error has to
take this mapping property into account. That, however, has an adverse effect on tensor sparsity  because the representation of
$\cB^{-1}$ in the eigenbasis is a diagonal operator of {\em infinite rank} since the diagonal entries $\lambda_\nu^{-1}$
are not separable, see also \cite{BD}.  More importantly,  as will be seen next, the actual ranks required to approximate $u$
to within a given tolerance in $\|\cdot\|_s$, for $f$ measured in  $\|\cdot\|_t$, say, strongly depends on the corresponding
``regularity shift''  $s-t$.
 Properly taking such ``regularity shifts''  into account   in the course of low-rank approximation is a central issue  in
the subsequent developments of this paper.

\smallskip

The following result is our first main regularity theorem.  {It will be convenient to define the smallest eigenvalue of $\cB$:
\be
\label{lambdabar}
\underline \lambda:=\min_{\nu\in \N^d} \lambda_\nu=\lambda_{(1,\dots,1)}.
\ee
}

 \begin{theorem}\label{r1th}
 Suppose that $t\in\R$ and  that $f\in\cA^\alpha((\Sigma_n)_{n \geq 0},\HH^{t})$, $\alpha>0$; then, the weak solution $u \in \HH^{t+2}$ to {\rm \eref{varprob1}} satisfies the following inequality:
 \be
 \label{urate}
  \sigma_{n[\log n]^2}(u)_{\HH^{t+2}}\le A\|f\|_{\cA^\alpha((\Sigma_n)_{n \geq 0},\HH^{t})}\,n^{-\alpha},\quad n=1,2,\dots,
  \ee
  where the constant $A$ depends only on $\alpha$ and the smallest eigenvalue {$\underline\lambda >0$}
  of $\cB$.
  In particular, $u$ belongs to $\cA^{\alpha'}((\Sigma_n)_{n \geq 0},\HH^{t+2})$ for all $\alpha' \in (0,\alpha)$.
\end{theorem}

 \begin{proof} We fix $n\ge 1$.  The assumption that $f\in \cA^\alpha((\Sigma_n)_{n \geq 0},\HH^{t})$ implies the existence of a $g_n\in\Sigma_n$
such that
\be
\label{ae11}
\|f-g_n\|_{{t}}\le Mn^{-\alpha},\quad M:=\|f\|_{\cA^\alpha((\Sigma_n)_{n \geq 0},\HH^{t})}.
\ee
Membership in $\Sigma_n$ means that $g_n$ is  of the form
\be
\label{gn}
g_n:=\sum_{j=1}^kh_j,\quad h_j:= T_{\Lambda(j),C(j)},
\ee
 where
\be
\label{Lambda}
\Lambda(j)=\Lambda_1(j)\times\cdots\times \Lambda_d(j)
\ee
and
\be
\label{Ctensor}
C(j)=\bc_1(j)\otimes\cdots\otimes \bc_d(j).
\ee
Here, each $\bc_i(j)$ is a vector in $\R^{m_i(j)}$, where $m_i(j):=\#(\Lambda_i(j))$.
We know that the mode sizes $m(j)=\sum_{i=1}^dm_i(j)=\sum_{i=1}^d\#(\Lambda_i(j))$ of the tensors $C(j)$  satisfy
\be
\label{size}
2\sum_{j=1}^k m(j)\le n.
\ee

Consider  now $\bar u_n=\cB^{-1} g_n$.  We also know that
\be
\label{ae12}
\|u-\bar u_n\|_{{t+2}}=\|f-g_n\|_{{t}}\le Mn^{-\alpha},\quad M:=\|f\|_{\cA^\alpha((\Sigma_n)_{n \geq 0},\HH^{t})}.
\ee
We deduce from \eref{gn} that
\be
\label{bgn}
\bar u_n=\sum_{j=1}^k\bar h_j,\quad  \bar h_j:= \cB^{-1}h_j.
\ee

We will next show how each of the functions $\bar h_j$ can be well approximated by functions from $\Sigma_n$.
To this end, we fix $j$ and write
\be
\label{reph}
 h_j=\sum_{\nu\in \Lambda(j)} c_\nu(j) e_\nu,
 \ee
so that
\be
\label{repbarh}
\bar h_j=\sum_{\nu\in \Lambda(j)} \lambda_\nu^{-1}c_\nu(j) e_\nu.
 \ee
where, of course,  
 the $c_\nu(j)$ are determined by the
 at most $n/2$ parameters defining the factors in \eref{Ctensor}.  
 While, for a fixed $j$,  the $c_\nu(j) $ form a rank-one tensor,  $\lambda_\nu^{-1}$ is not separable so that $\bar h_j$
has infinite rank.  To obtain low-rank approximations to $\bar h_j$ we use Lemma \ref{explemma}, where
   $\underline\lambda >0$ is  the   smallest of the eigenvalues $\lambda_\nu$ of $\cB$, and approximate $\lambda_\nu^{-1}$
   by an exponential sum $S_r(\lambda_\nu)$ which is a sum of $r$ separable terms.

 In fact, using $S_r$ as defined in Lemma \ref{explemma}, with $r$ to be chosen momentarily, we define
\be
\label{htilde}
\hat h_j:=\sum_{\nu} S_r(\lambda_\nu)\,c_\nu(j)\, e_\nu,
\ee
 which can be rewritten as
\be
\label{hjrep}
\hat h_j=\sum_{k=1}^r \omega_{k}\left\{\sum_{\nu\in\Lambda(j)}[{\rm e}^{-\alpha_{k}\lambda_\nu }c_\nu(j)] e_\nu\right\}.
\ee
We then define $\hat u_n = \sum_{j=1}^k \hat h_j$.

From the tensor structure of
$h_j$ (see \eref{gn}) and the fact that ${\rm e}^{-\alpha_k\lambda_\nu}={\rm e}^{-\alpha_k\lambda_{1,\nu_1}}\cdots {\rm e}^{-\alpha_k\lambda_{d,\nu_d}}$ is separable, we see that each of the functions
\be
\label{each}
\sum_{\nu\in\Lambda(j)}{\rm e}^{-\alpha_{k}\lambda_\nu }\,c_\nu(j)\, e_\nu,\quad k=1,\dots,r,
\ee
is in $\Sigma_{\bm(j)}$. Since, for each $j$,  there are $r$ such functions, we deduce that
  $\hat u_n$ is in $\Sigma_{rn}$.

Writing $\bar u_n=\sum_{\nu}\lambda_\nu^{-1}a_\nu e_\nu$ for a suitable sequence $(a_\nu)_{\nu\in \bigcup_{j=1}^k\Lambda(j)}$, we
have, by definition, that
$$
\hat u_n=\sum_{\nu}S_r(\lambda_\nu)a_\nu e_\nu.
$$
We can now bound $\|\bar u_n-\hat u_n\|_{t}$ as follows.
With $A_0:=16/\underline\lambda$, we obtain
\begin{eqnarray}
\label{diff1-0}
\|\bar u_n-\hat u_n\|^2_{{t+2}}&=&\|\sum_{\nu} [\lambda_\nu^{-1}-S_r(\lambda_\nu)]a_\nu e_\nu\|^2_{{t+2}}
= \sum_{\nu} \lambda_\nu [[\lambda_\nu^{-1}-S_r(\lambda_\nu)]a_\nu]^2
\nonumber\\
&\le&  [A_0{\rm e}^{-\pi\sqrt{r}}]^2\sum_\nu\lambda_\nu a_\nu^2
  =[A_0{\rm e}^{-\pi\sqrt{r}}]^2\,\|\bar u_n\|^2_{{t+2}}
  =[A_0{\rm e}^{-\pi\sqrt{r}}]^2\,\|g_n\|^2_{{t}}\nonumber\\
 &\le&
 [A_0{\rm e}^{-\pi\sqrt{r}}]^2\,\|f\|^2_{{t}}.
 \end{eqnarray}
 We now take $r$ as the smallest integer for which  $ \pi \sqrt{r}\ge  \alpha \log n$.   For this choice of $r$, we obtain
 \be
 \label{approxbargn}
 \|\bar u_n-\hat u_n\|_{{t+2}}\le A_0\|f\|_{{t}}n^{-\alpha}.
 \ee
 When this is combined with \eref{ae12} by a triangle inequality in the $\HH^{t+2}$ norm, we obtain
 \be
 \label{au}
 \|u-\hat u_n\|_{{t+2}}\le Mn^{-\alpha}+ A_0\|f\|_{{-t}}n^{-\alpha} {\le} (A_0+1)\|f\|_{\cA^\alpha((\Sigma_n)_{n \geq 0},\HH^{t})}n^{-\alpha}.
 \ee
 Since $\hat u_n$ is in $\Sigma_{rn}$ and $r \sim C_\alpha [\log n]^2$, where $C_\alpha$ is a positive constant
  dependent solely on $\alpha$, by absorbing the constants $A_0+1$ and $C_\alpha$ into the definition of the constant $A=A(\alpha,\underline\lambda)$,
  we deduce \eqref{urate}. That completes the proof.
 \end{proof}

\smallskip

Suppose that in place of $\Sigma_n$ we use the spaces $\Sigma_n(\cR)$ with restrictions on the indices of the $e_\nu$
as described in \S\ref{ss:restricted}.  The proof of the last theorem then applies verbatim with these restrictions.  Hence, we have the following result.
 \begin{theorem}\label{r1rth}
 Suppose that $t\in\R$ and  $f\in\cA^\alpha(\Sigma_n(\cR))_{n \geq 0},\HH^{t})$, $\alpha>0$;  then, the corresponding  solution $u\in \HH^{t+2}$ to {\rm \eref{varprob1}} satisfies the following inequality:
 \be
 \label{urate-2}
  \sigma_{n[\log n]^2}(u)_{\HH^{t+2}}\le A\|f\|_{\cA^\alpha((\Sigma_n(\R))_{n \geq 0},\HH^{t})}n^{-\alpha},\quad n=1,2,\dots,
  \ee
  where the constant $A$ depends only on $\alpha$ and the smallest eigenvalue $\underline\lambda>0$ of $\cB$.
    In particular, $u$ belongs to
$\cA^{\alpha'}((\Sigma_n(\cR))_{n \geq 0},\HH^{t+2})$ for all $\alpha' \in (0,\alpha)$.
\end{theorem}

We have formulated the above results in terms of the approximation spaces $\cA^\alpha$.  This only gives information for
polynomial order decay.   We can work in more generality.  Let $\gamma$  be a strictly monotonically increasing function defined on $[0,\infty)$, with $\gamma(0)> 0$.   We   define the approximation class $\cA^\gamma((\Sigma_n)_{n \geq 0},\HH^t)$ as the set of all $v\in \HH^t$ such that
\be
\label{sigma-n}
 \gamma(n)\,\sigma_n(f)_{\HH^t}\le M,\quad n\ge 0,
\ee
with (quasi-) norm defined as the smallest such $M$.   We then have the following theorem.

\begin{theorem}\label{r2th}  Suppose  that $t\in\R$ and that $f\in\cA^\gamma((\Sigma_n)_{n \geq 0},\HH^{t})$;  then, the corresponding  solution $u \in \HH^{t+2}$ to {\rm \eref{varprob1}} satisfies the following inequality:
 \be
 \label{urate2}
  \sigma_{n[\log \gamma(n)]^{2}}(u)_{\HH^{t+2}}\le A\|f\|_{\cA^\gamma((\Sigma_n(\R))_{n \geq 0},\HH^{t})}[\gamma(n)]^{-1},\quad n=1,2,\dots,
  \ee
  where the constant $A$ depends only on   the smallest eigenvalue $\underline\lambda>0$ of $\cB$.
    In particular, $u$ belongs to
$\cA^{\bar \gamma}((\Sigma_n)_{n \geq 0},\HH^{t+2})$, where  $\bar \gamma(x) := \gamma(G^{-1}(x))$  
 with $G(x):=x[\log \gamma(x)]^2$, $x > 0$.
\end{theorem}
\begin{proof}  The proof is in essence the same as that of Theorem \ref{r1th} with $n^{\alpha}$ replaced by $\gamma(n)$,
and $r \geq 1$ chosen so that $\pi \sqrt{r}\ge \log \gamma(n)$.  This gives \eref{urate2} from which the remaining claim easily follows
upon observing that \eref{urate2} just means $\sigma_{G(n)}(u)_{\HH^{t+2}}\le A\|f\|_{\cA^\gamma((\Sigma_n(\R))_{n \geq 0},\HH^{t})}[\gamma(n)]^{-1}$,
and putting $m=G(n), n=G^{-1}(m)$.
 \end{proof}

\begin{remark}
\label{rem:curse}  Let us remark on how the curse of dimensionality enters into  Theorem \ref{r2th}.  This theorem  says that for $f\in\cA^\gamma((\Sigma_n)_{n \geq 0},\HH^{t})$
the number of degrees of freedom needed to approximate the solution $u$ in $\HH^{t+2}$ to within accuracy $\ve$ is of the order
$\bar\gamma^{-1}(A\|f\|_{\cA^\gamma}/\ve)$ where $A$ is a fixed constant independent of the spatial dimension $d$.
Moreover, the approximations to $u$ are derived from nonlinear information on the data $f$ given in terms of low-rank approximations.
Hence, under this assumption on the data the curse of dimensionality is broken. However, the regularity assumption on
$f$ depends on $d$ because the larger $d$ the fewer degrees of freedom can be spent on each tensor factor.
 However, as addressed in more detail later, when the tensor factors approximate functions having a certain
(low-dimensional) Sobolev or Besov regularity the deterioration of accuracy when $d$ grows is expected to be at most algebraic in $d$
so that one can, in principle, still assert {\em tractability} in the sense that the computational work needed \rd{in order} to realize a given target accuracy
does not depend exponentially on $d$.
\end{remark}

%%%%%%%%%%%%%%%%%%%%%%%%%%%%%%%%%%%%%%%%%
\section{A Basis-Free Tensor-Sparsity Model}
\label{s:numerical}
\newcommand{\Ten}{\Theta}
%%%%%%%%%%%%%%%%%%%%%%%%%%%%%%%%%%%%%%%%%%%%%%%%%%%
 The primary purpose of the sparsity model considered \rd{thus} far is to highlight in a technically simple fashion, on the one hand,
the principal mechanism of how tensor-sparsity of the data is inherited by the solution and, on the other hand,
 the relevance of the structure of the spectrum  of elliptic operators in this context.
 While this is an interesting theoretical result, it does not lead directly to a numerical method \rd{that}
exploits this compressibility in an efficient numerical algorithm.  There are several reasons for this.

   The first of  these is the fact that, in general, the
  eigenfunction basis is not available to us, and computing even the first few eigenfunctions will generally require significant computational effort.  Moreover, even, if we had computed the eigenfunction basis
  and the representation of $f$ in this basis, it is not a simple task to find a good low-rank approximation $g$ to $f$.

   A second point is that  in the classical setting of differential operators, the assumption that
$f$ is compressible in the eigenfunction basis may in fact exclude rather simple functions.    Consider, for example,  the Poisson equation subject to a homogeneous Dirichlet boundary condition:
$$
-\Delta u = 1 \quad \mbox{in} \quad  D = (0,1)^d,\quad u=0\quad\mbox{on}\quad \partial D.
$$
In this case one has $e_\nu(x) = 2^{-d/2} \prod_{j=1}^d \sin(\pi \nu_j x_j)$, $\nu\in \N^d$. Note that the right-hand side $f(x)\equiv 1$
is a {\em rank-one} function. However, its expansion in terms of the $e_\nu$ has infinitely many terms. Because of the boundary conditions,
parametrizing the solution and the right-hand side with respect to the same basis may not be the most favorable approach. In fact, each
univariate factor $1$ of $f$ has an infinite expansion in the univariate basis $(\sin (\pi k x_j))_{k\in\N}$.
Thus we shall next propose a way \rd{of describing} tensor-sparsity without referring to any particular background basis
 and where the corresponding regularity notion does not become stronger when $d$ increases.

Perhaps, the main shortcoming of the idea of sparsity introduced in the previous sections is that  the tensors that arise in the  approximation lack stability.   Indeed, we have no control on how well the approximants $g$ can themselves be numerically resolved.  Typically, the computability of $g$ is related to its regularity, for example, in the family of spaces $\HH^s$.

The purpose of this section is to remedy some of these deficiencies.   This will be accomplished by putting forward a form of
tensor sparsity and compressibility that, on the one hand, does not depend on the eigenfunction basis and, in fact, is independent of any basis, and on the other hand, it imposes a regularity on the approximates $g$ to $f$, thereby making them computable.  Moreover,
as we shall show, this  leads to an efficient numerical implementation.  To accomplish this, we will impose a certain stability for the approximating tensors.  To understand the role of stability, let us begin with some general comments about tensor approximation.

Trying to approximate a function $f$ in a Banach space $X$ in terms of (possibly short) sums of rank-one functions, which are merely required to belong to the same space $X$, is an ill-posed problem. This is well-known even in the discrete (finite-dimensional) case,
see \cite{BiCaLoRo1979,LiSi2008}. For instance, the limit of a sequence of rank-two tensors may have rank three. It is also known (see \cite{KrDiSt2008})
that the elements of such sequences degenerate in the sense that the sums have uniformly bounded norms but the individual summands do not.
This is why one resorts (predominantly in the discrete setting) to  {\em subspace-based} tensor formats
that   are inherently more stable, although they are
also more involved technically \cite{HaKu2009,Gr2010}.
Note, however, that even if one managed to stably approximate $f$ in $X$   by a short sum of {\em arbitrary} rank-one tensors in $X$, in the continuous (infinite-dimensional) setting,
these rank-one functions could be arbitrarily hard to compute.
Thus, the rank-one functions that are allowed to enter into a competition for best tensor approximations to a given function $v$ need to be
{\em tamed} in one way or another.   The standard way of accomplishing this and the one we shall follow below is to
add a penalty (regularization) term to the approximation error.

%%%%%%%%%%%%%%%%%%%%%%%%%%%%

%%%%%%%%%%%%%%%%%%%%%%%%%%%%
\subsection{An alternative sparsity  model}\label{ssec:alter2}
%%%%%%%%%%%%%%%%%%%%%%%
 To understand the form   a penalty or regularization   should take, let us
assume that we wish to approximate $v$ in \rd{the} $\HH^t$ norm.     We denote by $\cT_r(\HH^t)$ the set of all rank $r$ tensors
\be
\label{rrt2}
g=\sum_{k=1}^rg^{(k)},\quad g^{(k)}=g_1^{(k)}\otimes\cdots\otimes g_d^{(k)},\quad g\in\HH^t,
\ee
and $\cT(\HH^t):= \bigcup_{r\in \N_0}\cT_r(\HH^t)$.
 When approximating $v$ by the elements from ${\cT_r(\HH^t)}$, the approximation error is not the only issue.   We also want that the approximant  $g$ itself to be approximable (also in $\HH^t$).   To guarantee this,  we will  impose regularity constraints on $g$ and therefore
its components $g_j^{(k)}$ (see Lemma \ref{lem:H-1}).   To keep matters simple, we take this regularity to be membership in the space $\HH^{t+\zeta}$ for some arbitrary but fixed $\zeta>0$.  As will be seen below, in the present particular context of solving \eref{varprob-2}, there is a natural choice of $\zeta$.
 Note that one could also assume other types of regularity
measures  such as Besov regularity in the classical PDE setting.

We consider for any $r\in \N_0$, $v\in \HH^t$, $\zeta >0$, the $K$-functional
\beqn
\label{K-funct}
K_r(v,\mu):=K_r(v,\mu;\HH^t,\HH^{t+\zeta}):= \inf_{g\in \cT_r(\HH^{t+\zeta})}\big\{\|v - g\|_t + \mu \tripnorm{g}_{r,t+\zeta}\big\},
\eeqn
where, for any $g\in\cT_r(\HH^s)$ of the form \eref{rrt2},
\beqn
\label{tripnorm}
\tripnorm{g}_{r,s} := \inf \left[\max\,\{\|g\|_s, \|g^{(k)}\|_s : k=1,\ldots, r\}\right],
\eeqn
where the infimum is taken over all representations of $g$,  of the form \eref{rrt2}, using $r$ rank-one tensors.
We infer from Lemma \ref{lem:compact}   that the infimum in \eref{K-funct} is attained for some $g\in \cT_r(\HH^{t+\zeta})$.
Let us make some comments that will illuminate the role of $K_r$.   This functional measures how well $v$ can be approximated by rank $r$ tensors (in the norm of $\HH^t$) with the added constraint that the \rd{approximants} $g$ have bounded  $\HH^{t+\zeta}$ norm.  One question that arises is how should $\zeta$ be chosen.  In the context of solving PDEs, to  approximate the solution $u$ in a specified $\HH^t$ norm with a certain rate requires some excess regularity of $u$.  This excess regularity is quantified by $\zeta$.  We guarantee this excess regularity by assuming some regularity for $f$ and then applying a regularity theorem, {see \eref{isoext}}.  A second remark is that if we know that the function $f$ in \eref{K-funct} is indeed
in $\HH^{t+\zeta}$, {when $\HH^{t+\zeta}$ agrees with a classical Sobolev space,} then standard constructions of approximants to $v$ would also be in $\HH^{t+\zeta}$ and have norm in this space that does not exceed $C\|v\|_{\HH^{t+\zeta}}$.  However, we do not wish to assume \rd{that} we know $\|v\|_{\HH^{t+\zeta}}$ or $C$,
which is why we resort to a general $\mu$.

Since $K_r(v,\mu)$ is similar to an error functional, we can obtain approximation classes in the same way we have defined the classes $\cA^\alpha$ (which give error decay rate $O(n^{-\alpha})$)  and the more general classes $\cA^\gamma$ (which give error decay rate $1/\gamma(n)$).  We describe this in the case of general $\gamma$ since it  subsumes the special case of polynomial growth rates.   We begin with any  {\em growth sequence} $\{\gamma(r)\}_{r\in \N_0}$, $\gamma(0)=1$, that monotonically increases to infinity.  Given such a $\gamma$, we define
\beqn
\label{approx-norm}
\| v\|_{\bar \cA^\gamma(\HH^t,\HH^{t+\zeta}) }:= \sup_{r\in \N_0}\gamma(r)\, K_r(v, 1/\gamma(r);\HH^t,\HH^{t+\zeta})
\eeqn
and the associated approximation class
\beqn
\label{appr-class2}
\bar \cA^\gamma:= \bar\cA^\gamma(\HH^t,\HH^{t+\zeta}):= \{v\in \HH^t: \| v\|_{\bar\cA^\gamma}<\infty\}.
\eeqn
We shall frequently use that, whenever  $v\in \bar\cA^\gamma(\HH^t,\HH^{t+\zeta})$, there exists for each $r\in \N_0$ a $g_r\in \cT_r(\HH^t)$ such that
\beqn
\label{consequences}
\|v - g_r\|_t \leq \gamma(r)^{-1} \| v\|_{\bar \cA^\gamma(\HH^t,\HH^{t+\zeta})},\quad \|g_r\|_{t+\zeta} \leq \tripnorm{g_r}_{r,t+\zeta} \leq \| v\|_{\bar \cA^\gamma(\HH^t,\HH^{t+\zeta})}.
\eeqn
 In other words, the approximants $g_r$ not only provide the approximation rate $1/\gamma(r)$ by rank $r$ tensors but also provide a control on their regularity.

Of particular interest are  sequences $(\gamma(n))_{n \geq 0}$ that {\em have a rapid growth}, because this reflects the closeness of $\HH^s$-stable approximations of rank-one summands to $v$.
For sequences $\gamma$ that grow faster  than any polynomial rate and hence   violate the requirement that
\beqn
\label{pol}
\gamma({2n})\leq C\gamma(n),\quad n\in \N,
\eeqn
 the corresponding approximation classes are no longer linear.
Nevertheless, when $\gamma(n)= {\rm e}^{\alpha n}$ for instance, then the sum of any two elements is still in $\bar\cA^{\hat\gamma}$ where $\hat\gamma(n)= {\rm e}^{\frac{\alpha}2 n}$.\\

%%%%%%%%%%%%%%%%%%%
\subsection{Approximations to \texorpdfstring{$\cB^{-1}$}{inv(B)}}\label{sec:main-a}
%%%%%%%%%%%%%%%%%%%

As in the case of Theorem \ref{r1th}, the key vehicle for proving regularity theorems for the new approximation spaces  are {\em exponential sum} approximations.  While these have been used above for the approximation of the scalars $\lambda_\nu^{-1}$,
we shall now employ them to construct {\em approximate inverses} to $\cB^{-1}$, again based
on Lemma \ref{explemma}.  In contrast with the approximate inversion of matrices (see \cite{Gr2004,Kh2008}), in the
infinite-dimensional setting considered here, we shall take into account the mapping properties of the operator on
the scale of generalized Sobolev spaces $\HH^t$ with $t \in \mathbb{R}$.

In order to apply Lemma \ref{explemma} for approximating  the solution of $\cB \uu = f$, we recall that, by  \eref{Fourier2},
$$
u= \sum_{\nu\in\N^d}\lambda_\nu^{-1}\dup{f,e_\nu}e_\nu.
$$
Note that by  \eref{eigenexp},
\beqn
\label{sr}
\lambda_\nu^{-1}e_\nu \approx S_r(\lambda_\nu)\,e_\nu = \sum_{k=1}^r \rd{\omega_{r,k}}\,{\rm e}^{-\rd{\alpha_{r,k}}\lambda_\nu}\,e_\nu =\sum_{k=1}^r \rd{\omega_{r,k}}\,{\rm e}^{-\rd{\alpha_{r,k}}\cB}\,e_\nu.
\eeqn
Hence, formally,
an approximation to $\cB^{-1}f$ is
given by
 \beqn
\label{expD}
S_r(\cB)f := \sum_{k=1}^r \rd{\omega_{r,k}}\,{\rm e}^{-\rd{\alpha_{r,k}}\cB}f.
 \eeqn
The following proposition establishes the mapping properties of the operators $S_r(\cB)$ and estimates for how well
this operator approximates $\cB^{-1}$.
\begin{proposition}
\label{thm3.1}
Let $C_0:=\max\{8,\frac{2}{\underline \lambda}\}$, $t\in\R$, and $v\in\HH^t$.

\noindent
{\rm (i)}  If $t\le s\le t+2$,
then,
\beqn
\label{norms*}
\|\cB^{-1} -  S_r(\cB)\|_{\HH^{t}\to \HH^s} \leq C_0  {\rm e}^{-\frac{(2-(s-t))\pi}{2}\sqrt{r}}.
\eeqn

\noindent
{\rm (ii)}
In particular, for any   $\xi \in [0,2]$, one has
 \beqn
\label{sterror}
\|{\cB}^{-1}v - S_r(\cB)v\|_{t+\xi}
\leq C_0 {\rm e}^{-\frac{(2-\xi)\pi}{2}\sqrt{r}}\|v\|_{{t}}.
\eeqn

\noindent
{\rm (iii)}
Moreover,
\beqn
\label{sDbound}
\|S_r(\cB)v\|_{t+2}\leq \left(C_0+1\right) \|v\|_{t}.  
\eeqn

\noindent
{\rm (iv)}  If $v$ is a rank-one tensor, then $S_r(\cB)v = \sum_{k=1}^r v^{(k)}$, where each $v^{(k)}:= \rd{\omega_{r,k}}\,{\rm e}^{-\rd{\alpha_{r,k}}\cB}v$ is a rank-one tensor, which satisfies
\be
\label{ronenorm}
\|v^{(k)}\|_{t+2} \le   {(C_0+1)^{1/2}}\|v\|_{t}.
\ee

\end{proposition}
\begin{proof}
(i) Defining
$$
\ve_{r,\nu}:= \lambda_\nu^{-1}- S_r(\lambda_\nu),\quad \nu\in \N^d,
$$
we know from \eref{expfa} that
\be
\label{inverseapprox}
|\ve_{r,\nu}|\le \frac{16}{\underline\lambda}{\rm e}^{-\pi\sqrt{r}},\quad \nu\in\N^d.
\ee
Now, we can write
\begin{eqnarray}
\label{1error}
{\cB}^{-1}v &=&   \sum_{\nu\in\N^d}\Mdual{ v,e_\nu }\big(S_r(\lambda_\nu) + \ve_{r,\nu}\big) e_\nu  
=  S_r(\cB)v +
  \sum_{\nu\in\N^d}\ve_{r,\nu} \Mdual{ v,e_\nu }e_\nu .  
\end{eqnarray}
Hence,
\begin{eqnarray}
\label{Hsest}
\|{\cB}^{-1}v - S_r(\cB)v\|_{s}&=&\left(\sum_{\nu\in\N^d}\ve_{r,\nu}^2 \lambda_\nu^s |\langle{v, e_\nu}\rangle|^2\right)^{1/2}\nonumber \\
&\le& \left(\sum_{\nu\in\N^d}  \lambda_\nu^t |\langle{v, e_\nu}\rangle|^2\right)^{1/2}\sup_{\nu\in \N^d}|\ve_{r,\nu}| \lambda_\nu^{\frac{s-t}{2}}\nonumber \\
& \le& \|v\|_{t}\sup_{\nu\in \N^d}|\ve_{r,\nu}| \lambda_\nu^{\frac{s-t}{2}}.
\end{eqnarray}
We write $s-t=2-\xi $  with  $\xi\in [0,2]$ and
rewrite $\ve_{r,\nu}^2 \lambda_\nu^s = \ve_{r,\nu}^2 \lambda_\nu^{2-\xi}\lambda_\nu^{t}$.   Now,
\beqn
\label{twocases}
\ve_{r,\nu}^2\lambda_\nu^{s-t}=\ve_{r,\nu}^2\lambda_\nu^{2-\xi}=\left\{
\begin{array}{ll}
(\ve_{r,\nu}\lambda_\nu)^2\lambda_\nu^{-\xi}\leq 8^{\xi}{\rm e}^{-\xi\pi\sqrt{r}},& \lambda_\nu > \frac 18 {\rm e}^{\pi\sqrt{r}},\\
(\ve_{r,\nu}\lambda_\nu)^{2-\xi} \ve_{r,\nu}^{\xi}\leq [\frac{2}{\underline \lambda}]^{2-\xi} {\rm e}^{-\xi\pi\sqrt{r}}, & \lambda_\nu \leq \frac 18 {\rm e}^{\pi\sqrt{r}},
\end{array}\right.
\eeqn
 where in the first estimate we used the fact that $\ve_{r,\nu}\lambda_\nu\le 1$ because of \eref{expfa1}.  The second estimate used \eref{inverseapprox}.   So, \eref{norms*} follows from \eref{Hsest}.   This proves (i).

\smallskip

 \noindent
 (ii) The assertion \eref{sterror} is just a restatement of  \eref{norms*}.

\smallskip

\noindent
 {(iii)}   The bound \eref{sDbound}   follows  from the fact that
$ \|\cB^{-1}\|_{\HH^t\to\HH^{t+2}}\le 1$.

\smallskip

\noindent
(iv)  Each of the components $v^{(k)}$ satisfies
\begin{eqnarray*}
 \|v^{(k)}\|_{t+2}^2 &=&  
\sum_{\nu\in\N^d}\lambda_\nu^{t+2}|\langle v^{(k)},e_\nu\rr|^2 
\\
& = &
 \rd{\omega_{r,k}^2} \sum_{\nu\in\N^d} \lambda_\nu^{t} (\lambda_\nu  {\rm e}^{-\rd{\alpha_{r,k}}\lambda_\nu})^2|\langle v ,e_\nu\rr|^2 \\
 &\le & \sup_{\nu\in\N^d} (\rd{\omega_{r,k}}\lambda_\nu  {\rm e}^{-\rd{\alpha_{r,k}}\lambda_\nu})^2\|v\|_{t}^2\leq  {(1+C_0)}\|v\|_{t}^2,
\end{eqnarray*}
where we \rd{have} used \eref{eigenexp} in the first equality.  In the last inequality, we use\rd{d} {that, by   \eref{expfa},
$xS_r(x)\leq 1+ \frac{2}{\underline\lambda}$ for $x\leq \frac 18 {\rm e}^{\pi\sqrt{r}}$ while, by \eref{expfa1},
$xS_r(x)\leq 1$ for $x\geq  \frac 18 {\rm e}^{\pi\sqrt{r}}$}
so that, for any $\nu\in\N^d$,
\be
\label{used}
\rd{\omega_{r,k}}\lambda_\nu  {\rm e}^{-\rd{\alpha_{r,k}}\lambda_\nu}\le \lambda_\nu  \sum_{k=1}^r \rd{\omega_{r,k}}{\rm e}^{-\rd{\alpha_{r,k}}\lambda_\nu }=\lambda_\nu S_r(\lambda_\nu) \le {1 +\frac{2}{\underline \lambda} \leq (1+C_0). } 
  \ee
 \end{proof}

\smallskip

Note that $({\cB}^{-1} -S_r(\cB))_{r\in \mathbb{N}}$ as a sequence of operators from $\HH^{t}$ to $\HH^s$ tends to zero as $r\to \infty$  in the corresponding operator norm as long as $s-t <2$, while the sequence is merely uniformly bounded when $s-t=2$.
\begin{remark}
\label{rem:remainvalid}
The statements of Proposition \ref{thm3.1} remain valid for $s<t$ and hence $\xi >2$ where, however, the constant $C_0$
depends then on $\xi = 2-(s-t)$, as can be seen from \eref{twocases}. Since Proposition \ref{thm3.1} will be applied later for $s\ge t$
we are content with the above formulation to avoid further dependencies of constants.
\end{remark}

%%%%%%%%%%%%%%%%%%%%%%%%%%%
\subsection{A new regularity theorem}
\label{ss:newregularity}
We can now state our new regularity theorem.

\begin{theorem}
\label{thm:mainII}   
Assume  that $f\in \bar\cA^\gamma(\HH^t,\HH^{t+\zeta})$ for the specified $t\in \R$ and  some $0<\zeta \le 2$. Let
\be
\label{defR}
R(r,\gamma):= \left\lceil C_1(\zeta)  
\big(\log(C_0\gamma(r))\big)^2\right\rceil,
\ee
where $C_1(\zeta):= \frac{4}{(\pi\zeta)^2}$ and $C_0$ is the constant in Proposition \ref{thm3.1},
and define the tempered sequence
\beqn
\label{gammahat}
\hat\gamma(m):= \gamma(r),\quad rR(r,\gamma)\le m<(r+1)R(r+1,\gamma), \quad r=1,2,\dots .
\eeqn

Then, the solution $u$  to the variational problem \eref{varprob-2}
 belongs to $\bar\cA^{\hat\gamma}(\HH^{t+2},\HH^{t+2+\zeta})$ and
\beqn
\label{hatgamma}
\| u\|_{\bar\cA^{\bar \gamma}(\HH^{t+2},\HH^{t+2+\zeta})}\leq  (3+C_0)\|f\|_{\bar \cA^\gamma(\HH^t,\HH^{t+\zeta})},
 \eeqn

Moreover, the mapping that takes   the data $f$ into a rank-$r$ approximation to
$u$, realizing the rate $\hat \gamma$, is continuous.
\end{theorem}
\begin{proof}
 Suppose that  {$f\in \bar \cA^\gamma:=\bar\cA^\gamma(\HH^t,\HH^{t+\zeta})$}, so that, by \eref{consequences},  there exists, for
each $r\in\N$,  a $g_r =\sum_{k=1}^{r }g^{(k)}_r$, where $g^{(k)}_r$ is of the form \eref{rrt2},
such that
\beqn
\label{a-r}
\|f- g_r\|_{t} \leq \gamma(r)^{-1}\|f\|_{\bar\cA^\gamma},  
\eeqn
and  
\beqn
\label{grbound}
\|g_r\|_{t+\zeta}\leq  \tripnorm{g_r}_{r,t+\zeta} \leq \|f\|_{\bar\cA^\gamma} .  
\eeqn
From the definition of  $R=R(r,\gamma)$ and \eref{sterror},
\be
\label{defR1}
\|\mathfrak{B}^{-1}g_r - S_R(\mathfrak{B})g_r\|_{t+2}\leq \gamma(r)^{-1}\|g_r\|_{t+\zeta}.
\ee
 Now, we define
\beqn
\label{Teta-II}
 \bar u   := S_{R}(\mathfrak{B})g_r   .  
\eeqn
The bound \eref{defR} gives
\begin{eqnarray}
\label{ubar}
\| u - \bar u \|_{t+2}&=& \|\cB^{-1}(f- \cB \bar u )\|_{t+2} = \|f- \cB \bar u \|_{t}
\leq  \|f- g_r\|_{t} + \|g_r - \cB S_{R }(\mathfrak{B})g_r \|_{t}\nonumber\\
&=&  \|f- g_r\|_{t}  + \|(\cB^{-1}  - S_{R }(\mathfrak{B}))g_r \|_{t+2}\nonumber\\
&=&   \|f- g_r\|_{t} + \gamma(r)^{-1} \| g_r \|_{t+\zeta}\nonumber\\
  & \leq & 2\gamma(r)^{-1}\|f\|_{\bar \cA^\gamma(\HH^t,\HH^{t+\zeta})},  
  \end{eqnarray}
where we have used \eref{a-r} and \eref{grbound} in the last step.  

By construction, the rank of $\bar u$ is bounded by $rR$.
 Moreover, expanding
\be
\label{SRB}
\bar u =S_R(\mathfrak{B})g_r = \sum_{k=1}^R\sum_{\ell=1}^r u_{k,\ell},\quad \mbox{where}\quad u_{k,\ell}:= \omega_{R,k} {\rm e}^{-\alpha_{R,k}\mathfrak{B}}g_r^{(\ell)},
\ee
we have from (iv)  of Proposition \ref{thm3.1} \rd{that}
\be
\label{tripbound-0}
\|u_{k,\ell}\|_{t+2+\zeta} \leq   {(1+ C_0)^{1/2}}\|g_r^{(\ell)}\|_{t+\zeta},\quad k=1,\dots R,\ \ell=1,\dots,r.
\ee
Combining this with \eref{sDbound}, we obtain
\beqn
\label{tripbound}
\tripnorm{\bar u}_{rR,t+2+\zeta} \leq (1+C_0)\tripnorm{g_r}_{r,t+\zeta}\leq (1+C_0)\|f\|_{\bar\cA^\gamma(\HH^t,\HH^{t+\zeta})}.
\eeqn

  The two inequalities \eref{ubar} and \eref{tripbound} allow us to conclude that, for any positive integer $m\in[rR(r,\gamma),(r+1)R(r+1,\gamma))$,
$$
\hat\gamma(m)K_{m}(u,  \hat\gamma(m)^{-1},\HH^{t+2},\HH^{t+2+\zeta}) \leq \gamma(r)\|u-\bar u\|_{t+2} + \tripnorm{\bar u}_{Rr,t+2+\zeta}
\leq (3+C_0)\|f\|_{\bar\cA^\gamma(\HH^t,\HH^{t+\zeta})}.
$$
Since this inequality covers all values of $m\in \N_0$,
this means that
\beqn
\label{u-2}
u\in \bar\cA^{\hat\gamma}(\HH^{t+2}, \HH^{t+2+\zeta}), \quad \| u\|_{\bar\cA^\gamma(\HH^{t+2},\HH^{t+2+\zeta})}\leq  (3+C_0)\|f\|_{\bar\cA^\gamma(\HH^t,\HH^{t+\zeta})}.
\eeqn

The asserted continuity of the mapping that takes $g_r$
into $\bar u$   follows from the boundedness of $S_{R}(\mathfrak{B})$ as a mapping from $\HH^{t+\zeta}$ to $\HH^{t+2}$,
$0\leq  \zeta\leq 2$; c.f. Proposition \ref{thm3.1}.
 \end{proof}  

Again, Theorem \ref{thm:mainII} remains valid for $\zeta >2$ with an adjusted constant $C_0$, see Remark \ref{rem:remainvalid}.
{Note also that in contrast with the previous models (see Remark \ref{rem:curse}),  the assumption $f\in \bar\cA^\gamma(\HH^{t},\HH^{t+\zeta})$
does not entail increasingly stronger   constraints on the approximants, and hence on $f$, when $d$ grows.}

The loss in the decay rate of Theorem \ref{thm:mainII} is essentially the same as in Theorem \ref{r2th}.
 For example, any algebraic convergence rate $\gamma(r)=r^\alpha$
is preserved up to a logarithmic factor. Perhaps more interesting in this model is the case of very fast decay rates,
as illustrated by the following example.

\begin{example}
\label{ex:exp}
If $\gamma(r) = {\rm e}^{\alpha r}$ for some $\alpha >0$, \rd{then} one has
\be
\label{gammahat-2}
\hat \gamma(r) \geq  \gamma({(r/C)^{1/3}}) = {\rm e}^{(\alpha r/C)^{1/3}},
 \ee
where $C=C(\alpha,f,\zeta)$.  Thus, on the one hand, the convergence rate for $u\in \bar\cA^{\hat\gamma}(\HH^{t+2},\HH^{t+2+\zeta})$  is still faster than any algebraic rate; on the other hand, in relative terms, the loss of tensor-sparsity is the larger the stronger the sparsity of the data. This is plausible since
even when $f$ is a rank-one tensor, $u$ will generally have infinite rank.

The proof of  \eref{gammahat} follows from the fact that
$R(r,\gamma) \approx r^2$,
with constants of equivalence depending only on $\zeta$.   Hence,
$rR(r,\gamma)\approx r^3$ and the result easily follows.
\end{example}

%%%%%%%%%%%%%%%%%%%%%%%%\smallskip
\section{{Complexity and Computational Algorithms}}\label{sec:complexity}
%%%%%%%%%%%%%%%%%%%%%%%%%%%%%%

While Theorem \ref{thm:mainII}  says that the solution $u$ to \eref{varprob1} can be well approximated
by sparse tensor sums with stable components,  whenever the right-hand side has a stable representation by sparse
tensor sums,
it does not  provide  an approximation  that is determined by a finite number of parameters nor does it offer
a numerical algorithm for computing a finitely parametrized approximation to $u$ that meets a prescribed target
accuracy.
Recall, for comparison,  that in low spatial dimensions,  classical numerical methods and their analysis show that the smoothness or regularity of the solution
to an elliptic problem determines the  complexity necessary to compute an approximation to any given target accuracy.

The central question addressed in this section is whether a regularity result like  Theorem \ref{thm:mainII}
can be translated into an analogous statement about computation even though we are now in a high-dimensional regime.   We address this question in two stages.
  We first show in \S \ref{ssec:rep-compl} that whenever $f$ is in one of the approximation classes $\bar\cA^\gamma(\HH^t,\HH^{t+\zeta})   $, for a $\gamma$ with at least power growth,  then the solution $u$ can be approximated in $\HH^t$ to within accuracy $\ve$ by a function that  is determined by $N(\ve,d)$ suitable parameters where
 \be
\label{poltract}
N(\ve,d) \leq d^{C_1} \ve^{-C_2},\quad \ve >0,\,\, d\in \N.
\ee
 In analogy with the terminology in {\em Information-Based Complexity}, we call such a result {\em representation tractability}.   In this sense Theorem \ref{thm:mainII} does establish a favorable relation between regularity and complexity.

Representation complexity bounds \rd{are, unfortunately, of} little practical relevance for computation, since they do  not provide in general a viable and implementable numerical algorithm.  In fact, the information used in deriving \eref{poltract} is the evaluation of exponential
maps ${\rm e}^{-\alpha\cB_j}$ applied to data approximations.  Unless one knows the eigenbases of the low-dimensional component operators $\cB_j$
this is a very restrictive assumption from a practical point of view, as mentioned earlier.  On the other hand,  the bounds can be viewed
as a {\em benchmark} for the performance of numerical schemes in   more realistic scenarios. If we are only allowed to query
the given data $f$ but cannot resort to an eigensystem  of $\cB$, as \rd{is the case in any} standard numerical framework for computing a solution to \eref{varprob1}, then  it is far less obvious how to arrive at a finitely parametrized approximation to $u$ and at what cost.
We refer to the corresponding computational cost as {\em numerical complexity}.  It addresses the computational complexity
of {\em approximately inverting} $\cB$ for data with certain structural properties.  This is   far less studied
in the context of complexity theory when $\cB$ cannot be diagonalized, although it is the central question in numerical computation.

We shall address  {numerical} complexity   in the remaining subsections of this section.
We emphasize already here that when treating computational complexity we assume that all relevant information
about the data $f$ is  available to us. In particular, we do {\em not} include the question of the cost
of providing approximations to $f$ in the desired format as described below. The morale of this
point of view is that ``data'' are part of the modeling process and are {\em given} by the user.
Even when the data take the simplest form, such as being a constant, conventional numerical tools would
render the solution of a high-dimensional diffusion problem, with certified accuracy in a relevant norm,
computationally intractable.  Our main contribution, in this direction, is therefore to show that, using
unconventional tools, as described in this paper, the problem is indeed numerically tractable at a computational cost not much higher than
\eref{poltract}.

\subsection{The right-hand side \texorpdfstring{$f$}{f}}
\label{ss:rs}
   We assume  that we are given an error tolerance $\ve>0$ that we wish to achieve with the numerical approximation.
Any numerical algorithm for solving \eref{varprob1}  begins with the input of the right-hand side $f$.  To exploit tensor sparsity, we need
that either $f$ is itself a rank $r$ tensor for some value of $r$ or it can be well approximated by a rank $r$ tensor.  Since the second case subsumes the first, we put ourselves into the second case.  Namely,  we know that for certain values of $\zeta$ and $\gamma$ we have 
 $f\in\bar\cA^\gamma(\HH^t,\HH^{t+\zeta})$.   We fix such  a value $\zeta\in (0,2)$ for the excess regularity of $f$.
As stressed earlier, we do not address the   problem of how one would create  a stable approximation to $f$, as is guaranteed by membership in this approximation class, but instead assume that
such an approximation to $f$ is  already given to us in this form.   We will comment later on the actual feasibility of such an assumption.

To avoid   additional technicalities, we will place very mild restrictions on the sequence $\gamma$.   We assume that $\gamma$ is strictly increasing  and has the following two properties:\\

\noindent
($\gamma1$): there exists a constant $\bar C$, depending on $\gamma$, such that
\be
\label{gamma1}
\gamma^{-1}(x)+1\le \gamma^{-1}(\bar C x),\quad x\geq 1,
\ee
where $\gamma^{-1}$ is the inverse function of $\gamma$, i.e., $\gamma^{-1}(\gamma(x))= x$.\\[3mm]
($\gamma2$): there exists a $\mu >0$ such that
\be
\label{gamma2}
x^\mu/\gamma(x) \leq C,\quad x\geq 1,
\ee
where $C$ is a constant.\\

Let us note that all polynomial growth sequences are admissible and  even
  sequences of the form $\gamma(r)={\rm e}^{cr^\beta}$, $c, \beta >0$ are included. Thus, these conditions are made
  only to exclude very fast and slow decaying sequences.  \rd{We note that the faster $\gamma$ increases the more stringent the first condition becomes}.   The second condition requires a minimum growth.
Of course, the type of tensor approximation discussed here is of primary interest when $\gamma$ increases rapidly, so \eref{gamma2} is not a serious restriction.

In summary, in the remainder of this paper, we make the following assumption:
\vskip .1in
\noindent
 {\bf (A1)}   {\it  We assume that we are given a sequence $\gamma$, satisfying {\rm ($\gamma1$)} and {\rm ($\gamma2$)},
a value $t\in\R$, and a value  $\zeta\in (0,2]$.  Whenever presented with an $r>0$ and  $f\in\bar\cA^\gamma(\HH^t,\HH^{t+\zeta})$,    we are given for free an approximation $g_r=\sum_{\ell=1}^r g^{(\ell)}$,
satisfying:
\be
\label{gr}
\gamma(r)\|f-g_r\|_{t}+ \tripnorm{g_r}_{r,t+\zeta} \le  A_1\|f\|_{\bar\cA^{\gamma}(\HH^t,\HH^{t+\zeta})},
\ee
with $A_1\ge 1$ an absolute constant.}
\vskip .1in
\noindent
Notice that the existence of such functions $g_r$ is guaranteed from the fact that $f\in\bar\cA^\gamma(\HH^t,\HH^{t+\zeta})$.

 %%%%%%%%%%%%%%%%%%%
\subsection{Representation complexity}\label{ssec:rep-compl}
%%%%%%%%%%%%%%%%%%%%%%

In this subsection, we prove that for any $\ve >0$,  the solution to \eref{varprob1} can be approximated to accuracy $\ve$  by   functions, which depend on a controllable number of parameters.
We fix any value of $t\in\R$ and $0<\zeta\le 2$ for this section.   In order to \rd{render} the presentation of the results of this subsection simple,
we will make an assumption on the \rd{growth} of the eigenvalues of $\cB$. This assumption will only be used in the present subsection and not in the following material, which introduces and analyzes numerical algorithms for solving \eref{varprob1}. The reader can easily verify  that
this assumption can be generalized in many ways but always at the expense of a more complicated statement of our results.
  \vskip .1in
  \noindent
 {\bf (RA)}   {\it  We assume that each of the the low-dimensional factor domains $D_j$ is a
 domain in $ \R^p$ with the same $p\in\{1,2,3\}$, and that the eigenvalues of $\cB$ satisfy, for some fixed $\beta>0$},
  \be
\label{lowlambda}
\lambda_{m}^* := \min_{j=1,\ldots,d}\lambda_{j,m} \geq \beta m^{2/p},\quad m\ge 1.
\ee
Classical results on eigenvalues for second order elliptic differential operators in low dimension establish {\bf (RA)} for a variety of settings,
see e.g. \cite{Grisvard,FS,FS:Arxiv}.  If the component operators
were of a different order, then the form of {\bf(RA)} would change but similar results could be obtained with obvious modifications.

  Since the following theorem is only of a theoretical flavor, we shall take $A_1=1$ in the assumption {\bf (A1)}.

\begin{theorem}
\label{thm:rep-compl}  Let $t\in\R$ and let $\zeta\in (0,2]$.
Assume that  {\bf (A1)} holds {\rm (}with constant $A_1=1${\rm )} for this $t$ and $\zeta$  and the assumption {\bf (RA)} also holds.
Then, for each $\ve >0$, there exists a function $v(\ve)$ with
\be
\label{uapprox}
v(\ve)\in \cT_{\bar r(\ve)}(\HH^{t+2+\zeta}), \quad \bar r(\ve)\leq 2 C_1(\zeta)\gamma^{-1}(4/\ve)C_1(\zeta)\big(\log(4C_0/\ve)\big)^2,
\ee
satisfying
\be
\label{approx-2}
\|u- v(\ve)\|_{t+2} \leq \ve \|f\|_{\bar\cA^\gamma(\HH^t,\HH^{t+\zeta})}
\ee
and
\be
\label{tripnorm1}
\tripnorm{v(\ve)}_{t+2+\zeta}\le (1+C_0)\|f\|_{\bar\cA^\gamma(\HH^t,\HH^{t+\zeta})}.
\ee
 Moreover, $v(\ve)$ is determined by
\be
\label{Ned}
N(\ve,d) \leq \bar B  d \ve^{-p/\zeta} \big(\gamma^{-1}(4/\ve)\big)^{1+p/\zeta}\big(\log(\bar C_0/\ve)\big)^{2+2p/\zeta}
\ee
parameters,
where the constant $\bar B$ depends on $\beta$ from {\rm \eref{lowlambda}}, on $C_0$ and $\zeta$ and where
$\bar C_0= 4C_0\bar C$  with $\bar C$ from {\rm \eref{gamma1}} and $C_0$ from Proposition \ref{thm3.1}.
\end{theorem}
\begin{proof} 
Given $\ve >0$,  we choose
$r=r(\ve)$ as the smallest integer such that $4\le \ve \gamma(r)$, which means that
 $$r(\ve) := \lceil \gamma^{-1}(4/\ve)\rceil.$$
  For this value of $r$ and this $\gamma$, we define $R$ by \eref{defR};
in other words, $R=R(\ve)$ is given by
$$R(\ve)= \big\lceil C_1(\zeta)\big(\log(4C_0/\ve)\big)^2\big\rceil.$$
We take $\bar u = S_R(\cB)g_r$ where $g_r$ is the approximation to $f$ asserted by {\bf (A1)}.
Arguing as in the derivation of  \eref{ubar}, we find that
\be
\label{know1}
\|u-\bar u\|_{t+2} \leq 2 \gamma(r)^{-1}\|f\|_{\bar\cA^\gamma(\HH^t,\HH^{t+\zeta})}\le \frac{\ve}{2} \|f\|_{\bar\cA^\gamma(\HH^t,\HH^{t+\zeta})}.
\ee
 From \eref{SRB}, we have  that
$$
\bar u(\ve)= \sum_{k=1}^{R(\ve)}\sum_{\ell=1}^{r(\ve)} u^{k,\ell}.
$$
Using the bounds \eref{tripbound-0} for  
the  rank-one terms
$$
u^{k,\ell}= \bigotimes_{j=1}^d u^{k,\ell}_j,
$$
one obtains, as in   \eref{tripbound}, that
\be
\label{tnorm}
\tripnorm{\bar u(\epsilon)}_{rR,t+2+\zeta} \leq (1+C_0)\|f\|_{\bar\cA^\gamma(\HH^t,\HH^{t+\zeta})}.
\ee
Now we invoke Lemma \ref{soblemma}, which provides for each $u_{k,\ell}$ a rank-one
approximation
$$
u^{k,\ell,m} = \bigotimes_{j=1}^d \Big(\sum_{\nu =1}^m \langle u^{k,\ell}_j,e_{j,\nu}\rr e_{j,\nu}\Big),
$$
satisfying
$$
\| u^{k,\ell} - u^{k,\ell,m}\|_{t+2}\leq (\lambda_{m+1}^*)^{-\zeta/2} \|u^{k,\ell}\|_{t+2+\zeta}\leq (1+ C_0)(\lambda_{m+1}^*)^{-\zeta/2}
\|f\|_{\bar\cA^\gamma(\HH^t,\HH^{t+\zeta})}.
$$
Hence, defining
$$
u_m:= \sum_{k=1}^{R(\ve)}\sum_{\ell=1}^{r(\ve)} u^{k,\ell,m},$$
we have that $u_m \in \cT_{rR}(\HH^{t+2+\zeta})$
and
satisfies
$$
\|u - u_m\|_{t+2} \leq \Big(\frac{\ve}{2} + (1+ C_0)(\lambda_{m+1}^*)^{-\zeta/2}r(\ve)R(\ve) \Big)
\|f\|_{\bar\cA^\gamma(\HH^t,\HH^{t+\zeta})}.
$$
We now choose $m=m(\ve)$, as the smallest integer, such that
\be
\label{smallenough}
(1+ C_0)(\lambda_{m+1}^*)^{-\zeta/2}r(\ve)R(\ve)\leq \frac{\ve}2.
\ee
If we define $v(\ve):=u_{m(\ve)}$, then the triangle inequality shows
that \eref{approx-2} holds.  The bound \eref{tripnorm1} follows from \eref{tnorm}
since each component $u^{k,\ell,m}$, from its very definition, has smaller $\HH^{t+2+\zeta}$ norm than that of
$u^{k,\ell}$.

We now check the complexity of $v(\ve)$.
By \eref{gamma1}, we have $r(\ve)= \lceil \gamma^{-1}(4/\ve)\rceil\leq \gamma^{-1}(\bar C 4/\ve)$ so that
\begin{eqnarray*}
r(\ve)R(\ve) &= &   \gamma^{-1}(4\bar C/\ve)  \, \big\lceil C_1(\zeta)\big(\log(4\bar C C_0/\ve))\big)^2\big\rceil . 
\end{eqnarray*}
Under the assumption \eref{lowlambda} the condition \eref{smallenough} is indeed satisfied provided that
\be
\label{m}
m(\ve)\ge B \ve^{-p/\zeta} \gamma^{-1}(4/\ve)^{p/\zeta}\big(\log (4\bar CC_0/\ve)\big)^{p/\zeta},
 \ee
where the constant $B$ depends on $\beta$ from \eref{lowlambda}, on $C_0,p$ and $\zeta$.
Since the number of parameters $N(\ve,d)$ needed to determine $v(\ve):= u_{m(\ve)}$ is at most
$r(\ve)R(\ve) m(\ve)d$ the assertion \eref{Ned} follows from the above choices of $r(\ve), R(\ve)$.
\end{proof}

The above reasoning reveals that under the assumption {\bf (A1)} the representation complexity
of $f$ can be bounded by $C d \ve^{-p/\zeta} \big(\gamma^{-1}(4/\ve)\big)^{1+p/\zeta}$. Thus, up to a logarithmic factor,
the  inversion of $\cB$ on tensor-sparse data does not worsen the representation complexity as already hinted at by Theorem \ref{r1th}.

\subsection{Numerical complexity}
 The remaining subsections of this paper are devoted to the construction of a numerical algorithm based on
Theorem \ref{thm:mainII}, using only queries of the approximate data $g_r$.
This will provide, in particular,  a bound on the numerical complexity of
the problem \eref{varprob-2}. The scheme we propose in \S\ref{thealgorithm} does {\em not} require any knowledge about the eigenbases of the operators $\cB_j$
but is based on the approximate application of
the operator $S_R(\cB)$ to $g_r$.

We shall always assume the validity of {\bf (A1)}.
In what follows let $t<0$ and let $0<\zeta \leq 2$ stand for some {\em excess regularity}.
Given a target accuracy $\ve >0$ we wish to formulate  a numerically implementable scheme \rd{that delivers} an approximation $\bar u(\ve)$ to the
solution $u$ of \eref{varprob-2}
of possibly low rank, satisfying
\be
\label{u-target}
\|u- \bar u(\ve)\|_{t+2} \le \ve \|f\|_{\bar\cA^\gamma(\HH^t,\HH^{t+\zeta})}.
\ee

The proof of Theorem \ref{thm:mainII} will serve as the main orientation for the construction of the numerical scheme, however, with
the approximate inverse $S_R(\cB)$ of $\cB$ replaced by the variant
  $\bar S_R(\cB)$ defined in
Lemma \ref{explemma1}.
 We will need the following counterpart to Proposition \ref{thm3.1},  which   follows by combining Lemma \ref{explemma1} with the arguments used in the proof of Proposition \ref{thm3.1}.
 
 \begin{proposition}
\label{prop:Sbar}
Fix any $\ua <\pi$. For any $s >t$, $s-t' < 2$, there exists a constant $\bar C_0$ depending on $\ulambda$, $s-t'$, and $\ua$
such that
\be
\label{barSapprox}
\|\cB^{-1}v - \bar S_r(\cB)v\|_s \leq \bar C_0 {\rm e}^{-\frac{(2-(s-t))\ua}{2}\sqrt{r}}\|v\|_{t'},\quad v\in \HH^{t'}.
\ee
Moreover, $\bar C_0=\bar C_0(s-t')$ tends to infinity as $s-t'\to 2$.
\end{proposition}
\vskip .1in

 Now given any prescribed error tolerance $\ve$, we choose
\be
\label{repsilon}
r=r(\ve):= \lceil \gamma^{-1}(4A_1/\ve)\rceil,
\ee
where $A_1$ is the constant appearing in {\bf (A1)}.
 In other words, $r$ is  the smallest integer such that
\be
\label{chooser}
\gamma(r)\ge \frac{4A_1 
}{\ve}.
\ee
It follows from {\bf (A1)} that $g(\ve):=g_{r(\ve)}$ satisfies
\be
\label{follows2}
\|f-g(\ve)\|_{t}\le \frac{\ve}4\|f\|_{\bar\cA^\gamma}.
\ee
With an eye towards \eref{ubar} we want to choose   $R=R(\ve)$ such that
\be
\label{goal-barS}
\|(\cB^{-1}-\bar S_R(\cB))g_r\|_{t+2}\le  \frac{\ve}{4A_1}\|g_r\|_{t+\zeta}.
\ee
By Proposition \ref{prop:Sbar} and \eref{chooser},  this means that we want
$$
\bar C_0 {\rm e}^{-\frac{\ua \zeta}2\sqrt{R}}\leq (\gamma(r))^{-1}.
$$
A suitable choice is
\be
\label{Repsilon}
R(\ve):= \Big\lceil \bar C_1(\zeta)\big(\log\Big(\bar C_0 \gamma(r(\ve))\big)\Big)^2\Big\rceil, \quad \bar C_1(\zeta):= \frac{4}{(\ua \zeta)^2}.
\ee
Since, by \eref{gamma1}, $r(\ve) \leq \gamma^{-1}(4\bar C A_1/\ve)$ we conclude that
\be
\label{Repsilon2}
R(\ve)\leq \Big\lceil \bar C_1(\zeta)\Big(\log\big(\bar C_2/\ve))\big)\Big)^2\Big\rceil, \quad \bar C_2 := 4 \bar C_0 \bar C A_1.
\ee
Then, defining the function $  u(\epsilon):=\bar S_{R(\ve)}(\cB)g_{r(\ve)}$ and estimating as in    \eref{ubar} gives
\be
\label{initial1}
\| u -  u (\ve) \|_{t+2}\le   2A_1 \gamma(r(\ve))^{-1}\|f\|_{\bar \cA^\gamma(\HH^t,\HH^{t+\zeta})}\le \frac{\ve}2 \|f\|_{\bar\cA^\gamma}.
\ee

Thus, the main issue we are faced with is how to give a numerically realizable
approximation $\bar u(\ve)$ to
$u(\ve)=\bar S_{R(\ve)}(\cB)g_{r(\ve)}$, which satisfies
\be
\label{left1}
 \|\bar u(\ve) -u(\ve)\|_{t+2}\leq \frac{\ve}2 \|f\|_{\bar\cA^\gamma(\HH^t,\HH^{t+\zeta})}.
 \ee
 The following section describes a numerical algorithm for constructing such a  $\bar u(\epsilon)$.  We emphasize that although $\bar S_R(\cB)$ is a linear operator, in order to preserve ranks, its numerical approximation
will be a nonlinear process.

 \subsection{Numerical approximation of exponential operators}
 
 The main issue in our numerical algorithm is to provide a numerical realization of the application of the operator $\bar S_R(\cB)$ acting on  finite rank functions $g$, in our case $g=g_r = \sum_{\ell=1}^r g^{(\ell)}$.  Since
\be
\label{comp1}
\bar S_R(\cB) g=\sum_{\ell=1}^r \bar S_R(\cB) g^{(\ell)},
\ee
we need a numerical implementation of the application of the operator $\bar S_R(\cB)$ on the {\em rank-one tensors} $g^{(\ell)}$.  Given such a rank-one tensor, which we will denote by
$\tau$, we have
\be
\label{comp2}
\bar S_R(\cB) \tau =\sum_{k=1}^R \bar\omega_k {\rm e}^{-\alpha_k\cB}\tau,  
\ee
where, by Lemma \ref{explemma1},
\be
\label{omega-0}
\alpha_k=\alpha_{R,k},\quad \bar\omega_k := \left\{
\begin{array}{ll}
\omega_{R,k} &\mbox{when}\,\, \alpha_{R,k}\geq T_R^{-1};\\
0 & \mbox{otherwise}.
\end{array}\right.
\ee
Thus, the core task   is
  to approximate terms of the form
\beqn
\label{core}
{\rm e}^{-\alpha \cB} \tau = \bigotimes_{j=1}^d \Big({\rm e}^{-\alpha \cB_j}\tau_j\Big),\quad \mbox{where } \, \tau=\bigotimes_{j=1}^d\tau_j  
\eeqn
is one of the summands $g^{(\ell)}$ of $g_{r(\epsilon)}$.

In this section, we give
  a procedure for approximating ${\rm e}^{-\alpha  \cB_j}\tau_j$ and analyze its error.  This numerical procedure is based on
  the Dunford representation of the exponential
\beqn
\label{dunford}
{\rm e}^{-\alpha \cB_j}= 
 \frac{1}{2\pi i}\int_\Gamma {\rm e}^{-\alpha z}(z \mI-\cB_j)^{-1}{\dd}z,\quad \alpha > 0,
\eeqn
where $\mI$ is the identity operator (canonical injection of $H_j^{t+2}\to H_j^t$) and $\Gamma$ is a suitable curve in the right complex half-plane.  Recall that the Dunford representation \eqref{dunford} holds for sectorial operators, and, in  particular, for the
symmetric linear elliptic operators considered here.
    For simplicity of exposition we shall assume that for each $j=1,\ldots,d$,
we can take the same  $\Gamma=\Gamma_j$ and that this curve is  symmetric with respect to $\R$, see Figure \ref{fig.Gamma} for an illustration.
\begin{figure}[ht]
\centering
\includegraphics[scale=0.4]{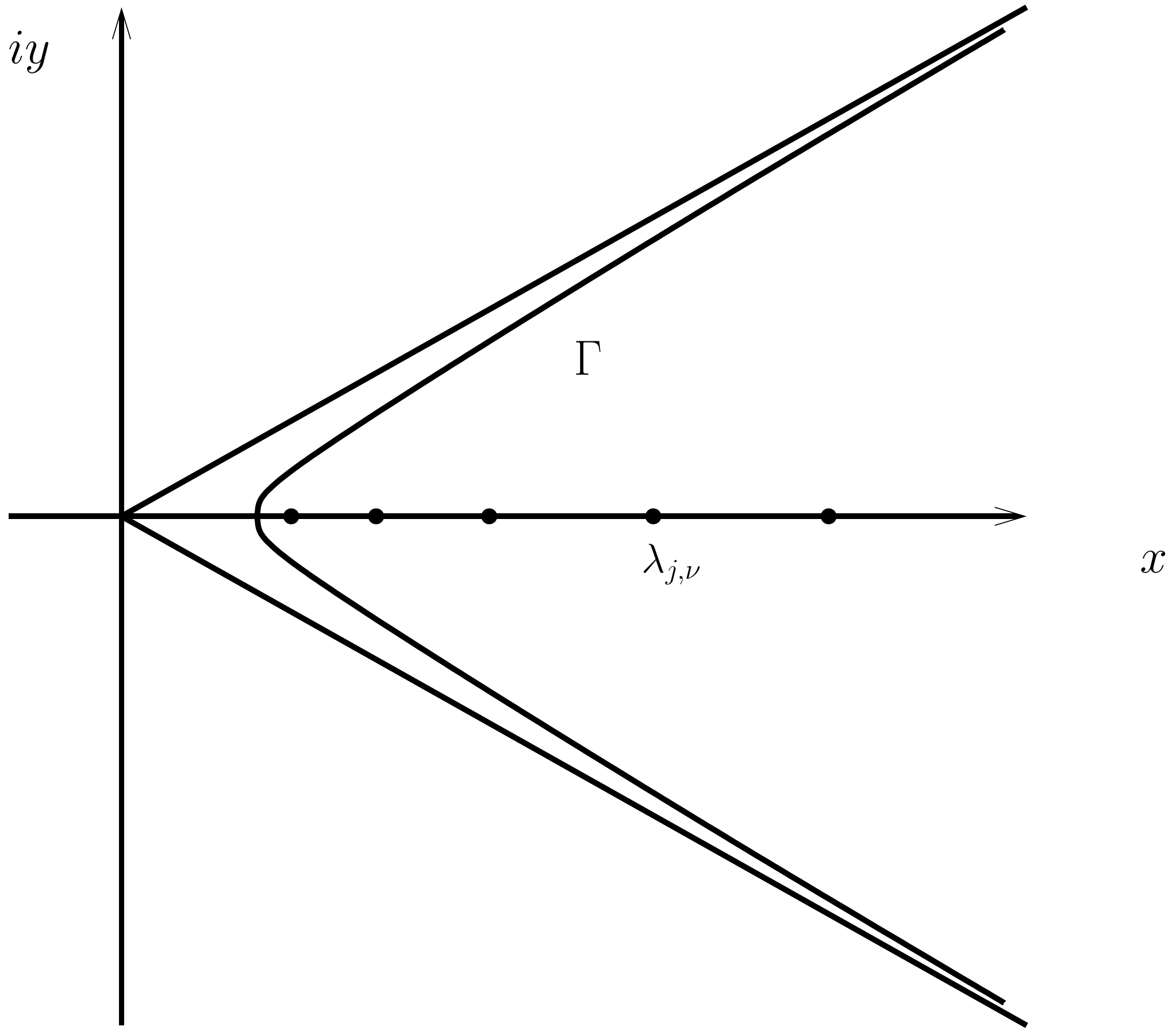}
\caption{The contour $\Gamma$\label{fig.Gamma}}
\end{figure}
The numerical realization of the operator exponential is done in two steps. The first one employs a quadrature for the above integral
with exact integrands. Since the integrands themselves are  solutions of operator equations the second step consists in
approximately solving these operator equations for the chosen quadrature points $z\in\Gamma$.

Quadrature methods for contour integrals of the above type are well studied, see e.g. \cite{Stenger,kress}. Here we follow
closely the results from \cite{J-diss,J-numath,DJ}, which are tailored to our present needs. Specifically, as  in \cite{J-diss}, we choose
for a fixed $\underline c\leq \ulambda/2$, the hyperbola
\be
\label{Gamma}
\Gamma(x):= \underline c + \cosh(x+ i \pi/6) = \underline c + \cosh(x)\cos(\pi/6) + i \sinh(x)\sin(\pi/6)
\ee
as a parametrization of the curve $\Gamma$,
which can be seen to have asymptotes $\underline c + t {\rm e}^{\pm i\pi/6}$. Denoting by $S_b\subset \C$ the symmetric strip
around the  real axis of width $2b$, $\Gamma$ extends to a holomorphic function
$$
\Gamma(z) = \underline c + \cosh(x+ i (\pi/6 + y)),\quad z= x+iy\in S_b
$$
for any $b>0$. Therefore, the operator-valued integrand (with clockwise orientation)
\be
\label{Fdef}
F_j(z,\alpha,\cdot) := -  {\frac{1}{2\pi i}} \sinh(z+i\pi/6) {\rm e}^{-\alpha \Gamma(z)}({\rnew \Gamma}(z)\mI -\cB_j)^{-1}
\ee
 is analytic in the strip $S_b$ provided that $\Gamma(S_b)$ does not intersect the spectrum of $\cB_j$.
As shown in \cite[\S 1.5.2]{J-diss}, this is ensured by choosing $b$ such that $\cos(\pi/6 -b)+\underline c < \ulambda$,
which we will assume  to hold  from now on.   Moreover, $\mathfrak{Re}(\Gamma(z))$ tends exponentially to
infinity as $\mathfrak{Re}(z)\to \infty$.

Under these premises we know that,  {for $t\in\R$},
\be
\label{res-bound}
\sup_{z\in\Gamma}\|(z\mI -\cB_j)^{-1}\|_{H_j^t \to H_j^{t+2}}\leq M,\quad j=1,\ldots, d,
\ee
for some constant $M>0$.

The following result, specialized to the present setting, has been shown in \cite[\S 1.5.4]{J-diss}, see also \cite{J-numath}.
 
\begin{theorem}
\label{thm:J-trapez}
Let {$\alpha>0$, $t\in \R$} and
\be
\label{betas}
\beta_0:= \frac{\alpha}{2}\cos(\pi/6),\quad \beta_1 := \frac{\alpha}{2}\cos(\pi/6 +b),\quad \beta_2:= \frac{\alpha}{2}\cos(\pi/6-b).
\ee
Define
\be
\label{constant}
C(\alpha):= \frac{M}{\pi} {\rm e}^{-\underline c \alpha}\Big(\frac{1}{\beta_0} + \frac{{\rm e}^2}{{\rm e}^2-1}\Big(\frac{1}{\beta_1}{\rm e}^{-\beta_1} +
\frac{1}{\beta_2}{\rm e}^{-\beta_2} \Big)\Big).
\ee
Then, for the operator
\be
\label{Q-def}
Q_N(F_j(\cdot,\alpha,\cdot)):= h \sum_{q=-N}^N F_j(qh,\alpha,\cdot),
\ee
 one has
\be
\label{quad-err}
\big\| {\rm e}^{-\alpha\cB_j}- Q_N(F_j(\cdot,\alpha,\cdot))\big\|_{H_j^t \to H_j^{t+2}}\leq C(\alpha){\rm e}^{-2\pi b/h},
\quad j=1,\ldots,d,
\ee
where
\be
\label{NQ}
N= N(h) := \max\left\{0, \left\lfloor\frac 1h \log (\beta_0)^{-1}\right\rfloor +1, \left\lfloor\frac 1h \log \Big(\frac{2\pi b}{\beta_0h}\Big) \right\rfloor +1
\right\}.
\ee
\end{theorem}

Notice that there exist positive constants $C_1, c_1$, depending only on $\ulambda$ and  $\underline c$ in \eref{Gamma} such that
\be
\label{Calpha}
C(\alpha) \leq \frac{C_1}{\alpha} {\rm e}^{-c_1\alpha} =: \bar C(\alpha).
\ee
We also   remark that one has, for $N,h$ related
by \eref{NQ},
\be
\label{hN}
h(2N+1)\leq C_3 |\log (\alpha h)|,
\ee
where $C_3$ depends only on $\ulambda$ and $\underline c$.

We   will use the above bound also for the smaller quantities
$$
\big\| {\rm e}^{-\alpha\cB_j}- Q_N(F_j(\cdot,\alpha,\cdot))\big\|_{H_j^{t'} \to H_j^s},$$
 when $0\leq s-t' \leq 2$.

Since  $Q_N(F_j(\cdot,\alpha,\cdot))$
is an approximation to ${\rm e}^{-\alpha\cB_j}$,  it is natural to approximate ${\rm e}^{-\alpha\cB}\tau$ by $\bigotimes_{j=1}^d Q_N(F_j(\cdot,\alpha,\cdot))\tau$.
The factors  $Q_N(F_j(\cdot,\alpha,\cdot))$ in the approximation to ${\rm e}^{-\alpha\cB}\tau$ are, of course,  not yet computable
because the quantities $u_{j,q}(\alpha,\tau_j) :=  F_j(qh,\alpha,\tau_j)$ are the {\em exact} solutions of
  the resolvent equations
\be
\label{resolvent}
(\Gamma(qh)\mI-\cB_j)u_{j,q}(\alpha,\tau_j) = -\frac{1}{2\pi}\sinh(qh+i\pi/6)\,{\rm e}^{-\alpha\Gamma(qh)}\tau_j
\ee
for $q=-N,\ldots,N$. Hence our numerical procedure will be based on {\em approximate} solutions
to \eref{resolvent}.
\begin{definition}
\label{def:accurate}
Let  $0\leq s$,  $s-t' <2$. An approximate solution $\bar u_{j,q}(\alpha,\tau_j)$ of \eref{resolvent} is called $(s,t')$-accurate for $\alpha$ if
\be
\label{baru}
\|u_{j,q}(\alpha,\tau_j) - \bar u_{j,q}(\alpha,\tau_j)\|_{H^{s}_j}\leq (C_3 |\log (\alpha h)|)^{-1}C(\alpha){\rm e}^{-2\pi b/h}\|\tau_j\|_{H_j^{t'}},
\ee
where $C_3$ is the constant from \eref{hN}.
\end{definition}

 Note that here we need $s-t'<2$ to be able to achieve a desired accuracy.
We postpone at this point the discussion of how and at what cost one can obtain such approximations.

Given the approximations $\bar u_{j,q}(\alpha,\tau_j)$ of the exact integrand $u_{j,q}(\alpha,\tau_j)=F_j(q\pi,\alpha,\tau_j)$, we define
the computable approximations
\be
\label{Ej}
E_j(\tau_j,\alpha,h):= h\sum_{q=-N}^N \bar u_{j,q}(\alpha,\tau_j),
\ee
to $Q_{N}(F_j(\cdot,\alpha,\tau_j))$ where $N=N(h)$, as well as
 \be
\label{Edef}
E(\tau,\alpha,h):= \bigotimes_{j=1}^d E_j(\tau_j,\alpha,h),
\ee
as a numerically realizable approximation to the rank-one tensor ${\rm e}^{-\alpha\cB}\tau$.

\subsection{The numerical approximation of \texorpdfstring{$u=\cB^{-1}f$}{u=inv(B) f}}
\label{numericalBinverse}

We are now in a position to specify our numerical approximation $u=\cB^{-1}f$.    Recall that
$u(\ve)= \bar S_{R(\ve)}(\cB)g_{r(\ve)}$, with the choice of $R(\ve)$ and $r(\ve)$ specified in \eref{Repsilon} and \eref{repsilon} respectively, satisfies
\be
\label{initial11}
\| u -  u (\ve) \|_{t+2}\le     \frac{\ve}2 \|f\|_{\bar\cA^\gamma(\HH^t,\HH^{t+\zeta})},
\ee
 because of \eref{initial1}.  We will now use as our approximation to $u$ the following computable quantity
\be
\label{uepsilon-0}
  \bar S_{R(\ve),h}(\cB)(g_{r(\ve)}) := \sum_{\ell=1}^{r(\ve)} \sum_{k=1}^{R(\ve)} \omega_{R(\ve),k} E(g^{(\ell)},\alpha_{R(\ve),k},h).
\ee
 It remains to specify $h$ sufficiently small, so as to  achieve the error bound
\be
\label{eb1}
\big\|u(\ve)-S_{R(\ve),h}(\cB)(g_{r(\ve)})\big\|_{t+2}\le  \frac{\ve}2 \|f\|_{\bar \cA^\gamma(\HH^t,\HH^{t+\zeta})}.
\ee
The following result whose proof is deferred to \S \ref{sec:proofs} provides a sufficient choice of $h=h(\ve)$.
 
 \begin{proposition}
\label{prop:SRBepsilon}
Assume that $t+\zeta \ge 0$ and $s- (t+\zeta)<  2$ and let  $u(\ve)$ be the approximation from {\rm \eref{initial1}}.
Then there exists a constant $c_6$,
depending only on $\zeta,\bar C, \bar C_2$,  defined in {\rm \eref{Repsilon2}},  $A_1$ from {\bf (A1)},  on  $\bar C_0$ {\rm(}hence depending on $s-t'>0${\rm)},
$\ulambda, \Gamma$,
as well as on the minimum growth of $\gamma$  in {\rm \eref{gamma2} }such
that for
\be
\label{condh2}
h=h(\ve) := c_6 \Big(\log\Big(\frac{ {d}}{\ve}\Big)\Big)^{-1}
 \ee
 the function
\be
\label{uepsilon}
\bar u(\ve):= S_{R(\ve),h(\ve)}(\cB)(g_{r(\ve)}),
\ee
given by {\rm \eref{uepsilon-0}} for $h=h(\ve)$, satisfies
\be
\label{errorh}
\|\bar u(\ve)- u(\ve)\|_{s} \leq \frac{\ve}{2A_1}\|g_{r(\ve)}\|_{t+\zeta}\le \frac{\ve}{2}\|f \|_{\bar A^\gamma(\HH^t,\HH^{t+\zeta})},
\ee
 provided that for
$j=1,\ldots,d$, $\ell=1,\ldots,r(\ve)$ and $q=-N(h(\ve)),\ldots,N(h(\ve))$, the approximate solutions $\bar u_{j,q}(\alpha_{R(\ve),k},g^{(\ell)}_j)$, entering \eref{Ej} are $(s,t+\zeta)$-  and $(0,0)$-accurate for $\alpha_{R(\ve),k}, k=1,\ldots,R(\ve)$.
 \end{proposition}

\section{A rate distortion theorem}  \label{sec:discr}
   Recall that    the spatial dimension of the factor
 domains $D_j$ is $p\in\{1,2,3\}$ and so a key task is to approximately solve low-dimensional elliptic problems, see \eref{resolvent}.
The size of the discretizations of the low-dimensional operator equations \eref{resolvent}, required in order to realize a desired target accuracy, depends on the regularity of the solution.  In the scale of spaces $H^t_j$ and $\HH^t$, $t\in \R$,  this   regularity follows, in view of \eref{isometry}
and analogous relations for the low-dimensional component operators, from  our assumption on the right-hand side $f$ stated in  assumption {\bf A1}.
This gives a     control on the rank-one terms $g^{(\ell)}_r$ in the $\|\cdot\|_{t+\zeta}$ norm  which, in turn, implies a regularity of $u$, see
\eref{tripnorm1}.  
Suitable approximations in terms of the eigensystems of the component operators $\cB_j$ could then be derived from Lemma \ref{soblemma}.
However, if one wants to dispense with using the eigensystems and, as we will do now,  employ instead standard numerical techniques for low-dimensional operator equations, then one needs more information about the spaces $\HH^t, H^t_j$.
 
  Since our goal is only to illustrate that, even in the absence of having an eigensystem at hand,
   it is possible to construct numerical algorithms that exhibit
 significant  computational savings when utilizing tensor structure, we shall, for the remainder of  \S \ref{sec:discr},  place ourselves in
 a more specific setting where  known standard finite element solvers can be employed and bounds for their numerical efficiency are available.
 Specifically, we shall limit ourselves to  approximating the solution in the energy norm $\HH^1$.   One issue to overcome is that this   is not a cross norm
 and hence is not as compliant with tensor structures as the $\LL_2$ norm.    So, we are interested in approximating the solution $u$
  in the $\HH^{1}$ norm, given $\HH^{-1+\zeta}$-data. Certain restrictions on $\zeta$ will be identified below.
  \vskip .1in
 \noindent
 {\bf Assumptions for the rate distortion theorem:} {\it
   We assume the validity of {\bf (A1)} and that $f\in \bar\cA^\gamma(\HH^{-1},\HH^{-1+\zeta})$. In addition, we make the following assumptions: }

  \vskip .1in
 \noindent
{\bf (A2):}
{\em \rd{For $s \in [0,2]$,} the  spaces $H^s_j$, $j=1,\ldots,d$,
 agree with classical Sobolev spaces $H^s(D_j)$
with equivalent norms.
 Hence, the same holds for the spaces $\HH^s$, $s\in [0,2]$, i.e.,  there exist \rd{positive} constants $c^*, C^*$ such
that
\beqn
\label{isoext}
c^*\|\cB v\|_{{s-2} }\leq \| v\|_{H^s(D)} \leq C^*\|\cB v\|_{s-2}, \quad v\in H^s(D),\quad 0\leq s \leq   2.
\eeqn
{\bf (A3):} The factor domains $D_j$ have the same``small'' dimension $p\in \{1,2,3\}$. The operators $\cB_j$ are symmetric $H_j^1$-elliptic
operators {\rm (}see \eref{Di-ell}{\rm)} so that whenever $z\in \C$ satisfies  $\min\,\{|z-\lambda_{j,k}|:k\in\N\}> c\, \mathfrak{Re}(z)$, the problem
\be
\label{ldproblem}
(z\mI -\cB_j)v=w
\ee
possesses, for $j=1,\ldots,d$, and any $w\in H^{-1}_j$ a unique solution $v$. In particular, by {\rm \eref{res-bound}},   for $w\in H_j^{-1+\zeta}$
and $z\in \Gamma$,
  one has $\|v\|_{H^{1+\zeta}_j}\leq M \|w\|_{H^{-1+\zeta}_j}$, $j=1,\ldots,d$.

  Moreover,    we have at our disposal a discretization method using linear spaces $V_{j,\delta}$ with the following properties: For any target
accuracy $\delta >0$, to   obtain an approximate solution $v_\delta \in V_{j,\delta}\subset H^1_j(D_j)$
satisfying
\be
\label{target-sol}
\|v-v_\delta\|_{H_j^{1 }}\leq \delta \|w\|_{H_j^{-1+\zeta}},
\ee
requires a trial space of dimension at most
\beqn
\label{Vjdelta}
n_\delta := {\rm dim}\, V_{j,\delta} \leq C_8  
\delta^{-p/\zeta  }, \quad \delta >0,
\eeqn
and the number of operations ${\bf flops}(v_\delta)$, required to compute $v_\delta$, is bounded by
\beqn
\label{flops}
{\bf flops}(v_\delta) \leq C_9 \delta^{-p/\zeta}, \quad \delta >0,
\eeqn
with $ C_8,C_9$ independent of $j=1,\ldots,d$.}\\

Clearly, under the assumption  {\bf (A2)}, the hypotheses in {\bf (A3)} are well justified, simply resorting to
available computational techniques in low spatial dimensions.

As for the justification of {\bf (A2)}, we state the following facts.

\begin{remark}
\label{regularity+}
Suppose, for example, that $\mathfrak{B}_j$ is a second-order strongly elliptic operator of the form
\[\mathfrak{B}_j v := \sum_{k,m=1}^p \frac{\partial}{\partial x_{j_k}}\left(a_{j_k j_m}(x_{j_1},\dots, x_{j_p})\frac{\partial v}{\partial x_{j_m}}\right),\qquad j=1,\dots,d,\]
with $a_{j_k j_m} = a_{j_m j_k} \in C^{0,1}(\overline{D_j})$
for $k,m=1,\dots,p$, where $D_j$ is a bounded open convex domain in $\mathbb{R}^p$, $p \in \{1,2,3\}$, $j=1,\dots, d$. Hence, $D = \times_{j=1}^d D_j$ is
a bounded open convex domain (cf.\ \cite{HUL}, p.~23) and $\mathfrak{B}$ is a second-order strongly elliptic operator on $D$. Assuming a homogeneous Dirichlet boundary condition on $\partial D$, the second inequality in \eqref{isoext} holds with $s=s^*=2$ thanks to Theorem 3.2.1.2 in Grisvard \cite{Grisvard} and the equivalence of the $H^2(D)$ norm on $H^2(D)\cap \HH^1_0(D)$ with the standard Sobolev norm of $H^2(D)$. The first inequality follows trivially, by noting the regularity hypothesis on the coefficients $a_{j_k,j_m}$ and the equivalence of the $\HH^2$ norm with the standard Sobolev norm of $H^2(D)$ on $H^2(D)\cap H^1_0(D)$. For $s \in (1,2)$, $s \neq 3/2$, the pair of inequalities
\eqref{isoext} is deduced by function space interpolation. For elliptic regularity results of the kind \eqref{isoext} with $s=s^*=2$ for second-order degenerate elliptic operators appearing in Fokker--Planck equations we refer to \cite{FS,FS:Arxiv}.
\end{remark}

  \subsection{The numerical algorithm}
 \label{thealgorithm} We shall now specify the above construction of the finitely parametrized finite rank approximation
 $\bar u(\ve)= S_{R(\ve),h(\ve)}(\cB)(g_{r(\ve)})$ from \eref{uepsilon}
in the following scheme.\\

\noindent
{\bf Scheme-Exp:} Given $\ve >0$,  $g(\ve)=g_{r(\ve)}\in \cT_{r(\ve)}(\HH^{t+\zeta})$ satisfying \eref{follows2}, $R=R(\ve)$, $r=r(\ve)$,
defined by \eref{Repsilon}, \eref{repsilon},
 the scheme produces a numerical approximation $\bar u(\ve)$ to the solution $u$ of \eref{varprob-2} as follows:
\begin{itemize}
\item[(i)]
Fix  
the curves $\Gamma_j =\Gamma $, according to
\eref{Gamma}, $j=1,\ldots,d$, along with the collections of
 quadrature points  $\Gamma_{ Q}:=\{\Gamma(qh): q=-N,\ldots,N\}$, $N=N(h)$ (see \eref{NQ}), $h=h(\ve)$, given by \eref{condh2};
 \item[(ii)]   For $k=1,\ldots,R(\ve)$, $\ell=1,\ldots,r$, $q=-N,\ldots, N$, $j=1,\ldots, d$, compute
 an approximate solution $\bar u_{j,q} (\alpha_{R,k},g^{(\ell)}_j)$ to \eref{resolvent}, satisfying \eref{baru}
 for $s= 1$, $t'= -1+\zeta$;
 \item[(iii)]   compute $ E_j(g_j^{(\ell)},\alpha_{R,k},h(\ve))$ by \eref{Ej} for $j=1,\ldots,d$, $\ell=1,\ldots,r$, $k=1,\ldots, R$,
and output
\be
\label{uepsilon-2}
\bar u(\ve):= S_{R(\ve),h(\ve)}(\cB)(g_{r(\ve)}).
\ee
   \end{itemize}
\medskip

\subsection{Statement of the Theorem}

We can now formulate a theorem, which
bounds the number of computations necessary for achieving a prescribed accuracy $\ve$
using the numerical {\bf Scheme-Exp} described above.  As has been already mentioned, we describe this only for approximating the solution
$u$ of \eref{varprob-2} in the $\HH^1$ norm.
The main result of \S \ref{sec:complexity} reads as follows.

\begin{theorem}
\label{thm:complex}
 We assume that  {\bf (A1)} holds, i.e., we are given
a function $f\in \bar\cA^\gamma(\HH^{-1},\HH^{-1+\zeta})$ and for each $r\in\N$
a $g_r\in \cT(\HH^{-1+\zeta})$, satisfying {\rm \eref{gr}}.
Moreover, assume that $1\le \zeta \leq 2$ and that
{\bf (A2)}, {\bf (A3)} hold.  Then, given any target accuracy  $\ve >0$,
the approximate solution
$$
\bar u(\ve) = S_{R(\ve),h(\ve)}(\cB,g_{r(\ve)}),
$$
of rank at most $r(\ve)R(\ve)$, produced by {\bf Scheme-Exp} has the following properties:\\

\noindent
{\rm (i)} $\bar u(\ve)$ is $\ve$-accurate and stable, i.e.,
\be
\label{u-1}
\|u-\bar u(\ve)\|_1 \le \ve \|f\|_{\bar\cA^\gamma(\HH^{-1},\HH^{-1+\zeta})},\quad \tripnorm{\bar u(\ve)}_{r(\ve)R(\ve),1 } \leq %C_7
A_2\|f\|_{\bar\cA^\gamma(\HH^{-1},\HH^{-1+\zeta})},
\ee
where $A_2$ depends only on $C_0$ and $A_1$.\\

\noindent
{\rm (ii)} $\bar u(\ve)$ is finitely parametrized and the total number of parameters determining $\bar u(\ve)$
is bounded by
\be
\label{parameters}
A_3 d^{1+ \bar\rho p/\zeta}  {\ve}^{-\bar\rho p/\zeta}\gamma^{-1}(4\bar C A_1/\ve) \Big(\log\Big(\frac{\bar C_2}{\ve}\Big)\Big)^2 \Big(\log\Big(\frac{d}{\ve}\Big)\Big)^2,
 \ee
where $\bar\rho > 2\pi b/c_6$ is any fixed number.

\noindent
{\rm (iii)} The computational complexity needed to compute $\bar u(\ve)$ by {\bf Scheme-Exp} is bounded by
\be
\label{total-compl}
\cost (\bar u(\ve)) \leq  
A_4 d^{1+ \bar\rho p/\zeta}  {\ve}^{-\bar\rho p/\zeta}\gamma^{-1}(4 \bar C A_1/\ve)\Big(\log\Big(\frac{\bar C_2}{\ve}\Big)\Big)^2  
\Big(\log\Big(\frac{d}{\ve}\Big)\Big)^2,
\ee
where  the constants $A_3$, $A_4$ depend on $\bar\rho$, $\ulambda$, $\Gamma$,  $A_1$, and the constants $C_{8}$, $C_9$ in  {\rm \eref{target-sol}}
and {\rm \eref{Vjdelta}} in
assumption {\bf (A3)}).
 \end{theorem}
The bounds \eref{parameters}, \eref{total-compl} given in Theorem \ref{thm:complex} are somewhat
worse than the benchmark provided by Theorem \ref{thm:rep-compl} and are perhaps not best possible.
They rely on the specific way of approximating $\cB^{-1}$ and its discretization via Dunford integrals.
In particular, it is not clear whether $\zeta\geq 1$ is necessary. The proof of Theorem \ref{thm:complex}, which
will be given in \S \ref{sec:proofs}, and Remark \ref{rem:problem} below will shed some light on this constraint.

It is also not clear whether the upper limitation $\zeta \leq 2$ is necessary, i.e.,  whether the bounds
in \eref{parameters} and \eref{total-compl} continue to hold for $\zeta >2$,  which would mean that one could exploit even higher regularity of
the solutions to the resolvent equations. As will be seen below, the restriction arises when applying
Proposition \ref{prop:SRBepsilon} requiring the numerical approximation to be simultaneously $(1,-1+\zeta)$- and $(0,0)$-accurate.

Nevertheless, the theorem confirms
polynomial numerical inversion tractability under overall very mild smoothness assumptions, exploiting instead a
``global structural sparsity'' of the solutions inferred from such a sparsity of the data. In fact, save for logarithmic terms
these bounds are comparable to those for the simple example about integration given in the introduction, see
\eref{N2}. In particular, although the growth in $\ve^{-1}$ and $d$ is somewhat stronger than in \eref{N2} or \eref{Ned},
one does benefit from a larger  excess regularity at least up to $\zeta\le 2$ of the low-dimensional problems.

Finally,   in contrast with Theorem \ref{thm:mainII} the
approximate solution $\bar u(\ve)$ is only guaranteed to be {\em tensor-stable}   with a controlled $\tripnorm{\cdot}_{rR,1}$ norm
not with a controlled $\tripnorm{\cdot}_{rR,1+\zeta}$ norm. First of all, the latter bound would require
the approximate solutions of the resolvent equations to belong to $H^{1+\zeta}_j$, which is not the case for $\zeta \geq 1/2$
and standard $C^0$-continuous finite elements. However, it can be seen from the proof of Theorem \ref{thm:complex}
that $\bar u(\ve)$ could be arranged to satisfy $\tripnorm{\bar u(\ve)}_{rR,1+\zeta'}\leq C \|f\|_{\bar \cA^\gamma(\HH^{-1},\HH^{-1+\zeta})}$
for some $0<\zeta' <\zeta$, for which the low-dimensional trial spaces fulfill $V_{j,\delta}\subset H^{1+\zeta'}_j$.
However, the numerical approximations to \eref{resolvent} would need to be $(1+\zeta',-1+\zeta)$-accurate and thus require possibly
finer discretizations. As a consequence, the bounds  \eref{parameters} and \eref{total-compl} would only hold for $\zeta$ replaced by
$\zeta-\zeta'$. Since, the asserted tensor-stability in $\HH^1$ suffices to avoid the occurrence of norm-degenerations often encountered
with the canonical tensor format, we  omit the details.

The results are to be contrasted with the intractability results in
  \cite{NW09} stating that the approximation of a high-dimensional function in a Sobolev norm of positive order is intractable
  even under excessively high smoothness assumptions.
 The results in \cite{WW} on the tractability of high-dimensional Helmholtz problems are different in spirit. The conditions on the
 data required in \cite{WW} strongly constrain the dependence of $f$ on some of the variables, and \rd{adopt the additional assumption that} $\cB^{-1}$ as a mapping
 from such a data class into $\HH^1$ can still be diagonalized. Thus the inversion of $\cB$, which is the focus of the present work, is not an issue there.

\section{Proof of Theorem \ref{thm:complex}}\label{sec:proofs}
 
The proof of Theorem \ref{thm:complex} is based on a series of intermediate, partly technical results.
The first of these is to give an approximation to  exponential operators of the type ${\rm e}^{-\alpha\cB}$ that appear in the definition of the numerical approximation $\bar u(\ve)$ of \eref{uepsilon-2}.
 
\subsection{Approximation of exponential operators} The present subsection and the next subsection require only the assumption {\bf (A1)}.
A recurring issue is the need to bound the norms of ${\rm e}^{-\alpha\cB_j}$ and the corresponding quadrature-based counterparts
as mappings from $H_j^{t'}$ to $H_j^{s}$ also when $s> t'$ while keeping track of the role of $\alpha >0$ when $\alpha$ gets
small.

In what follows we make an effort to trace the dependence of the various constants on the problem parameters and, in particular, on $d$.
To simplify the exposition somewhat we assume throughout this section that $\lambda_{j,m}\ge 1$, $j=1,\ldots,d$, $m\in \N$, which
implies that $\ulambda \ge 1$. As a consequence we have $\|\cdot\|_{H^s_j}\leq \|\cdot\|_{H^{s'}_j}$ for $s'\ge s$.
Of course, this can always be achieved by renormalizing $\cB$ by a fixed constant factor. Hence all subsequent conclusions
remain valid in the general case up to another adjustments of the constants depending on the smallest eigenvalues.

We begin with the exact quadrature operator $Q_N(F_j(\cdot,\alpha,\cdot))$.
 
\begin{lemma}
\label{lem:semigroupbound}
For any $s $ and $t'$, satisfying $s-t'\leq 2$,  and any $\alpha >0$, one has the following bounds:
\be
\label{e-alpha-s}
\|{\rm e}^{-\alpha\cB_j} \|_{H_j^{t'}\to H^s_j} \leq  
\Big(\frac{(s-t' )_+}{2e\alpha}\Big)^{\frac{(s-t')_+}{2}}, 
\ee
as well as
\be
\label{quad-bound-s}
\|Q_N(F_j(\cdot,\alpha,\cdot))\|_{H_j^{t'}\to H^s_j}\leq  
\Big(\frac{(s-t')_+}{2e\alpha}\Big)^{\frac{(s-t')_+}{2}} + \bar C(\alpha){\rm e}^{-2\pi b/h}.
\ee
\end{lemma}
\begin{proof}
We have
$$  
\|{\rm e}^{-\alpha\cB_j}\tau_j\|_{H^s_j}^2 = \sum_{k=1}^\infty {\rm e}^{-2\alpha\lambda_{j,k}}\lambda_{j,k}^{s-t'} \lambda_{j,k}^{t'}|
\langle \tau_j,e_{j,k}\rr|^2.
$$
When  {$s-t' \ge 0$} we note that the function $h(x)= x^{\beta}{\rm e}^{-\alpha x}$ attains its
maximum at $x^*=\beta/\alpha$, i.e.,
\be
\label{xbeta}
x^\beta {\rm e}^{-\alpha x} \leq   (\beta/\alpha)^{\beta}{\rm e}^{-\beta}.
\ee
 Taking $\beta = (s-t')/2$ confirms \eref{e-alpha-s} in this case.
 When $s-t' < 0$, we apply \eref{e-alpha-s} for $s=t'$ and use then that $\|\tau_j\|_{H^s_j}\leq \|\tau_j\|_{H^{t'}_j}$.
 The bound \eref{quad-bound-s} follows from \eref{e-alpha-s}
 and Theorem \ref{thm:J-trapez}.
 \end{proof}

\begin{remark}
\label{rem:0}
Note that, in particular, one has for $s\leq t'$ the following bounds:
\be
\label{specialest}
\|{\rm e}^{-\alpha\cB_j} \|_{H_j^{t'}\to H_j^s} \leq 1,\quad \|Q_N(F_j(\cdot,\alpha,\cdot))\|_{H_j^{t'}\to H_j^s}\leq 1+ C(\alpha){\rm e}^{-2\pi b/h}.
\ee
\end{remark}

Next we shall address the effect of replacing the quantities $u_{j,q}(\alpha,\tau_j) :=  F_j(qh,\alpha,\tau_j)$, for $q=-N,\ldots,N$, used in
$Q_N(F_j(\cdot,\alpha,\tau_j))$, which are the exact solutions of the resolvent equations \eref{resolvent}, by $(s,t')$-accurate
approximate solutions $\bar u_{j,q}(\alpha,\tau_j)$ for $\alpha$ (see  \eref{baru}) where $s,t'$ with $s-t'< 2$ are fixed.
We shall later need this for $s=t+2$   and $t'=t+\zeta$ with $0 < \zeta\le 2$ and will
analyze the complexity of these low-dimensional problems for these choices. We recall from \eref{Ej} the definition of the computable approximations $E_j(\tau_j,\alpha,h)$ to $Q_{N}(F_j(\cdot,\alpha,\tau_j))$.

 \begin{lemma}
\label{lem:Ej}
Let $s, t'\in \R$   satisfy  $s-t' < 2$. Assume that {\rm \eref{baru}} holds for $s$ and $t'$. Then, one has
\be
\label{Ejapprox}
\|{\rm e}^{-\alpha\cB_j}\tau_j - E_j(\tau_j,\alpha,h)\|_{H_j^s}\leq 2 \bar C(\alpha){\rm e}^{-2\pi b/h}\|\tau_j\|_{H_j^{t'}},
\ee
and
\be
\label{Ejs}
\|E_j(\tau_j,\alpha,h)\|_{H_j^s}\leq \left\{\Big(\frac{(s-t')_+}{2{\rm e}\alpha}\Big)^{\frac{(s-t')_+}{2}} +2 \bar C(\alpha){\rm e}^{-2\pi b/h}\right\}
\|\tau_j\|_{H_j^{t'}}.
\ee
\end{lemma}
\begin{proof}
Using the fact that each of the terms in the sums for $E_j$ and $Q_N$ satisfies \eref{baru}  and the number  $2N+1$ of terms satisfies \eref{hN},  gives
$$
\|Q_N(F_j(\cdot,\alpha,\tau_j))- E_j(\tau_j,\alpha,h)\|_{H_j^s}\leq   {\rnew C(\alpha)}{\rm e}^{-2\pi b/h}\|\tau_j\|_{H_j^{t'}}.
$$
The first estimate \eref{Ejapprox} now follows from Theorem \ref{thm:J-trapez} by  using the triangle inequality together with the fact that $C(\alpha)\le \bar C(\alpha)$. Likewise
\eref{Ejs} is a consequence of \eref{Ejapprox} and \eref{quad-bound-s}.
\end{proof}

Defining $E(\tau,\alpha,h):= \bigotimes_{j=1}^d E_j(\tau_j,\alpha,h)$ according to \eref{Edef},
one has the following error bounds.

\begin{lemma}
\label{lem:fullB}
Assume that $s\ge 0$ and  $s-t' < 2$. Furthermore, assume that the approximate solutions $\bar u_{j,q}(\alpha,\tau_j)$
used in $E(\tau,\alpha,h)$
are $(s,t')$-accurate as well as $(0,0)$-accurate for $\alpha$ (see \eref{baru}).
 Then, whenever
 \be
\label{hcond1}
 h\leq h_0:= 2\pi b \Big(\log \Big(\frac{2C_1 d}{\alpha}\Big)\Big)^{-1}
 \ee
holds, one has:

\noindent
{\rm (i)} for $\alpha \le 1$,
 \begin{eqnarray}
 \label{fullBest}
 \|{\rm e}^{-\alpha\cB}\tau - E(\tau,\alpha,h)\|_{ s} &\leq &C_2d^2\big(d^{\max\{0,s-1\}})^{1/2}\bar C(\alpha){\rm e}^{-2\pi b/h}  \left\{ ({\rm e}\alpha)^{-1}
  +2 \bar C(\alpha){\rm e}^{-2\pi b/h}\right\}\nonumber\\
 &&\times  \left\{ \Big(\frac{(-t')_+}{2{\rm e}\alpha}\Big)^{\frac{(-t')_+}{2}} +\frac 1d\right\}^{d-2} \|\tau\|_{t'};
 \end{eqnarray}
 {\rm (ii)} for $\alpha \ge 1$,
 \begin{eqnarray}
 \label{fullBest1}
 \|{\rm e}^{-\alpha\cB}\tau - E(\tau,\alpha,h)\|_{ s} &\leq &C_2d^2\big(d^{\max\{0,s-1\}})^{1/2}
 \bar C(\alpha){\rm e}^{-2\pi b/h}  \nonumber\\
 &&\times  \left\{ \Big(\frac{(-t')_+}{2{\rm e}\alpha}\Big)^{\frac{(-t')_+}{2}} +\frac 1d\right\}^{d-1} \|\tau\|_{t'},
 \end{eqnarray}
 where the constant $C_2$ depends only on the constant in Corollary 1.
 \end{lemma}
\begin{proof}
 We use a telescoping  decomposition, as in \eref{ae24}, to obtain the bound
\be
\label{decomp}
 \|{\rm e}^{-\alpha\cB}\tau - E(\tau,\alpha,h)\|_{s} \leq \sum_{i=1}^d\|S_i\|_s,
 \ee
where
\be
\label{Si1}
S_i:= 
\bigotimes_{j=1}^{i-1} E_j(\tau_j,\alpha,h)\otimes  ({\rm e}^{-\alpha\cB_i}\tau_i- E_i(\tau_i,\alpha,h) )\bigotimes_{j=i+1}^d {\rm e}^{-\alpha\cB_j}\tau_j:= \bigotimes_{j=1}^d S_{i,j}.  
\ee
We estimate next the terms $\|S_i\|_{s}$
with the aid of
 Lemma \ref{lem:H-1}, using the assumption $s\geq 0$.
 Since the $\ell_2$ norm \rd{does not exceed} the $\ell_1$ norm, this lemma gives that
 \be
 \label{Si}
 \|S_i\|_s\le  (d^{\max\{0,s-1\}})^{1/2}   \sum_{j=1}^{d}\|S_{i,j}\|_{H^{s}_j}\prod_{k\neq j}\|S_{i,k}\|_{L_2(D_k)}.
\ee
The previous lemmas provide the following \rd{bounds} for each $i=1,\dots,d$:
  \be
 \label{est1}
  \|S_{i,k}\|_{H_k^{s}(D_k)}\leq \|\tau_k\|_{H_k^{t'}(D_k)}
 \left\{
\begin{array}{ll}
 2 \bar C(\alpha){\rm e}^{-2\pi b/h},& k=i;\\
 \Big(\frac{(s-t')_+}{2{\rm e}\alpha}\Big)^{\frac{(s-t')_+}{2}} +2 \bar C(\alpha){\rm e}^{-2\pi b/h}, & k\neq i.
\end{array}
\right.
\ee
Indeed,  by $(s,t')$-accuracy, the first inequality in \eref{est1} follows from \eref{Ejapprox} while the second follows from
\eref{Ejs} and \eref{e-alpha-s}.
Similarly, we have the following bounds on  the $L_2$ norms for $\mu \in \{0,t'\}$,
\be
 \label{est2}
 \|S_{i,k}\|_{L_2(D_k)}\leq \|\tau_k\|_{H_k^{\mu}(D_k)}
 \left\{
\begin{array}{ll}
 2 \bar C(\alpha){\rm e}^{-2\pi b/h},& k=i;\\
 \Big(\frac{(-\mu)_+}{2{\rm e}\alpha}\Big)^{\frac{(-\mu)_+}{2}} +2 \bar C(\alpha){\rm e}^{-2\pi b/h}, & k\neq i.
\end{array}
\right.
 \ee
Here,  the first inequality follows from \eref{Ejapprox} while the second follows from
\eref{Ejs} and \eref{e-alpha-s}. In fact, since $\|\cdot\|_{L_2(D_j)}\le \|\cdot\|_{H^s_j}$ we have used $(s,t')$-accuracy
when $\mu= t'$ while for $\mu=0$ we used $(0,0)$-accuracy.
Moreover,  under the assumption \eref{hcond1}, we
have
\be
\label{then1}
2\bar C(\alpha){\rm e}^{-2\pi b/h}\leq \frac{2C_1}{\alpha}{\rm e}^{-2\pi b/h}\leq 1/d.
\ee

{\bf Case $t'\ge0$:}   Let us first note that   from the first inequality in Lemma \ref{lem:H-1}, we have
   \be
\label{lbl}
\sum_{i=1}^d \|\tau_i\|_{H_i^{t'}}
\prod_{k\neq i} \|\tau_k\|_{L_2(D_k)}\le   \|\tau\|_{t'}.
\ee
We first consider any term appearing in the sum on the right-hand side of \eref{Si} with $j\neq i$ and obtain
\begin{eqnarray}
\label{jneqi}
&&\|S_{i,j}\|_{H_j^s}\prod_{k\neq j}\|S_{i,k}\|_{L_2(D_k)}\nonumber \\
&&\;\le \{ ({\rm e}\alpha)^{-1}
 +2\bar C(\alpha){\rm e}^{-2\pi b/h}\}\|\tau_j\|_{H^{t'}_j} \{2\bar C(\alpha){\rm e}^{-2\pi b/h}\|\tau_i\|_{L_2(D_i)}\}\prod_{k\neq i,j} (1+d^{-1})\|\tau_k\|_{L_2(D_k)}\nonumber \\
 &&\;\le  \{ ({\rm e}\alpha)^{-1}
 +2\bar C(\alpha){\rm e}^{-bh}\} \{2\bar C(\alpha){\rm e}^{-2\pi b/h}\}\{1+d^{-1}\}^{d-2}\|\tau_j\|_{H^{t'}_j} \prod_{k\neq j} \|\tau_k\|_{L_2(D_k)},
\end{eqnarray}
where we used \eref{est1}  
 for the  term outside of the product, and we used
\eref{est2} with $\mu=0$  and the bound \eref{then1} for the remaining terms.  Similarly, for the term $j=i$ in \eref{Si},  we have
\begin{eqnarray}
\label{jeqi}
&&\|S_{i,i}\|_{H_i^s}\prod_{k\neq i}\|S_{i,k}\|_{L_2(D_k)}\nonumber\\
&&\;\le \{ (2\bar C(\alpha){\rm e}^{-2\pi b/h}\|\tau_i\|_{H_i^{t'}}\}
\prod_{k\neq i} (1+d^{-1})\|\tau_k\|_{L_2(D_k)}\nonumber\\
&&\;\le  \{2\bar C(\alpha){\rm e}^{-2\pi b/h}\}\{1+d^{-1}\}^{d-1}\|\tau_i\|_{H_i^{t'}}
\prod_{k\neq i} \|\tau_k\|_{L_2(D_k)}.
\end{eqnarray}

If $\alpha\le 1$, we use the \rd{bounds} \eref{jneqi} and \eref{jeqi} in the sum on the right-hand side of \eref{Si},
and we arrive at
\be
\label{Siest}
\|S_i\|_{\HH^s}\le (d^{\max\{0,s-1\}})^{1/2}  d \{ ({\rm e}\alpha)^{-1}
 +2\tilde C(\alpha){\rm e}^{-bh}\} \{2\tilde C(\alpha){\rm e}^{-bh}\}\{1+d^{-1}\}^{d-2}\|\tau\|_{\tau'},
\ee
where we used \eref{lbl} and the fact that $(1+1/d)\le 2e \{ ({\rm e}\alpha)^{-1}
 +2\tilde C(\alpha)\}$.  If we now sum over $i=1,\dots,d$, we arrive at \eref{fullBest} and complete the proof in this case.
A similar argument gives the case (ii) when $\alpha>1$.

{\bf Case $t'<0$:} The proof in this case is similar to that of {\bf Case 1}
 except that we use Corollary \ref{cor:duality} in place of \eref{lbl} and
 this forces us to use \eref{est2} for $\mu=t'$ rather than $\mu=0$.   Since we will not
use this case in what follows (see Remark \ref{rem:problem}), we do not include the details.\end{proof}

\begin{remark}
\label{rem:problem}
In view of the value of  $\bar C(\alpha)$ given in {\rm \eref{Calpha}}, the estimate
{\rm \eref{fullBest1}} shows, in particular, that
$$
 \|{\rm e}^{-\alpha\cB}\tau - E(\tau,\alpha,h)\|_{ s}\le Cd^2\big(d^{\max\{0,s-1\}})^{1/2}\alpha^{-1}{\rm e}^{-c_1\alpha}{\rm e}^{-2\pi b/h}\|\tau\|_{t'},\quad \mbox{when}\,\,\alpha\geq 1,
$$
 where $C$ depends only the embedding constant from Corollary \ref{cor:duality} and on $C_1$ from {\rm \eref{Calpha}}.
Therefore, {\rm \eref{fullBest}} is of primary importance for small $\alpha$. In this regime, when $t' < 0$, the right-hand side of {\rm \eref{fullBest}}
contains the factor $\big( (|t'|/2{\rm e}\alpha)^{|t'|/2} +(1/d)\big)^{d-2}$. Thus, whenever $(|t'|/2{\rm e}\alpha) >\big((d-1)/d\big)^{2/|t'|}$
this factor exhibits an exponential growth of the form  $\alpha^{-|t'|(d-2)/2}$.  In fact, as will be seen below, $\alpha$
can be as small as $\ve^2$ so that, in order  to compensate this growth,  the factor ${\rm e}^{-2\pi b/h}$ would have to satisfy at least
${\rm e}^{-2\pi b/h}\lsim \ve^{|t'|(d-2)}$. Thus, the accuracy {\rm \eref{ubar}} needed in the low-dimensional problems would scale at least  like
$\ve^{|t'|(d-2)}$, which is exponential in $\ve$. To avoid this, at least in our proof, one apparently has to require $t' \geq 0$. Later this means that
the excess regularity $\zeta$ should satisfy $t+\zeta \geq 0$, which becomes \rd{increasingly} stringent \rd{with decreasing} $t$.
 \end{remark}

\begin{corollary}
\label{cor:fromnowon}
Under the assumptions of Lemma \ref{lem:fullB} suppose that in addition
 \be
\label{classical}
t' \geq 0.
\ee
Then,   for $s-t' < 2$, for which {\rm \eref{baru}} holds, one has
\be
\label{fullBest-2}
\|{\rm e}^{-\alpha\cB}\tau - E(\tau,\alpha,h)\|_{s}  \leq  C_5 d^2 \big(d^{\max\{0,s-1\}})^{1/2}
\max\{\alpha^{-1},\alpha^{-2}\}{\rm e}^{-2\pi b/h} {\rm e}^{-c_1\alpha} \|\tau\|_{t'},
\ee
where $C_5$ depends only on the embedding constant from Corollary \ref{cor:duality} and on $\ulambda$ and $\Gamma$ through
the constants in {\rm \eref{Calpha}}.
\end{corollary}

 \subsection{The analysis and proof of Proposition \ref{prop:SRBepsilon}}
We proceed now to the analysis of the approximation of $u$ by
$$  
  \bar S_{R(\ve),h}(\cB)(g_{r(\ve)}) := \sum_{\ell=1}^{r(\ve)} \sum_{k=1}^{R(\ve)} \omega_{R(\ve),k} E(g^{(\ell)},\alpha_{R(\ve),k},h),
$$ 
as defined by \eref{uepsilon-0}. Recall that our goal is to show that when $h=h(\ve)$ is given by \eref{condh2}, then we have
$$
\big\|u(\ve)-S_{R(\ve),h}(\cB)(g_{r(\ve)})\big\|_{t+2}\le  \frac{\ve}2 \|f\|_{\bar \cA^\gamma(\HH^t,\HH^{t+\zeta})},
$$
where $u(\ve)= \bar S_{R(\ve)}(\cB)g_{r(\ve)}$, with the choice of $R(\ve)$ and $r(\ve)$ specified in \eref{Repsilon} and \eref{repsilon},
  is known to satisfy
$\| u -  u (\ve) \|_{t+2}\le     \frac{\ve}2 \|f\|_{\bar\cA^\gamma(\HH^t,\HH^{t+\zeta})}$.

 In going further, we recall from Lemma \ref{explemma1} that any summand that appears in $\bar S_{R(\ve)}(\cB)$ satisfies
 \be
\label{alphalower}
\alpha_{R(\ve),k} \geq T_{R(\ve)}^{-1} =  8 {\rm e}^{-\pi\sqrt{R(\ve)}} \ge  {8{\rm e}^{-\pi}}\Big(\frac{\bar C_2}{\ve}\Big)^{-\frac{2\pi}{\zeta \ua}},
\ee
 where we have used  the definition of $R(\ve)$ and $C_1(\zeta)$ in  \eref{Repsilon}.  This
means that
\be
\label{alphalower2}
\alpha_{R(\ve),k} \geq   c_3\ve^{e(\zeta)} =: \alpha(\ve),\quad e(\zeta):= \frac{2\pi}{\zeta \ua},
\ee
where $c_3$  depends only on $\bar C_0, \bar C$, and $A_1$, see Proposition \ref{prop:Sbar}, {\bf (A1)}, and \eref{gamma1}.\\

\noindent
{\bf Proof of Proposition \ref{prop:SRBepsilon}:}
 First note that  if $c_6$ is chosen sufficiently small,  then any $h\le h(\ve)$ will   comply with the threshold \eref{hcond1} required in Lemma \ref{lem:fullB}.
In fact, inserting the expression \eref{alphalower2} into \eref{hcond1}, gives
$$
h_0= 2\pi b \Big(\log\Big(\Big(\frac{2C_1}{c_3}\Big)\Big(\frac{ {d}^{1/e(\zeta)}}{\ve}\Big)^{e(\zeta)} \Big)\Big)^{-1}
\ge c_6\big(\log(d/\ve)\big)^{-1}
$$
provided $c_6$ is chosen sufficiently small (depending in part on $\zeta$),   because we know that $\ua<\pi$.
 We consider only such $c_6$ in what follows.  Therefore, Lemma \ref{lem:fullB}, respectively \eref{fullBest-2} in Corollary \ref{cor:fromnowon},
apply.  Bearing \eref{Calpha} in mind, this yields
 for $r=r(\ve)$, $R=R(\ve)$,
given by \eref{Repsilon}, \eref{repsilon},
\be
\label{error}
\|\bar u(\ve)- u(\ve)\|_{s} \leq d^2 \big(d^{\max\{0,s-1\}})^{1/2}C_5  {\rm e}^{-2\pi b/h}
 \Big(
\sum_{\ell =1}^r \sum_{k=1}^R \omega_{R,k} {\rm e}^{-\alpha_{R,k}c_1} \max\{\alpha_{R,k}^{-1},\alpha_{R,k}^{-2}\}
\Big) \tripnorm{g_{r}}_{r,t+\zeta}.
\ee
To identify further stipulations on $c_6$ that eventually guarantee \eref{errorh}, we infer next  from Lemma \ref{explemma},  \eref{expfa} and \eref{Calpha} that there exists a constant $C(\Gamma)$
depending only on $\Gamma$ and $\ulambda$ such that
$$
\sum_{k=1}^R \omega_{R,k}{\rm e}^{-\alpha_{R,k}c_1} \leq C(\Gamma).
$$
Hence, $\|\bar u(\ve)- u(\ve)\|_{s} \leq \frac{\ve}{2A_1}\|g_{r(\ve)}\|_{t+\zeta}$ follows provided we can show that
\be
\label{when1}
d^2 \big(d^{\max\{0,s-1\}})^{1/2} C_5C(\Gamma)\,   {\rm e}^{-2\pi b/h}
\le \frac{\ve T_{R(\ve)}^{-2}}{2A_1r(\ve)}.
\ee
Now, by the condition \eref{gamma2} on the minimum
growth of $\gamma$, there exists \rd{a} $\mu >0$ and a \rd{positive} constant $C_\mu$ such that $x^\mu \leq C_\mu \gamma(x)$ \rd{for $x \geq 1$},
which implies that $\gamma^{-1}(x)\leq (C_\mu  x)^{1/\mu}$ \rd{for} $x\geq 1$. By the definition \eref{repsilon} of
$r(\ve)$ we thus conclude that
$$
r(\ve) \leq \gamma^{-1}(\bar C 4 A_1/\ve)\leq \Big(\frac{4 C_\mu \bar C A_1}{\ve}\Big)^{1/\mu}.
$$
Placing this into the right-hand side of \eref{when1}, we are left to show that
\be
\label{when2}
d^2 \big(d^{\max\{0,s-1\}})^{1/2} C_5C(\Gamma)\,   {\rm e}^{-2\pi b/h}
\le \frac{\ve T_{R(\ve)}^{-2}}{ 2A_1} \Big(\frac{4 C_\mu\bar C A_1}{\ve}\Big)^{-1/\mu},
\ee
where $\mu$ is the constant in the minimum growth condition \eref{gamma2}. 

Now using the bound for $T^{-1}_{R(\ve)}$ given in \eref{alphalower}, we need only show \rd{that}
\be
\label{when3}
d^2 \big(d^{\max\{0,s-1\}})^{1/2} C_5C(\Gamma)\,   {\rm e}^{-2\pi b/h}
\le   {\rm e}^{-2\pi}\Big(\frac{\bar C_2}{\ve}\Big)^{-\frac{4\pi}{\zeta \ua}}  \frac{32\ve }{ A_1} \Big(\frac{4C_\mu\bar C A_1}{\ve}\Big)^{-1/\mu}.
\ee
Hence, it suffices to show that
\be
\label{when4}
{\rm e}^{-2\pi b/h}\le \hat C \Big (\frac{\ve}{d}\Big )^{\hat a},
\ee
where $\hat C$ and $\hat a$ are appropriate constants.  Taking a logarithm, we see that if $c_6$ is sufficiently small, then \eref{when4}
will be satisfied for $h\le c_6/\log (\frac {d}{\ve})$.
\hfill $\Box$\\

\subsection{Proof of Theorem \ref{thm:complex}, (i)}
\begin{theorem}
\label{thm:uepsilon}
Assume that {\bf (A1)} is valid and hence the conditions $(\gamma1), (\gamma2)$ hold and let $t+\zeta \geq 0$.
 Furthermore, assume that  for $k=1,\ldots,R(\ve)$, $\ell=1,\ldots,r(\ve)$, $j=1,\ldots,d$, $q=-N,\ldots,N$, $N=N(h(\ve))$,
 the approximate solutions $\bar u_{j,q}(\alpha_{R(\ve),k},g^{(\ell)}_j)$, entering
  the $E_j$ from {\rm \eref{Ej}} in the definition of $S_{R(\ve),h(\ve)}(\cB)(g_{r(\ve)})$, satisfy \eref{baru}
for the pairs
$$
s:= t+2 ,\quad t':= t+\zeta,\quad \mbox{and}\quad s=t'=0.
$$
 Then, the finitely parametrized finite rank function $u(\ve)$, given by {\rm \eref{uepsilon}},  
satisfies
\be
\label{accuracy}
\|u- \bar u(\ve)\|_{t+2}\le \ve  \|f\|_{\bar\cA^\gamma(\HH^{t},\HH^{t+\zeta})},
\ee
and is stable in the sense that  
\be
\label{stability}
  \tripnorm{\bar u(\ve)}_{r(\ve)R(\ve),t+2}
\leq C_7 \|f\|_{\bar\cA^\gamma(\HH^{t},\HH^{t+\zeta})},
\ee
where $C_7$ depends   on $C_0$ and $A_1$.
\end{theorem}
\begin{proof}
The estimate \eref{accuracy} follows directly from {\bf (A1)} and \eref{errorh} in Proposition \ref{prop:SRBepsilon} for $s=t+2$.
Concerning \eref{stability},
we need to estimate the terms $\omega_{R(\ve),k} \|E(g^{(\ell)},\alpha_{R(\ve),k},h(\ve))\|_{t+2}$.
In \eref{tripbound-0} we have already shown that
$$
\omega_{R(\ve),k}\|{\rm e}^{-\alpha_{R(\ve),k}\cB}g^{(\ell)}\|_{t+2+\zeta} \leq (1+ C_0)\|g^{(\ell)}\|_{t+\zeta}.
$$
Hence, by Proposition \ref{prop:SRBepsilon}, and in particular \eref{when1},
\begin{eqnarray*}
&&\omega_{R(\ve),k}\|E(g^{(\ell)},\alpha_{R(\ve),k},h(\ve))\|_{t+2} \le  (1+ C_0)\|g^{(\ell)}\|_{t+\zeta}\\[2mm]
&&\quad\qquad\qquad +
\omega_{R(\ve),k} d^{2+(t+1)/2}  
C_5 {\rm e}^{-c_1\alpha_{R(\ve),k}}
 \alpha_{R(\ve),k}^{-2}{\rm e}^{-2\pi b/h}
\|g^{(\ell)}\|_{t+\zeta}\\[2mm]
&&\qquad\qquad\le  (1+ C_0)\|g^{(\ell)}\|_{t+\zeta}
+C(\ulambda)d^{2+(t+1)/2}  
C_5 \alpha_{R(\ve),k}^{-2}{\rm e}^{-2\pi b/h}\|g^{(\ell)}\|_{t+\zeta}\\
&&\qquad\qquad\le  \Big(1+C_0 + \frac{\ve}{2A_1}\Big) \|g^{(\ell)}\|_{t+\zeta},
\end{eqnarray*}
which completes the proof.
\end{proof}

To prove now (i) of Theorem \ref{thm:complex}, we specialize Theorem \ref{thm:uepsilon} to the case $t=-1$, $1\le \zeta \le 2$.
By step (ii) of {\bf Scheme-Exp}, we know that the approximate resolvent solutions are $(1,-1+\zeta)$-accurate.
Thus, we only need to verify that they are automatically $(0,0)$-accurate as well.
To this end, we now require also the validity of assumptions {\bf (A2), (A3)}. In fact, \rd{we} set $\zeta := 1+ \zeta'$ for some $\zeta'\ge 0$,
and assume that the finite element solution $v_\delta\in V_{j,\delta}$
satisfies \eref{target-sol}.  By interpolation, we then obtain
$$
\|v-v_\delta\|_{H^1_j}\le C \delta^{\frac{1}{\zeta}}\|w\|_{L_2(D_j)}.
$$
Using the standard Aubin--Nitsche duality argument, this in turn, yields
$$
\|v-v_\delta\|_{L_2(D_j)}\le C \big(\delta^{\frac{1}{\zeta}}\big)^2 \|w\|_{L_2(D_j)}\le C \delta \|w\|_{L_2(D_j)},
$$
since $\zeta \le 2$. Hence, Theorem \ref{thm:uepsilon} is applicable and proves (i) of Theorem \ref{thm:complex}.

\subsection{The proof of Theorem \ref{thm:complex}, (ii) and (iii)}

The major part of the computational work in {\bf Scheme-Exp} is obviously \rd{associated with} the approximate solution of the resolvent equations \eref{resolvent},
which could be done completely in parallel. Let us denote by $\cost(q,k,\ell,j)$ the computational cost
of \eref{resolvent} for $\tau_j= g^{(\ell)}_j$ for $s= 1$, $t'= -1+\zeta$. Hence, the total cost of computing
$\bar u(\ve)$ is given by
\be
\label{cost}
\cost(\bar u(\ve)) = \sum_{q=-N}^N\sum_{\ell=1}^{r(\ve)}\sum_{k=1}^{R(\ve)}\sum_{j=1}^d\cost(q,k,\ell,j).
\ee

In order to continue with the analysis of the computational cost \eref{cost} of the {\bf Scheme-Exp} we require
the validity of assumptions {\bf (A1) -- (A3)}.

To estimate the cost, $\cost(q,k,\ell,j)$, of computing the approximate solution
$\bar u_{j,q}(\alpha_{R(\ve),k},g^{(\ell)}_j)$ requires a few further preparatory remarks.   First note that, by \eref{Gamma},
the right-hand side of the corresponding resolvent problem  \eref{resolvent} is given by
$$
\rhs(q,k,j,\ell) := -\frac{1}{2\pi i}\sinh(qh(\ve) +i\pi/6){\rm e}^{-\alpha_{R(\ve),k}\Gamma(qh(\ve))}g^{(\ell)}_j,
$$
so that
\be
\label{rhs1}
\|\rhs(q,k,j,\ell)\|_{H^{-1+\zeta}_j} \leq C_{10} {\rm e}^{|q h(\ve)|}{\rm e}^{-\alpha_{R(\ve),k}{\rm e}^{|qh(\ve)|/2}}\|g^{(\ell)}_j\|_{H^{-1+\zeta}_j},
\ee
where $C_{10}$ depends only on $\Gamma$ and $\ulambda$.
We record the following simple observation.

\begin{lemma}
\label{lem:crudebound}
For $\rhs(q,k,j,\ell)$ defined above and all $q=-N\ldots,N$,  $N=N(h(\ve))$, $j=1,\ldots,d$, $\ell=1,\ldots,R(\ve)$, $\ell=1,\ldots,r(\ve)$, we have
\be
\label{crudebound1}
\|\rhs(q,k,j,\ell)\|_{H^{-1+\zeta}_j} \leq C_{11}   
\alpha_{R(\ve),k}^{-1} \|g^{(\ell)}_j\|_{H^{-1+\zeta}_j},
\ee
where the constant $C_{11}$ depends only on $\Gamma$ and $\ulambda$ and where $\alpha(\ve)$ is
defined by \eref{alphalower2}.
\end{lemma}
\begin{proof}
Since the function $x{\rm e}^{-\alpha x/2}$ attains its maximum {on $[0,\infty)$} at $x^*= 2/\alpha$
{where it takes the} value $\frac{2}{\alpha}{\rm e}^{-1}$ the assertion is an immediate consequence of \eref{rhs1}.
 \end{proof}

Recall that by \eref{baru} the target accuracy depends also on $\alpha_{R(\ve),k}$ but becomes more stringent \rd{as}
$\alpha_{R(\ve),k}$ \rd{increases}. Combining Lemma \ref{lem:crudebound} with \eref{baru} shows that the target accuracy
$\delta(q,k,j,\ell)$, at which   the resolvent problem
$$
(\Gamma(g h(\ve))\mI -\cB_j)u_{j,q}(\alpha_{R(\ve),k},g^{(\ell)}_j)=-\frac{1}{2\pi  i}\sinh(qh(\ve)+i\pi/6)\,{\rm e}^{-\alpha_{R(\ve),k}\Gamma(qh(\ve))}
g^{(\ell)}_j
$$
has to be solved in order to satisfy the tolerance required in \eref{baru} for $s=1, t'=-1+\zeta$, should satisfy, in view of \eref{Calpha},
$$
C_{11} (\alpha_{R(\ve),k})^{-1} \delta(q,k,j,\ell) \leq \big(C_3|\log(h(\ve)\alpha_{R(\ve),k})|\big)^{-1} C_1\alpha_{R(\ve),k}^{-1} {\rm e}^{-2\pi b/h(\ve)}.
$$
This gives
\be
\label{target-delta}
\delta(q,k,j,\ell) \leq C_{12} \big|\log(h(\ve)\alpha_{R(\ve),k})\big|^{-1} {\rm e}^{-2\pi b/h(\ve)},
\ee
where $C_{12}$ depends only on $\ulambda$ and $\Gamma$.
We can now estimate the complexity of the approximate resolvent solutions.
\begin{lemma}
\label{lem:sol-complexity}
Under the hypotheses of Theorem \ref{thm:complex},  
for each $\ell=1,\ldots, r(\ve)$, $k=1,\ldots,R(\ve)$, $q=-N,\ldots,N$, $j=1,\ldots,d$,
the trial spaces $V_{q,k,\ell,j}$ required to approximate the solution to \eref{resolvent} with accuracy
specified in \eref{baru} have dimension at most
\be
\label{dim-enough}
{\rm dim}\, V_{q,k,\ell,j} \leq  
  C_{15} \Big(\frac{\ve}{d}\Big)^{-\bar\rho p/\zeta}.
\ee
Moreover, the computational work required to determine $\bar u_{j,q}(\alpha_{R(\ve),k},g^{(\ell)}_j)$
is bounded by
\be
\label{sol-complex}
\cost(q,k,\ell,j) \leq C_{16}  
\Big(\frac{\ve}{d}\Big)^{-\bar\rho p/\zeta},
 \ee
where  $\bar\rho$ is any fixed constant satisfying $\bar\rho > 2\pi b/c_6$, and the constants $C_{15},C_{16}$
depend on
$\bar\rho, C_8, C_9$, $c_6, c_3$.
\end{lemma}
\begin{proof}
We need to estimate the two factors on the right-hand side of \eref{target-delta}
from below.
To that end, recalling \eref{condh2} and \eref{repsilon},  
we have
\be
\label{e-first}
{\rm e}^{-2\pi b/h(\ve)} = (\ve/d)^{2\pi b/c_6} .
\ee
Moreover,
by \eref{condh2} and \eref{alphalower2} we obtain for $\alpha_{R(\ve),k}\leq 1$ (see \eref{alphalower2}) \rd{that}
\begin{eqnarray}
\label{e-second}
\big|\log(h(\ve)\alpha_{R(\ve),k})\big|  \leq  \big|\log\Big(\frac{c_6 c_3\ve^{e(\zeta)}}{\log( d/\ve)} \Big)\big|. 
\end{eqnarray}
Using that $\log (d/\ve) \leq C (d/\ve)^\rho$ holds for any $\rho >0$ with $C$ depending on $\rho$, we obtain
\be
\label{e-3}
\big|\log(h(\ve)\alpha_{R(\ve),k})\big| \le \log \Big(\frac{C_{13} d^\rho}{  \ve^{e(\zeta)+1}}\Big),
\ee
where $C_{13}$ depends on $\rho, c_3, c_6$. Combining \eref{e-first} and \eref{e-3} it suffices
to require that
\be
\label{dj-3}
\delta(q,k,j,\ell)\leq \Big(\frac{\ve}{d}\Big)^{2\pi b/c_6}\Big(\log \Big(\frac{C_{13} d^\rho}{  \ve^{e(\zeta)+1}}\Big)\Big)^{-1}.
\ee
Hence, choosing any fixed number $\bar\rho > 2\pi b/c_6$, there exists a constant $C_{14}$, depending on $\zeta,\rho, c_3, c_6$
such that
\be
\label{e-4}
\delta(q,k,j,\ell)\le C_{14} \Big(\frac{\ve}{d}\Big)^{\bar\rho}
\ee
guarantees the validity of \eref{baru}. The assertion of Lemma \ref{lem:sol-complexity} follows now
from \eref{Vjdelta} and \eref{flops}.
\end{proof}
 \medskip

We can now complete the proof of Theorem \ref{thm:complex}.
In fact,  in total $dr(\ve)R(\ve)(2N(h(\ve))+1)$ low-dimensional problems have to
be solved. Recall from \eref{hN} that
$$
(2N(h(\ve))+1)\leq C_3 h(\ve)^{-1}|\log (h(\ve)\alpha(\ve))|.
$$
By \eref{condh2} and \eref{e-3}, we conclude that
$$
(2N(h(\ve))+1)\leq C \Big(\log\Big(\frac{d}{\ve}\Big)\Big)^2,
$$
where $C$ depends on $C_3,c_3,c_6$, and $\rho$.
Now (ii) and (iii) are immediate consequences of
\eref{cost},  Lemma \ref{lem:sol-complexity},  and the bounds \eref{Repsilon2}, \eref{repsilon} for
$R(\ve)$, $r(\ve)$, respectively.\hfill $\Box$\\

The complexity bounds are still based on generous overestimations, since the constraints on the $\delta(q,k,j,\ell)$
are uniform and refer to the least favorable constellation of parameters. Also one need not use the same number of quadrature points for all $\alpha_{R(\ve),k}$.

\section{Concluding Remarks - Some Loss of Tensor-Sparsity}\label{sect:loss}
 
The prototypical example for the above setting is $\cB = - \Delta$.
A natural question is what could be said about tensor-sparsity when $-\Delta$ is replaced, for example, by \rd{the second-order elliptic differential operator $
w \mapsto -{\rm div}(A\nabla w) + w$},
where $A$ is a symmetric positive definite $(d\times d)$-matrix. Intuitively, since $A$ now combines different variables one expects that the solution to
\beqn
\label{diff1}
-{\rm div}(A\nabla w)  + w = g \quad \mbox{in}\,\, \R^d,\quad w\in H^1(\R^d),
\eeqn
 is ``less tensor-sparse''.
 In fact, since $A$ is symmetric positive definite, it can be factorized as $A=Q^TDQ$ where $Q$ is unitary. Defining $\uu (x) := w(Qx)=w(y)$, \eqref{diff1}
takes the form
\beqn
\label{diff2}
- \sum_{k=1}^d D_{k,k} \uu(x) + \uu(x) = \cB \uu(x) = f(x) := g(Qx),
\eeqn
where $\cB$ is again of the form \eref{tensor-sum}.
Hence, when $f$ in \eref{diff2} is tensor-sparse the results above apply to \eref{diff2} providing a tensor-sparse
solution in the new coordinate system. However, this does not imply, of course, any tensor-sparsity of $w$ as a function
of $y$ in the original coordinate system. Understanding the general circumstances in which
it is possible to quantify any tensor-sparsity of solutions to \eref{diff1} is the subject of ongoing work.

 \end{document}